\documentclass[12pt,leqno]{article}
\usepackage{amsmath, amssymb, comment, amsthm}
\usepackage{latexsym}
\usepackage{nextpage}
\usepackage{graphicx}
\usepackage{thmtools}
\usepackage{float}
\usepackage{array, multirow, hhline}

\usepackage{tocloft} 
\usepackage{verbatim} 
\usepackage{setspace} 
\usepackage{indentfirst} 
\usepackage[font=normal,labelfont=normal,labelsep=period,tableposition=top,justification=raggedright]{caption} 
\usepackage{flafter} 
\usepackage{mathrsfs} 
\usepackage{ragged2e} 
\usepackage[compact]{titlesec} 

\allowdisplaybreaks
\raggedright

\makeatletter
\renewcommand\section{\@startsection {section}{1}{\z@}%
{-3.5ex \@plus -1ex \@minus -.2ex}%
{2.3ex \@plus.2ex}%
{\normalfont\centering}}
\makeatother

\makeatletter
\renewcommand\subsection{\@startsection {subsection}{1}{\z@}%
{-3.5ex \@plus -1ex \@minus -.2ex}%
{0.8pt plus 2pt minus 2pt}%
{\normalfont}}
\makeatother

\makeatletter
\renewcommand\subsubsection{\@startsection {subsubsection}{1}{\z@}%
{-3.5ex \@plus -1ex \@minus -.2ex}%
{0.8pt plus 2pt minus 2pt}%
{\normalfont}}
\makeatother

\usepackage{fancyhdr}
\fancyheadoffset[HR]{0.02in}

\usepackage{anysize}
\marginsize{1.7in}{1in}{1.2in}{1.6in}

\usepackage[all]{nowidow} 
\usepackage{array}
\usepackage[usenames,dvipsnames]{color}
\usepackage{tikz}

\numberwithin{equation}{section}
\numberwithin{figure}{section}
\theoremstyle{plain}
\newtheorem{thm}{\protect\theoremname}
\theoremstyle{plain}
\newtheorem{conjecture}[thm]{\protect\conjecturename}
\theoremstyle{plain}
\newtheorem{lem}[thm]{\protect\lemmaname}
\theoremstyle{plain}
\newtheorem{prop}[thm]{\protect\propositionname}

\newtheorem{definition}[thm]{\protect\definitionname}
\newtheorem{corollary}[thm]{\protect\corollaryname}
\newtheorem{example}[thm]{\protect\examplename}
\usepackage{amstext}
\usepackage{todonotes}
\usepackage{graphicx}

\makeatother

\usepackage{babel}
\providecommand{\conjecturename}{Conjecture}
\providecommand{\lemmaname}{Lemma}
\providecommand{\theoremname}{Theorem}
\providecommand{\propositionname}{Proposition}
\providecommand{\exname}{Exercise}
\providecommand{\problemname}{Problem}
\providecommand{\definitionname}{Definition}
\providecommand{\corollaryname}{Corollary}
\providecommand{\examplename}{Example}

\renewcommand{\liminf}{ \underline{\lim} \,}
\renewcommand{\limsup}{\overline{\lim} \,}

\renewcommand{\Re}[1]{\mbox{\rm Re}(#1)}

\global\long\def\Re{\operatorname{Re}}

\begin{document}

\begin{center}
ABSTRACT
\end{center}

\begin{center}
Zero Distribution of Generated Taylor Polynomials
\end{center}

\doublespacing
Our  goal in this paper is to study the zero distribution of a sequence of polynomials whose coefficients satisfy a three-term recurrence. Equivalently, these polynomials are Taylor polynomials of a rational function with a polynomial denominator of degree 2. We will use complex analysis to find the bi-variate generating function for that sequence of Taylor polynomials. With this generating function, we prove that the zeros of these Taylor polynomials lie on one side of the closed disk centered at the origin. The radius of the disk is exactly the reciprocal of the modulus of the smaller zero of the degree two denominator.  
Finally, we show that the zeros of these Taylor polynomials approach the boundary of the disk above. 

\singlespacing
\vskip 0.3in
\noindent Juhoon Chung

\noindent May 2022 
\thispagestyle{empty}
\newpage

\mbox{}
\thispagestyle{empty}
\newpage

\begin{center}
Zero Distribution of Generated Taylor Polynomials
\end{center}

\vskip 2in

\begin{center}
by
\end{center}

\begin{center}
Juhoon Chung
\end{center}

\vskip 2in

\begin{center}
A project \\[6pt]
submitted in partial \\[6pt]
fulfillment of the requirements for the degree of \\[6pt]
Master of Science in Mathematics \\[6pt]
in the College of Science and Mathematics \\[6pt]
California State University, Fresno
\end{center}

\begin{center}
May 2022 
\end{center}
\setcounter{page}{1}
\pagenumbering{roman}
\thispagestyle{empty}
\newpage

\begin{center}
APPROVED
\end{center}

\begin{center}
For the Department of Mathematics:
\end{center}


\noindent
\begin{quote}
We, the undersigned, certify that the project of the following student meets the required standards of scholarship, format, and style of the university and the student's graduate degree program for the awarding of the master's degree.
\end{quote}

\vskip 0.5in

\begin{center}
Juhoon Chung  
\vskip -0.1in
\makebox[\textwidth]{\hspace{4pc}\hrulefill\hspace{4pc}}
Project Author
\end{center}

\vskip 0.5in

\begin{center}
 \makebox[\textwidth]{\hspace{2pc}\hrulefill\hspace{2pc}}
 \end{center}
 \vskip -0.2in
 \hspace{2pc}
Dr. Khang Tran \hfill Mathematics  \hspace{2pc}


\vskip 0.3in
\begin{center}
 \makebox[\textwidth]{\hspace{2pc}\hrulefill\hspace{2pc}}
 \end{center}
 \vskip -0.2in
 \hspace{2pc}
Dr. Stefaan Delcroix \hfill Mathematics  \hspace{2pc}

\vskip 0.3in
\begin{center}
\makebox[\textwidth]{\hspace{2pc}\hrulefill\hspace{2pc}}
\end{center}
\vskip -0.2in
\hspace{2pc}
Dr. Oscar Vega \hfill Mathematics \hspace{2pc}   

\thispagestyle{empty}
\newpage

\begin{center}
ACKNOWLEDGMENTS
\end{center}
\doublespacing

I do not know how many times my parents were disappointed by me. 
What I can tell is that they always try to encourage me and cheer me.
Throughout working with Dr. Tran, I did not catch up his expectation. He still keeps his respect upon me, and helps me. I appreciate his patience.
I would like to thank Dr.Tran for his intuitive conjecture and skills for this project.

I have felt the passion for their work from every professors I have been taught by ever since I entered Fresno State.
I felt their diligent behavior to their works make me a little humble and push myself to make effort.
\newline\newline\newline

\singlespacing
\thispagestyle{empty}
\newpage

\addtocontents{toc}{\hfill Page\par}

\begin{center}
\tableofcontents
\end{center}
\thispagestyle{empty}
\newpage



\addtocontents{lot}{\hfill Page\par}


\addtocontents{lof}{\hfill Page\par}

\addcontentsline{toc}{section}{LIST OF FIGURES} 
\begin{center}
\listoffigures
\end{center}
\thispagestyle{empty}
\newpage

\doublespacing

\begin{center}
\section{INTRODUCTION}
\end{center}
\setcounter{page}{1}
\pagenumbering{arabic}
\thispagestyle{empty} 


Many famous sequences of numbers in mathematics, such as the Fibonacci numbers, satisfy recurrent relations. An equivalent way to represent those relations is to use ordinary generating functions. For example, it is known that the sequence of Fibonacci numbers has the generating function (see \cite{Paul})
\[
\frac{t}{1-t-t^2}.
\]
 For a rigorous definition of generating function, we refer the reader to the section `Polynomials Generated by $\frac{1}{(at^2 + bt + c)(1-tz)}$'.
Similarly, we can define a recursive sequence of functions from a bi-variate generating function. For example, the famous sequence of Chebyshev polynomials, $U_m(z)$, satisfying the recurrence
\begin{align*}
    U_0(z) = &1\\
    U_1(z) = &2z\\
    U_{n+1}(z) = &2zU_n(z) - U_{n-1}(z), n\geq 1,
\end{align*}
has the generating function (see \cite{Andrews})
\[
\sum_{m=0}^\infty U_m(z) t^m = \frac{1}{1-2zt+t^2}.
\]
A question may occur: ``Is there a pattern in the zeros of such a sequence of functions?"
Many studies have taken place to answer this question.

In 2017, Forg\'{a}cs and Tran in \cite{FT1} found sufficient conditions of $P(t)$ and $r$ for $ \{H_m(z)\}_{m\geq 0}$, for large $m$, generated by 
\[
\frac{1}{P(t) + zt^r},
\]
to have zeros inside an interval $(a,b)$.
They found explicit $a$ and $b$ such that, if we denote $\mathcal{Z}(H_m)$ the set of all the zeros of $H_m(z)$, $\bigcup_{m\in\mathbb{M}}\mathcal{Z}(H_m)$ is dense in $(a,b)$.

In 2020, Forg\'{a}cs and Tran in \cite{FT2}, for $P(t)$ and $Q(t)$ polynomials with positive real zeros, found sufficient condition of $r$ and polynomials $P(t), Q(t)$ for $\{H_m(z)\}_{m\geq 0}$, generated by 
\[
\frac{1}{P(t) + zt^rQ(t)},
\]
to have real zeros of the same sign, for all $m\in\mathbb{N}$.
They showed that $\bigcup_{m\geq 0}\mathcal{Z}(H_m)$ is dense in an explicit real ray.

In 2018, Tran and Zumba in \cite{AT2} found that the zeros of polynomials $\{H_m(z)\}_{m\geq 0}$ generated by
\[
\frac{1}{1 + ct + bt^2 + zt^3}
\]
are real, for all $m\in\mathbb{N}$, if and only if one of the following conditions holds: 

\begin{enumerate}
\item $c= 0$  and  $b\geq 0$\\
\item $c\neq 0$ and $-1\leq \frac{b}{c^2} \leq \frac{1}{3}$.
\end{enumerate}

If the first condition 1 holds, the union of all the zeros $\bigcup_{m\geq 0}\mathcal{Z}(H_m)$ is dense on $(-\infty,\infty)$.
If the second condition 2 holds, $\bigcup_{m\geq 0}\mathcal{Z}(H_m)$ is dense on an explicit interval depending on $b$ and $c$.

In 2020, Tran and Zumba found necessary and sufficient conditions for the reality of the zeros of the sequence of polynomials generated by 
\[
\frac{1}{1 + at + (b + cz)t^2 + (d + ez)t^3},
\]
where $a,b,c,d,e\geq 0$ and $ae\neq0$ (see \cite{AT1}). In the same paper, the authors also found an optimal interval containing the zeros of the generated sequence. Here, the optimality means that the union of all the zeros of the generated sequence forms a dense subset of the given interval.

In this paper, we will study some results similar to those of the above papers. 
Then we will make some new contribution to study the zero distribution of polynomials. 
Our contribution would be related to Taylor polynomials, which are the partial sums of the Taylor series. For example, from the famous Taylor series for $e^x$
\[
e^x = \sum_{k=0}^{\infty}\frac{x^k}{k!},
\]
we can define the sequence of Taylor polynomials $\{P_n(x)\}_{n\geq 0}$ as 
\[
P_n(x) = \sum_{k=0}^{n}\frac{x^k}{k!}.
\]
\begin{figure}[htp]\label{f0.0}
    \centering
    \includegraphics[width=4cm]{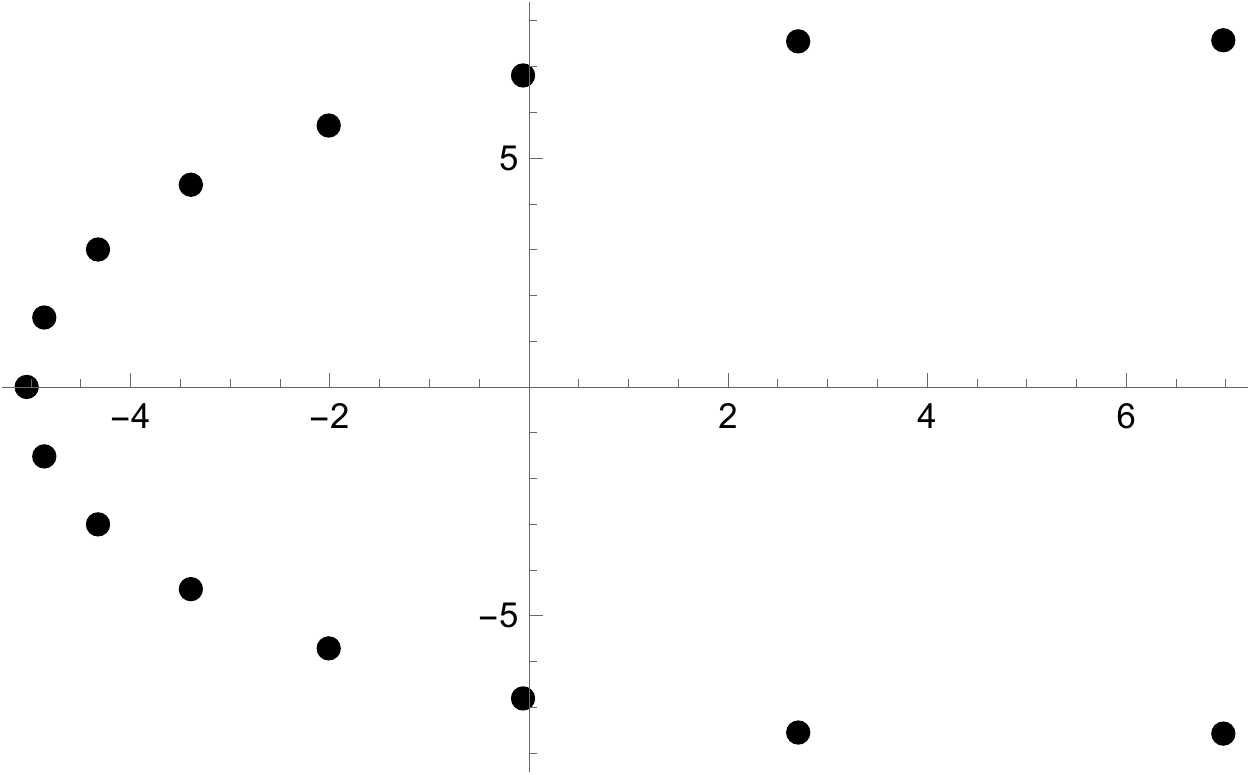}
    \includegraphics[width=4cm]{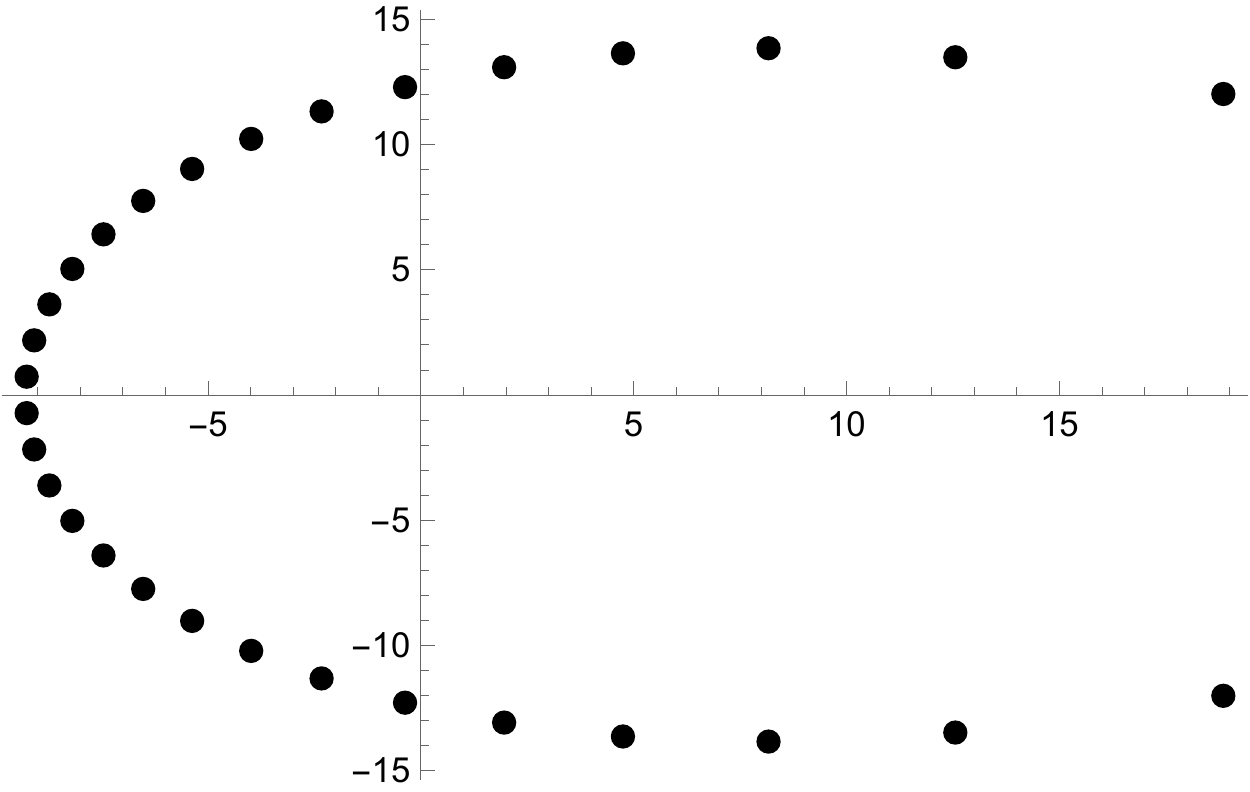}
    \includegraphics[width=4cm]{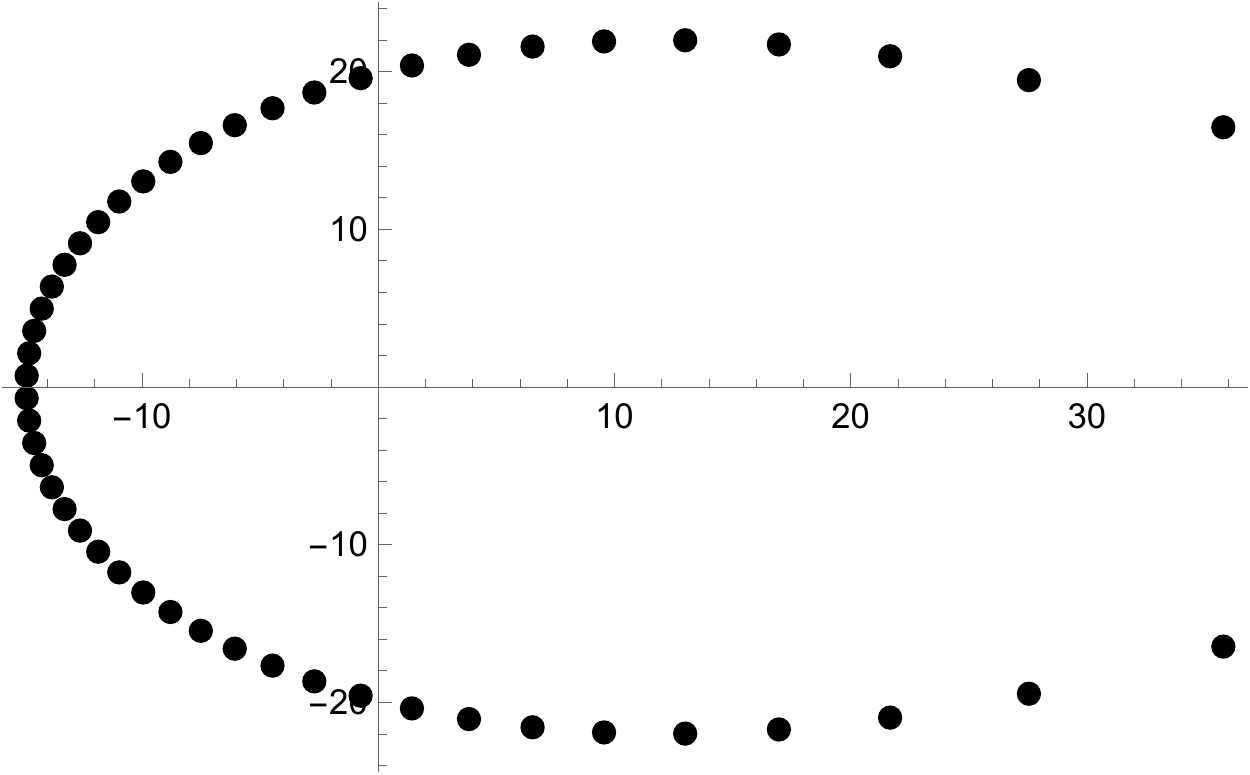}
    \caption{ Zeros of Taylor polynomials of $e^z$ for $N = 15, 30, 50$}
    \label{czego}
\end{figure}
In 1924, Szeg\"{o} studied the zero distribution of the Taylor polynomials of $e^z$ (see \cite{Szego}).
Figure \ref{czego} shows those zeros are about to form a curve.

In this paper, we use generating functions to study Taylor polynomials.
For example,
\[
\frac{e^t}{1 - zt}
\]
is the generating function of the sequence of the reciprocal of the Taylor Polynomials of $e^z$.
The reader may found reasoning for this generating function in Section `Complex Analysis'.

Our main goal is to study the zero distribution of a Taylor polynomial generated by 
\[
\frac{1}{at^2+bt+c}.
\]
In particular, we prove the following results.
\begin{thm}
Let $a,b,c\in\mathbb{R}$, 
where $ac<0$ and $b\neq 0$.
Suppose  $\{P_m(z)\}_{m\geq 0}$ is a sequence of functions of $z$ generated  as follows:
\[
\sum_{m=0}^\infty P_m(z) t^m =\frac{1}{(at^2+bt+c)(1-tz)}.
\]
Then, for all $m\geq 0$, all the zeros of $P_m(z)$ lie on the closed disk centered at the origin with radius $1/|\alpha|$ where $\alpha$ is the smallest (in modulus) zero of $at^2+bt+c$.
\end{thm}
\begin{thm}
Let $a,b,c\in\mathbb{R}$ where $ac>0$ and $b^2 - 4ac > 0$.
Suppose  $\{P_m(z)\}_{m\geq 0}$ is a sequence of functions of $z$ generated as follows:
\[
\sum_{m=0}^\infty P_m(z) t^m =\frac{1}{(at^2+bt+c)(1-tz)}.
\]
Then, for all $m\geq 0$, no zero of $P_m(z)$ lie inside the closed disk centered at the origin with radius $1/|\alpha|$ where $\alpha$ is the smallest (in modulus) zero of $at^2+bt+c$.
\end{thm}

Finally, we prove the following theorem about limiting behavior of the zero distribution.
\begin{thm}
Let $a,b,c\in\mathbb{R}\backslash\{0\}$ where $b^2 - 4ac > 0$.
Suppose  $\{P_m(z)\}_{m\geq 0}$ is a sequence of functions of $z$ generated as follows:
\[
\sum_{m=0}^\infty P_m(z) t^m =\frac{1}{(at^2+bt+c)(1-tz)}.
\]
Then 
\[\lim \mathcal{Z}(P_m) = \left\{z\in\mathbb{C}: |z| = \frac{1}{|\alpha|}\right\}
\]
where $\alpha$ is the smallest (in modulus) zero of $at^2+bt+c$.
\end{thm}
For the formal definition of $\lim \mathcal{Z}(P_m)$, we refer the reader to Section `Application of Limit Distribution of Zeros'.

The paper is organized as follows. 
\begin{enumerate}
\item Review concepts in Complex Analysis as a basis for the following sections;\\
\item Study previous theorems relative to the zero distribution of the generated sequence of functions;\\
\item Applying skills learned above and with our methods to prove two main results.
\end{enumerate}
The two main results are the followings
\begin{itemize}
    \item the zero distribution of Taylor Polynomials of rational polynomials on one side of the circle (c.f. Theorem 1, Theorem 2),
    \item the limiting behavior of the zeros of Taylor polynomials to the circle (c.f. Theorem 3).
\end{itemize}
We will prove those two main results in the last two sections, respectively.
\clearpage

\section{COMPLEX ANALYSIS}

In most of theorems from this paper, we will use generating functions to study the zeros of a generated sequence of functions. We will develop rigorous definitions of a bi-variate generating function. In order to achieve this goal, we will review basic complex analysis. We start with the definition of the radius of convergence of a power series. 

\begin{thm}[Proposition 1.3, \cite{John}]\label{t5.1}
For a given power series $\sum_{n=0}^{\infty}a_n(z-a)^n$, define the extended real number $R$, $0\leq R\leq\infty$, by
\[
\frac{1}{R} = \lim\sup |a_n|^{1/n}.
\]
Then
\begin{enumerate}
\item if $|z-a|<R$, the series converges absolutely;\\
\item if $|z-a| >R$, the terms of the series become unbounded and so the series diverges;\\
\item if $0<r<R$ then the series converges uniformly on $\{z;|z|\leq r\}$.
\end{enumerate}
Moreover, $R$ is unique having properties 1 and 1.

The extended real number $R$ is called the radius of convergence of the power series.
\end{thm}

Many of the following theorems are relative to `analytic functions', defined as follows.
\begin{definition}[Definition 2.3, \cite{John}]
Let $G$ be an open connected set. A function $f:G\rightarrow\mathbb{C}$ is analytic if $f$ is continuously differentiable.
\end{definition}

This paper studies theorems related to sequence of functions. 
In several theorems, an explicit form of sequence of functions is unknown.
So we must derive it from the given generating function.
We will find the region where the given generating function is analytic.
Then with the following theorem we can transform the generating function to an infinite series. 
\begin{thm}[Theorem 2.8, \cite{John}]\label{t5.0}
Let $f$ be analytic on $B(R;a)$; then $f(z) = {\displaystyle\sum_{n=0}^{\infty}a_n(z-a)^n}$ for $|z-a|<R$ where $a_n = \frac{1}{n!}f^{(n)}(a)$ and this series has radius of convergence greater than or equal to $R$.
\end{thm}

If we consider bi-variate function and build infinite series about one variable, the coefficient would depend on the other variable.
So we may consider it as a function about the other variable.
So we shall define the generated sequence of functions as follows.
\begin{definition}
For $z\in\mathbb{C}$, assume that $f(t,z)$ is analytic in $t$ in a neighborhood around the origin. If $P_m(z)$ is the $t^m$-coefficient of $f(t,z)$ expanded as a power series in $t$, then we say the sequence $\{P_m(z)\}$ is generated by $f(t,z)$.
\end{definition}
Although Theorem \ref{t5.0} provide us a way to obtain the $t^m$-coefficient as follows
\[
P_m(z) = \frac{1}{m!}\frac{\partial f^m}{\partial^m t}(0,z),
\] 
we will not use this method to obtain $\{P_m(z)\}_{m\geq 0}$.
In all the three main theorems, we will work with the rational polynomial \[
f(z,t) = \frac{1}{(at^2+bt+c)(1-tz)}.
\]
The partial fractions will transform the generating function $f(z,t)$ into the form 
\[
\frac{1}{1 - pt} + \frac{1}{1 - qt} + \frac{1}{1 - zt}.
\]
for some $p,q\in\mathbb{R}$.
Then we expand each fraction as follows
\[
\sum_{k=0}^{\infty}(at)^k + \sum_{k=0}^{\infty}(bt)^k + \sum_{k=0}^{\infty}(zt)^k.
\]
Then we will combine those three series into one series.
Before combine series, we need to verify if they are absolutely convergent by checking their radius of convergence.
The following theorem allow us to combine these series.
\begin{prop}[Proposition 1.6, \cite{John}]\label{p5.0}
Let $\sum_{n=0}^{\infty}a_n(z-a)^n$ and $\sum_{n=0}^{\infty}b_n(z-a)^n$ be power series with radius of convergence greater than or equal to $r>0$. Put
\[
c_n = \sum_{k=0}^na_kb_{n-k};
\]
Then both power series $\sum_{n=0}^{\infty}(a_n+b_n)(z-a)^n$ and $\sum_{n=0}^{\infty}c_n(z-a)^n$ have radius of convergence greater than or equal to $r$, and 
\begin{align}
\sum_{n=0}^{\infty}(a_n+b_n)(z-a)^n &= \left[\sum_{n=0}^{\infty}a_n(z-a)^n + \sum_{n=0}^{\infty}b_n(z-a)^n\right]\label{pe5.1}\\
\sum_{n=0}^{\infty}c_n(z-a)^n &= \left[\sum_{n=0}^{\infty}a_n(z-a)^n\right]\left[\sum_{n=0}^{\infty}b_n(z-a)^n\right]\nonumber
\end{align}
for $|z-a| < r$.
\end{prop}
So we can combine series into one series by grouping $t^m$-coefficients as Equation \eqref{pe5.1} shows.

\begin{prop}[Proposition 2.5, \cite{John}]\label{p5.1}
Let $f(z) = {\displaystyle\sum_{n=0}^{\infty}a_n(z-a)^n}$ have radius of convergence $R>0$.
Then
\begin{enumerate}
\item 

For each $k\geq 1$ the series 
\begin{equation}\label{pe5.0}
\sum_{n=0}^{\infty}n(n-1)\cdots(n-k+1)a_n(z-a)^{n-k}
\end{equation}
has radius of convergence $R$;

\item The function $f$ is infinitely differentiable on $B(R;a)$ and, furthermore, $f^{(k)}(z)$ is given by the series \eqref{pe5.0}, for all $k\geq 1$ and $|z-a|<R$;

\item For $n\geq 0$,
\[
a_n = \frac{1}{n!}f^{(n)}(a).
\]
\end{enumerate}
\end{prop}

In this paper, we study relations between generated sequence of functions and recurrence relation, and the consequence 3 of Proposition \ref{p5.1} confirm the uniqueness of coefficient.
For example, if we are given a relation
\[
1 = \sum_{k=0}^{\infty}[f_k(z)) + zf_{k+1}(z)]t^k
\]
we consider both sides as a series.
For the left hand side we consider it as a series with 0 coefficient for all $t^k$, $k = 1,2,\dots$, except for $k = 0$.
By uniqueness of coefficient, if we compare both sides' coefficient, we obtain as follows
\begin{align*}
    f_0(z) + zf_{1}(z) = &1\\
    f_1(z) + zf_{2}(z) = &0
\end{align*}
\begin{align*}
    f_2(z) + zf_{3}(z) = &0\\
    \vdots
\end{align*}
which yields recurrence relation
\[
f_m(z) + zf_{m+1}(z) = 0
\]
Our Proposition \ref{p5.1} allows us to prove the following proposition.
The following statement of Proposition \ref{p5.2} is a little different from that of Corollary 2.9 in \cite{John}.
We can deduce that the following Proposition \ref{p5.2} and Corollary 2.9 in \cite{John} are equivalent by Theorem \ref{t5.1}.
\begin{prop}\label{p5.2}(Corollary 2.9 \cite{John})
Suppose that $\sum_{n=0}^{\infty}a_nz^n$ is absolutely convergent on $B(R,0)$.
Then $\sum_{n=0}^{\infty}a_nz^n$ is analytic on $B(R,0)$.
\end{prop}
\begin{proof}
Let $r$ be the radius of convergence of $\sum_{n=0}^{\infty}a_nz^n$.
First we will prove that $r\geq R$.
Suppose, by contrary, that $r<R$. 
Notice that $B(R,0)\backslash B(r,0)$ is not empty. 
If we choose $z_0\in B(R,0)\backslash B(r,0)$ arbitrarily, we are given in the hypothesis that $\sum_{n=0}^{\infty}a_nz_0^n$ converges absolutely.
Meanwhile,  by Theorem \ref{t5.1}, $\sum_{n=0}^{\infty}a_nz_0^n$ diverges, a contradiction.
Thus $r\geq R$.

Moreover, Proposition \ref{p5.1} induces that $\sum_{n=0}^{\infty}a_nz^n$ is continuously differentiable on $B(r,0)$.
That is, analytic on $B(r,0)$.
Consequently, $\sum_{n=0}^{\infty}a_nz^n$ is analytic on $B(R,0)$.
\end{proof}

If we combine a couple of series into one seires, by Proposition \ref{p5.0}, the combined series is still absolutely convergent.
Then above Proposition \ref{p5.2} comfirms that the combined series is analytic.

All the main theorems in this paper are about the reciprocal of a Taylor polynomial.
We shall define what a Taylor polynomial is, and what is its reciprocal polynomial.
\begin{definition}[p118, \cite{William}]
We shall call 
\[
P_n(z) = \sum_{k=0}^n\frac{f^{(k)}(a)}{k!}(z-a)^k
\]
the Taylor Polynomial of order $n$ generated by $f$ centered at $a$.
\end{definition}

\begin{definition}
Let $f:\mathbb{C}\rightarrow\mathbb{C}$ be a polynomial
\[
f(t) = a_0 + a_1t + \dots + a_nt^n
\]
with $a_n\neq 0$.
Then 
\[
g(t) = a_n + a_{n-1}t + \dots + a_1t^{n-1} + a_0t^n = t^nf\left(\frac{1}{t}\right)
\]
is the reciprocal polynomial of $f$.
\end{definition}
In the first page of this paper, we have seen our generating function
\[
\frac{1}{(at^2+bt+c)(1-tz)}.
\]
Next, we will obtain the relation between a generating function and the corresponding reciprocal Taylor polynomial.
To study the relation, we need to reorder the addends of series so we would review the following theorem.
\begin{thm}[Theorem 6.27, \cite{William}]\label{t5.3}
If $\sum_{k=1}^{\infty}a_k$ converges absolutely, and $\sum_{k=1}^{\infty}b_k$ is any rearrangement of $\sum_{k=1}^{\infty}a_k$, then $\sum_{k=1}^{\infty}b_k$ converges and 
\[
\sum_{k=1}^{\infty}a_k = \sum_{k=1}^{\infty}b_k.
\]
\end{thm}
We shall read the theorem of relation between the generating function and the reciprocal Taylor polynomial.
\begin{thm}
The function $\frac{f(t)}{1-tz}$ is the generating function of the reciprocal polynomial of Taylor Polynomials generated by $f(t)$ centered at $0$.
\end{thm}
\begin{proof}
Let
\[
f(z) = \sum_{m=0}^{\infty}a_mz^m
\]
be the Taylor Series of $f$.
Define
\[
T_n(z) := \sum_{m=0}^{n}a_mz^m
\]
be the Taylor Polynomial of $f$ centered at $0$.
Let $z\neq 0$ be arbitrary.
Suppose that $r$ is the radius of convergence of $\sum_{m=0}^{\infty}T_m(z)t^m$.
Observe that, for arbitrary $t$ such that $|t|<\max(r,1)$, the series is absolutely convergent and thus, by Theorem \ref{t5.3}, we may rearrange the order of the series as follows:
\begin{align*}
    &\sum_{m=0}^{\infty}T_m(z)t^m\\
 = &a_0 + (a_0 + a_1z)t + (a_0 + a_1z + a_2z^2)t^2 \cdots\\
    = &a_0(1 + t + t^2 + \cdots) + a_1z(t + t^2 + t^3 + \cdots) + a_2z^2(t^2 + t^3 + \dots) + \dots
\end{align*}
Since $|t| < \max(r,1) \leq 1$, in view of geometric series, it holds that
\begin{align*}
    1 + t + t^2 + \cdots = &\frac{1}{1-t}.
\end{align*}
This implies that 
\begin{align*}
    &a_0(1 + t + t^2 + \cdots) + a_1z(t + t^2 + t^3 + \cdots) + a_2z(t^2 + t^3 + \cdots) + \cdots\\
    = &\frac{a_0}{1-t} + \frac{a_1zt}{1-t} + \frac{a_2z^2t^2}{1-t}+\cdots\\
    = &\frac{1}{1-t}(a_0 + a_1zt + a_2z^2t^2 + \cdots)\\
    = &\frac{1}{1-t}f(zt).
\end{align*}
So we obtain
\[
\sum_{m=0}^{\infty}T_m(z)t^m = \frac{f(zt)}{1-t}.
\]
Let us choose $t_0$ such that $|t_0/z|<r$.
If we substitute $t_0/z$ for $t$, we obtain
\[
\frac{f(t_0)}{1-t_0/z} = \sum_{m=0}^{\infty}T_m(z)\left(\frac{t_0}{z}\right)^m.
\]
Let $z_0 = 1/z$.
Then it yields that 
\[
\frac{f(t_0)}{1-t_0z_0} = \sum_{m=0}^{\infty}T_m(1/z_0)z_0^mt_0^m.
\]
Observe that
\begin{align*}
    T_m(1/z_0)z_0^m = &\left(a_0 + a_1\frac{1}{z_0} + a_2\frac{1}{z_0^2} + \cdots + a_m\frac{1}{z_0^m}\right)z_0^m
\end{align*}
\begin{align*}
    = &a_0z_0^m + a_1z_0^{m-1} + \cdots + a_{m-1}z_0 + a_m
\end{align*}
which is the reciprocal polynomial of $T_m(z_0)$.
Thus 
\[
\frac{f(t_0)}{1-t_0z_0} 
\]
is the generating function of the reciprocal polynomial of $T_m(z)$.
\end{proof}

When we work on our main theorems, we count number of zeros inside a circle and compare this number with generated polynomial.
The following theorem becomes the tool for it.
\begin{thm}\label{rouche}[Rouch\'{e}'s Theorem](Theorem 3, p177,  \cite{Stephen})
Suppose $f$ and $g$ are analytic on an open set containing a piece-wise smooth simple closed curve $\gamma$ and its interior.
If 
\[|f(z)  + g(z)| < |f(z)| \text{, for all } z\in\gamma,\]\
then $f$ and $g$ have the same number of zeros inside $\gamma$ counting multiplicities.
\end{thm}
We will study some examples that follow from this theorem. These examples are from \cite{Stephen}.
\begin{example}
We want to show that all the zeros of 
\[p(z) = 3z^3 - 2z^2 + 2iz - 8\]
lie in the annulus $1 < |z| < 2$.
\end{example}

Let 
\begin{align*}
f(z) &= -3z^3 + 8\\
g(z) &= 3z^3 - 2z^2 + 2iz - 8.
\end{align*}
Observe that 
\begin{equation}\label{e2.1}
|f(z) + g(z)| = |-2z^2 + 2iz| \leq 2|z|^2 + 2|z|
\end{equation}
and that 
\begin{equation}\label{e2.2}
|f(z)| = |-3z^3 + 8| \geq 8 - 3|z|^3.
\end{equation}
On the circle $|z| = 1$, in view of Inequality \eqref{e2.1},
\begin{align*}
|f(z) + g(z)|&\leq 2|z|^2 + 2|z| = 4<5= 8 - 3|z|^3\leq |f(z)|
\end{align*}
Meanwhile, the zeros of $f(z) = -3z^3 + 8$ are 
\[
z = (8/3)^{1/3}e^{i2\pi/3}, (8/3)^{1/3}e^{i4\pi/3}, (8/3)^{1/3},
\]
and thus there are no zeros are inside the circle $|z| = 1$.
Then, given that $f,g$ are entire, Rouch\'{e}'s Theorem implies that $3z^3 - 2z^2 + 2iz - 8 = 0$ has no zeros inside the unit circle.\\

For $|z| = 2$, we may deduce similarly as follows
\begin{align*}
|f(z) + g(z)|&\leq 2|z|^2 + 2|z|\\
&= 8 + 4\\
&< 16\\
&=  3|z|^3 - 8 \\
&\leq |-3z^3 + 8|\\
&= |f(z)|.
\end{align*}
Notice that all the zeros 
\[
z = (8/3)^{1/3}e^{i2\pi/3}, (8/3)^{1/3}e^{i4\pi/3}, (8/3)^{1/3}
\]
of $f(z)$ are inside the circle $|z| = 2$.
Thus again, by Rouch\'{e}'s Theorem, $3z^3 - 2z^2 + 2iz - 8 = 0$ has 3 zeros inside the circle $|z| = 2$.
That is, all the zeros are inside the circle $|z| = 2$.
Observe that, for any $z\in\mathbb{C}$ such that $|z| = 1$, 
\begin{align*}
|g(z)| &= |3z^3 - 2z^2 + 2iz - 8|\\
&\geq 8 - |3z^3 - 2z^2 + 2iz|\\
&\geq 8 - 3|z|^3 - 2|z|^2 - 2|z|\\
&= 1.
\end{align*}
This implies that there is no zero on the curve $|z| = 1$.
All these results combine to conclude that all the zeros lie in the annulus $1 < |z| < 2$.

Our next example is also from \cite{Stephen}.
\begin{example}
We want to find the number of the roots of the equation
\[
\frac{z^2-4}{z^2+4} + \frac{2z^2 - 1}{z^2 + 6} = 0
\]
that lie within the unit circle $|z| = 1$.
\end{example}

Observe that 
\begin{align*}
    \frac{z^2-4}{z^2+4} + \frac{2z^2 - 1}{z^2 + 6} &= \frac{(z^2-4)(z^2+6) + (2z^2 -1 )(z^2+4)}{(z^2+4)(z^2+6)}\\
    &= \frac{(z^4 + 2z^2 - 24) + (2z^4 + 7z^2 - 4)}{(z^2+4)(z^2+6)}\\
    &= \frac{3z^4 + 9z^2 - 28}{(z^2+4)(z^2+6)}.
\end{align*}
Let 
\begin{align*}
    f(z) &= \frac{28}{(z^2+4)(z^2+6)}\\
    g(z) &= \frac{3z^4 + 9z^2 - 28}{(z^2+4)(z^2+6)}.
\end{align*}
Observe that, on $|z| = 1$,
\begin{align*}
    |f(z) + g(z)| &= \frac{|3z^4 + 9z^2|}{|(z^2+4)(z^2+6)|}\\
    &\leq \frac{3|z|^4 + 9|z|^2}{|(z^2+4)(z^2+6)|}\\
    &= \frac{12}{|(z^2+4)(z^2+6)|}
\end{align*}
\begin{align*}
    &< \frac{28}{|(z^2+4)(z^2+6)|}\\
    &= \frac{|28|}{|(z^2+4)(z^2+6)|}\\
    &= |f(z)|.
\end{align*}
Notice that $f,g$ are analytic on the open ball $|z|< 3/2$, and that the unit circle $|z| = 1$ is inside that open ball.
Then, given $f(z) =  \frac{28}{(z^2+4)(z^2+6)}$ has no zeros inside the unit circle, Rouch\'{e}'s Theorem implies that $g(z) = \frac{3z^4 + 5z^2 - 28}{(z^2+4)(z^2+6)}$ has no zeros within the unit circle.\\

In this section, we reviewed theorems and concepts from Complex Analysis.
Many theorems related to series will be used when we find explicit form of generated sequence of functions.
We review Rouch\'{e}'s Theorem and studied some examples.
We will see how Rouch\'{e}'s Theorem become useful when we work on our main theorems.

\clearpage
\begin{center}
\section{CONSEQUENCES OF THEOREMS ABOUT ZERO DISTRIBUTIONS}
\end{center}
\thispagestyle{empty} 

This paper studies the zero distribution of generated sequence of functions.
In this section we will study some special cases of studies mentioned in `Introduction' for motivation.
These studies are about the zero distribution of generated sequence of functions.
We first study the zero distribution of the sequence $\{G_m(z)\}_{m\geq 0}$ generated as 
\[
\sum_{m=0}^{\infty}G_{m}(z)t^{m}=\frac{1}{1+t+z t^2}.
\]
In particular, we prove the theorem below.
\begin{thm}\label{ex1}
Let $\{G_m(z)\}$ be the sequence of polynomials generated by 
\[
\sum_{m=0}^{\infty}G_{m}(z)t^{m}=\frac{1}{1+t+z t^2}.
\]
For any $m$, the zeros of $G_m(z)$ lie in the interval $(1/4,\infty)$.
\end{thm}
This theorem is a special case of Theorem 1 in \cite{Tran}. 
Our plan is to show that (i) the degree of $G_m(z)$ is at most $\lfloor\frac{m}{2}\rfloor$ and (ii) $G_m(z)$ has at least $\lfloor\frac{m}{2}\rfloor$ zeros in $(1/4,\infty)$. 
Theorem \ref{ex1} will then follow from the Fundamental Theorem of Algebra. 
\begin{lem}\label{l1.1}
The sequence of polynomials $\{G_m(z)\}_{m\geq 0}$ satisfies the recurrence
\[
G_m(z)=-G_{m-1}(z)-zG_{m-2}(z),
\]
for each integer $m\geq 2$. 
Moreover, the degree of the polynomial $G_m(z)$ is at most $\lfloor\frac{m}{2}\rfloor$.
\end{lem}
\begin{proof}
We first show that the sequence $\{G_m(z)\}_{m\geq 0}$ satisfies the recurrence 
\[
G_m(z)=-G_{m-1}(z)-zG_{m-2}(z)
\]
with the initial condition $G_0(z)=1$ and $G_1(z)=-1$. 
Let $z\in\mathbb{C}$ be arbitrary.
Notice that $1 + t + zt^2$ does not have 0 as a zero, and hence $\frac{1}{1+t+zt^2}$ is analytic on $B(r,0)$ for some $0<r$ where $r$ is less than both moduli of the zeros of $1+t+zt^2$.
Theorem \ref{t5.0} implies the existence of a radius of convergence, say $R$, of the series $\sum_{m=0}^{\infty}G_{m}(z)t^{m}$.
Let $t\in\mathbb{C}$ such that $|t|<R$.
Then
\[
\sum_{m=0}^{\infty}G_{m}(z)t^{m}=\frac{1}{1+t+zt^2}
\]
holds.
We multiply both sides of 
\[
\sum_{m=0}^{\infty}G_{m}(z)t^{m}=\frac{1}{1+t+zt^2}
\]
by $1+t+zt^2$ to obtain
\[
1 = (1+t+zt^2)\sum_{m=0}^{\infty}G_{m}(z)t^{m}
\]
and then, by Proposition \ref{p5.0}, given that $|t| < R$, it holds that
\begin{align*}
 &(1+t+zt^2)\sum_{m=0}^{\infty}G_{m}(z)t^{m}\\
 = &G_0(z) + (G_0(z) + G_1(z))t + (zG_0(z) + G_1(z) + G_2(z))t^2 +\cdots.
\end{align*}
Consequently, 
\[
1 = G_0(z) + (G_0(z) + G_1(z))t + (zG_0(z) + G_1(z) + G_2(z))t^2 +\cdots.
\]
Because of Proposition \ref{p5.1}, we can equate the coefficients of $t^m$ in both sides, and get
\begin{align*}
    G_0(z) &= 1\\
    G_1(z) &= -1\\
    G_2(z) &= -zG_0(z) - G_1(z)\\
    &\vdots\\
    G_m(z) &= -zG_{m-2}(z) - G_{m-1}(z), m\geq 2\\
    &\vdots
\end{align*}
which yields the recurrence relation and the initial conditions for the sequence $\{G_m(z)\}_{m\geq 0}$:
\begin{align*}
G_0(z) &= 1\\
    G_1(z) &= -1\\
G_m(z) &= -zG_{m-2}(z) - G_{m-1}(z), m\geq 2.
\end{align*}
With this recurrence relation and initial condition, we can infer that $\{G_m(z)\}_{m\geq 0}$ is a sequence of polynomials.\\

The next goal is prove that the degree of $G_m(z)$ is at most  $\lfloor m/2 \rfloor$ by induction on $m$. The initial step for induction follows from $G_0(z) = 1$ and that $G_1(z) = -1$. For the induction step, we assume the degree of $G_k(z)$ is at most $\left\lfloor k/2\right\rfloor$, for $k=1,2,\dots, m-2,m-1$, for some $m\geq 2$. 
We consider two cases: $m$ being even and $m$ being odd.
If $m$ is even, then $m = 2a$, for some $a\in\mathbb{N}$, and consequently the degree of $G_{m-2}(z)$ is at most
\[ \left\lfloor \frac{m-2}{2}\right\rfloor = \left\lfloor \frac{2a-2}{2}\right\rfloor = \left\lfloor a-1\right\rfloor = a-1,
\] 
and the degree of $G_{m-1}(z)$ is 
at most
\[
\left\lfloor \frac{m-1}{2}\right\rfloor = \left\lfloor \frac{2a-1}{2}\right\rfloor = \left\lfloor a-1 + \frac{1}{2}\right\rfloor = a-1.
\]
It follows that the degree of $G_m(z) = -zG_{m-2}(z) - G_{m-1}(z)$ is less than equal to the maximum of $\deg(zG_{m-2}(z))\leq a-1 + 1$ and $\deg(G_{m-1}(z))\leq a-1$, and thus less than equal to 
\[
\max(a-1,a-1 + 1) = a = \left\lfloor m/2\right\rfloor.
\]
This implies that the inductive step holds for the case $m$ is even.

If $m$ is odd, then $m = 2b -1$, for some $b\in\mathbb{N}$.
Observe that the degree of $G_{m-2}(z)$ is at most
\[ 
\left\lfloor \frac{m-2}{2}\right\rfloor = \left\lfloor \frac{2b-1-2}{2}\right\rfloor = \left\lfloor b-2 + \frac{1}{2}\right\rfloor = b-2,
\] 
and that the degree of $G_{m-1}(z)$ is at most
\[
\left\lfloor \frac{m-1}{2}\right\rfloor = \left\lfloor \frac{2b-1-1}{2}\right\rfloor = \left\lfloor b-1 \right\rfloor = b-1.
\]
It implies that  the degree of $G_m(z) = -zG_{m-2}(z) - G_{m-1}(z)$ is less than equal to the maximum of $\deg(zG_{m-2}(z))\leq b-2 + 1$ and $\deg(G_{m-1}(z))\leq b-1$, and thus less than equal to 
\[
\max(b-2 + 1,b-1) = b-1 =  \left\lfloor b-1 + \frac{1}{2}\right\rfloor =  \left\lfloor \frac{2b-1}{2}\right\rfloor = \left\lfloor m/2\right\rfloor.
\]
This concludes that the inductive step holds, and thus our claim holds.
\end{proof}

Now let us prove that $G_m(z)$ has at least $\lfloor\frac{m}{2}\rfloor$ zeros in $(1/4,\infty)$.
The proof relies on the four functions $r(\theta)$, $t_1(\theta)$, $t_2(\theta)$ and $z(\theta)$ defined as follows on $\theta\in\left(\frac{\pi}{2},\pi\right)$,
\begin{equation}
\begin{split}\label{ed5.0}
    r(\theta) &= -2\cos(\theta),\\
    t_1(\theta) &= r(\theta) e^{i\theta},\\
    t_2(\theta) &= r(\theta) e^{-i\theta} ,\\
    z(\theta) &=\frac{1}{4\cos^2(\theta)}.
\end{split}
\end{equation}

\begin{lem}
For each $\theta\in\left(\frac{\pi}{2},\pi\right)$, the two zeros in $t$ of $1+t+z(\theta)t^2$ are $t_1(\theta)$ and $t_2(\theta)$, and they are non-zero and distinct.
\end{lem}
\begin{proof}
We will show that for each $\theta\in\left(\frac{\pi}{2},\pi\right)$
\begin{align*}
    t_1(\theta)+t_2(\theta) &= -\frac{1}{z(\theta)},\\
    t_1(\theta)t_2(\theta) &= \frac{1}{z(\theta)}.
\end{align*}
Indeed,
\begin{align*}
t_1(\theta)+t_2(\theta) &= r(\theta)e^{i\theta} + r(\theta) e^{-i\theta} \\
&= r(\theta)(\cos(\theta) + i\sin(\theta)) + r(\theta)(\cos(\theta) - i\sin(\theta)) \\
&= 2r(\theta) \cos(\theta) \\
&= -4\cos^2(\theta) \\
&= -\frac{1}{z(\theta)}.
\end{align*}
We also have 
\[
t_1(\theta)t_2(\theta) = r(\theta) e^{i\theta} r(\theta) e^{-i\theta} = r(\theta)^2 = 4\cos^2(\theta) = \frac{1}{z(\theta)}.
\]
Consequently, it follows that 
\begin{align*}
    z(\theta)(t - t_1(\theta))(t - t_2(\theta)) &= z(\theta)(t^2 - (t_1(\theta) + t_2(\theta))t + t_1(\theta)t_2(\theta))\\
    &= z(\theta)\left(t^2 +\frac{1}{z(\theta)}t +\frac{1}{z(\theta)}\right)\\
    &= 1+t+z(\theta)t^2
\end{align*}
This concludes that $t_1(\theta)$ and $t_2(\theta)$ are the zeros of $1+t+z(\theta)t^2$.
Notice that if $t_1(\theta) = t_2(\theta)$, then
\[
-\frac{1}{z(\theta)} = t_1(\theta) + t_2(\theta) = 2t_1(\theta)
\]
implies that $t_1(\theta), t_2(\theta)$ are real.
However, observe that, for $\theta\in \left(\frac{\pi}{2},\pi\right)$, $r(\theta) = -2\cos(\theta)$ is not 0  and that $e^{i\theta}$ is not real.
Hence $t_1(\theta) = t_2(\theta)$ does not hold.
Moreover, $t_1(\theta),t_2(\theta)\neq 0$.
\end{proof}

\begin{lem}
$G_m(z)$ has at least $\lfloor\frac{m}{2}\rfloor$ zeros on $(1/4,\infty)$.
\end{lem}
\begin{proof}

To prove this lemma, we express $G_m(z(\theta))$ as a function of $\theta$. Since $t_1 = t_1(\theta)$, $t_2 = t_2(\theta)$ are the distinct zeros of $1+t+z(\theta) t^2$, we have, for each $\theta\in(\pi/2,\pi)$,
\[
\frac{1}{1+t+z t^2} = \frac{1}{z(t-t_1)(t-t_2)}
\]
where $z=z(\theta) \ne 0$. By partial fractions, the right side is
\[
\frac{1}{z}\left(\frac{1}{(t-t_1)(t_1-t_2)} + \frac{1}{(t-t_2)(t_2-t_1)}\right).
\]
Meanwhile, observe that, for $t<\min(|t_1|,|t_2|)\ne 0$, geometric series implies that 
\[
\frac{1}{t-t_1} = \frac{1}{t_1}\frac{1}{\frac{t}{t_1} - 1} = -\frac{1}{t_1}\frac{1}{1- \frac{t}{t_1}} = -\frac{1}{t_1}\sum_{n=0}^{\infty}\left(\frac{t}{t_1}\right)^n.
\]
Similarly,
\[
\frac{1}{t-t_2} = -\frac{1}{t_2}\sum_{n=0}^{\infty}\left(\frac{t}{t_2}\right)^n.
\]
We combine identities above and deduce that
\begin{align*}
 &\sum_{m=0}^{\infty}G_{m}(z)t^{m} \\
= &\frac{1}{1+t+zt^2} \\
= &\frac{1}{z}\left(\left(-\frac{1}{t_1}\sum_{m=0}^{\infty}\left(\frac{t}{t_1}\right)^m\right)\frac{1}{(t_1-t_2)} +\left(-\frac{1}{t_2}\sum_{m=0}^{\infty}\left(\frac{t}{t_2}\right)^m\right)\frac{1}{(t_2-t_1)}\right)\\
    = &\frac{1}{z}\left(-\sum_{m=0}^{\infty}\frac{1}{t_1^{m+1}}\frac{1}{(t_1-t_2)}t^m -\sum_{m=0}^{\infty}\frac{1}{t_2^{m+1}}\frac{1}{(t_2-t_1)}t^m\right)\\
    = &\sum_{m=0}^{\infty}-\frac{1}{z}\left(\frac{1}{t_1^{m+1}(t_1-t_2)} + \frac{1}{t_2^{m+1}(t_2-t_1)}\right)t^m.
\end{align*}
Thus by the uniqueness of the power series in $t$ of $(1+t+zt^2)^{-1}$ around the origin, we conclude that
\[
G_m(z) = -\frac{1}{z}\left(\frac{1}{t_1^{m+1}(t_1-t_2)} + \frac{1}{t_2^{m+1}(t_2-t_1)}\right).
\]
We apply Definitions \ref{ed5.0} to get
\begin{align*}
    r &= r(\theta) = -2\cos(\theta),\\
    t_1(\theta) &= r(\theta)e^{i\theta},\\
    t_2(\theta) &= r(\theta)e^{-i\theta},\\
    z(\theta) &= \frac{1}{r(\theta)^2} = \frac{1}{4\cos^2(\theta)},
\end{align*}
and conclude that, for each $\theta\in(\pi/2,\pi)$,
\begin{align*}
    G_m(z(\theta)) &= -\frac{1}{z}\left(\frac{1}{t_1^{m+1}(t_1-t_2)} + \frac{1}{t_2^{m+1}(t_2-t_1)}\right)\\
    &= -\frac{-t_1^{m+1} + t_2^{m+1}}{zt_1^{m+1}t_2^{m+1}(t_1-t_2)}\\
    &= -\frac{-r^{m+1}e^{i(m+1)\theta} + r^{m+1}e^{-i(m+1)\theta}}{\left(\frac{1}{4\cos^2(\theta)}\right)r^{m+1}e^{i(m+1)\theta} r^{m+1}e^{-i(m+1)\theta}(re^{i\theta} - re^{-i\theta})}\\
    &= -4\cos^2(\theta)\frac{-r^{m+1}e^{i(m+1)\theta} + r^{m+1}e^{-i(m+1)\theta}}{r^{2m+2}(i2r\sin(\theta))}\\
	&= 4\cos^2(\theta)\frac{r^{m+1}(e^{i(m+1)\theta} - e^{-i(m+1)\theta})}{r^{2m+3}(i2\sin(\theta))}\\
	&= 4\cos^2(\theta)\frac{(i2\sin((m+1)\theta)}{r^{m+2}(i2\sin(\theta))}\\
    &= \frac{4\cos^2(\theta)\sin((m+1)\theta)}{r^{m+2}\sin(\theta)}.
\end{align*}
Observe that $\theta = \pi +\frac{1}{m+1}\pi, \pi  +\frac{2}{m+1}\pi, \dots, \pi + \frac{\lfloor m/2\rfloor}{m+1}\pi$ are the zeros of  $\sin((m+1)\theta)$ and thus are the zeros of
\[
G_m(z(\theta)) = \frac{4\cos^2(\theta)\sin((m+1)\theta)}{r^{m+2}\sin(\theta)}.
\]
For $\theta = \pi +\frac{1}{m+1}\pi, \pi  +\frac{2}{m+1}\pi, \dots, \pi + \frac{\lfloor m/2\rfloor}{m+1}\pi$, notice that $z(\theta) = \frac{1}{4\cos^2(\theta)}$ yields $\lfloor m/2\rfloor$ number of different values.
Here we have found $\lfloor m/2\rfloor$ number of different values of $z(\theta)$ such that $G_m(z(\theta)) = 0$.
Recall that Lemma \ref{l1.1} says that $G_m(z)$ is of degree at most $\lfloor m/2\rfloor$.
The Fundamental Theorem of Algebra implies that $G_m(z)$ has at most $\lfloor m/2\rfloor$ number of zeros counting multiplicities.
Since we have $\lfloor m/2\rfloor$ number of values of $z(\theta)$ such that $G_m(z(\theta)) = 0$, those $z(\theta)$'s are all the zeros of $G_m(z) = 0$.
Meanwhile, notice that $z(\theta) = \frac{1}{4\cos^2(\theta)}$.
Since $\cos^2(\theta)<1$, it follows that 
\[
z(\theta) = \frac{1}{4\cos^2(\theta)} > \frac{1}{4},
\]
for $\theta = \frac{1}{m+1}\pi, \dots, \frac{\lfloor m/2\rfloor}{m+1}\pi$.
This concludes that all the zeros of $G_m(z)$ lie in the interval $(1/4,\infty)$.
\end{proof}

We shall study another example, which is a spcial case of Theorem 1 in \cite{AT1}.
\begin{thm}\label{t1.2}
Let $\{J_m(z)\}_{m\geq 0}$ be the sequence of polynomials generated by
\[
\frac{1}{1 + t+ zt^3}.
\]
For each $m\geq 0$, all the zeros of $J_m(z)$ lie in the interval $(-\infty,-4/27)$.
\end{thm}

Our first goal is to prove that, for each $m\in\mathbb{N}$, $J_m(z)$ is of degree at most $\lfloor m/3\rfloor$.

\begin{lem}\label{l2.1}
For each $m\geq 0$, $J_m(z)$ is a polynomial of degree at most $\lfloor m/3\rfloor$.
\end{lem}

\begin{proof}
We first show that the sequence $\{J_m(z)\}$ satisfies the recurrence 
\[
J_m(z) = -J_{m-1}(z) - zJ_{m-3}(z)
\]
with the initial condition 
\begin{align*}
J_0(z) &= 1, &J_1(z) &= -1, &J_2(z) &= 1.
\end{align*}
Let $z\in\mathbb{C}$ be arbitrary.
Notice that $1 + t + zt^3$ does not have 0 as a zero, and hence $\frac{1}{1+t+zt^3}$ is analytic on $B(r,0)$ for some $r>0$ where $r$ is less than all the moduli of each of the zeros of $1+t+zt^3$.
Then Theorem \ref{t5.0} implies the existence of radius of convergence, say $R$, of the series $\sum_{m=0}^{\infty}J_{m}(z)t^{m}$.
Let $t\in\mathbb{C}$ such that $|t|<R$.
Then
\[
\sum_{m=0}^{\infty}J_{m}(z)t^{m}=\frac{1}{1+t+zt^3}
\]
holds.
We multiply both sides of 
\[
\sum_{m=0}^{\infty}J_{m}(z)t^{m}=\frac{1}{1+t+zt^3}
\]
by $1+t+zt^3$ to obtain
\[
1 = (1+t+zt^3)\sum_{m=0}^{\infty}J_{m}(z)t^{m}
\]
and then, by Proposition \ref{p5.0}, given that $|t| < R$, it holds that

\begin{align*}
1 &= (1 + t+ zt^3)\left(\sum_{m=0}^{\infty}J_m(z)t^m\right)\\
&= J_0(z) + J_0(z)t + zJ_0(z)t^3 + J_1(z)t + J_1(z)t^2 + zJ_1(z)t^4 +\cdots\\
&= J_0(z) + \left[J_1(z) + J_0(z)\right]t + \left[J_2(z) + J_1(z)\right]t^2 + \left[J_3(z) + J_2(z) + zJ_0(z)\right]t^3 +\\
&\cdots + \left[J_m(z) + J_{m-1}(z) + zJ_{m-3}(z)\right]t^m+\cdots
\end{align*}
Under Proposition \ref{p5.1}, we may equate the coefficients of $t^m$ in both sides, and it yields the following reccurence relation 
\begin{align*}
J_m(z) &= -J_{m-1}(z) - zJ_{m-3}(z)\text{, for }m\geq 3
\end{align*}
with the initial condition
\begin{align*}
J_0(z) &= 1\\
J_1(z) &= -1\\
J_2(z) & = 1.
\end{align*}
With this recurrence relation, we may deduce that $\{J_m(z)\}_{m\geq 0}$ is a sequence of polynomials.

The next step is to prove that $J_m(z)$ has degree at most $\lfloor m/3\rfloor$.
Observe that this claim holds for $J_0(z),J_1(z)$ and $J_2(z)$.
Now let us take the induction step.
Suppose that the degrees of $J_{m-3}(z), J_{m-1}(z)$ are at most $\lfloor(m-3)/3\rfloor$ and $\lfloor(m-1)/3\rfloor$ respectively, for some $m\geq 3$.
There are 3 cases:  $m \equiv 0$ (mod 3), or $m \equiv 1$ (mod 3) or $m \equiv 2$ (mod 3).\\

If  $m \equiv 0$ (mod 3), then $m = 3k$, for some $k\in\mathbb{N}$.
Notice that 
\[
\left\lfloor\frac{m-3}{3}\right\rfloor = \left\lfloor\frac{3k-3}{3}\right\rfloor = \left\lfloor k-1\right\rfloor = k-1
\]
and that 
\[
\left\lfloor\frac{m-1}{3}\right\rfloor = \left\lfloor\frac{3k-3 + 2}{3}\right\rfloor = \left\lfloor k-1 + \frac{2}{3}\right\rfloor = k-1.
\]
It follows that the degree of $J_m(z) = -zJ_{m-3}(z) - J_{m-1}(z)$ is at most
\[
\max\left(\left\lfloor\frac{m-3}{3}\right\rfloor +1, \left\lfloor\frac{m-1}{3}\right\rfloor\right) = \max\left(k-1 +1, k-1\right) = k =  \left\lfloor\frac{3k}{3}\right\rfloor = \left\lfloor\frac{m}{3}\right\rfloor.
\]

If $m \equiv 1$ (mod 3), then $m = 3k + 1$, for some $k\in\mathbb{N}$.
Notice that 
\[
\left\lfloor\frac{m-3}{3}\right\rfloor = \left\lfloor\frac{3k+1-3}{3}\right\rfloor = \left\lfloor k-1 + \frac{1}{3}\right\rfloor = k-1
\]
and that
\[
\left\lfloor\frac{m-1}{3}\right\rfloor = \left\lfloor\frac{3k+1-1}{3}\right\rfloor = \left\lfloor k\right\rfloor = k.
\]
Then the degree of $J_m(z) = -zJ_{m-3}(z) - J_{m-1}(z)$ is at most
\[
\max\left(\left\lfloor\frac{m-3}{3}\right\rfloor +1, \left\lfloor\frac{m-1}{3}\right\rfloor\right) = k = \left\lfloor k + \frac{1}{3}\right\rfloor = \left\lfloor\frac{3k + 1}{3}\right\rfloor = \left\lfloor\frac{m}{3}\right\rfloor.
\]

If $m \equiv 2$ (mod 3), then $m = 3k + 2$, for some $k\in\mathbb{N}$.
Notice that 
\[
\left\lfloor\frac{m-3}{3}\right\rfloor = \left\lfloor\frac{3k+2-3}{3}\right\rfloor = \left\lfloor k-1 + \frac{2}{3}\right\rfloor = k-1
\]
and that 
\[
\left\lfloor\frac{m-1}{3}\right\rfloor = \left\lfloor\frac{3k+2-1}{3}\right\rfloor = \left\lfloor k + \frac{1}{3}\right\rfloor = k,
\]
Then the degree of $J_m(z) = -zJ_{m-3}(z) - J_{m-1}(z)$ is at most
\[
\max\left(\left\lfloor\frac{m-3}{3}\right\rfloor +1, \left\lfloor\frac{m-1}{3}\right\rfloor\right) = k = \left\lfloor k + \frac{2}{3}\right\rfloor = \left\lfloor\frac{3k + 2}{3}\right\rfloor = \left\lfloor\frac{m}{3}\right\rfloor.
\]
This induces that, for all cases, $J_m(z)$ is of degree at most $\left\lfloor\frac{m}{3}\right\rfloor$, which implies that the inductive step hold.
It concludes that our claim holds.
\end{proof}

\begin{lem}\label{l1.2.1}
Let 
\begin{align}
r=r(\theta) &= \frac{1 - 4\cos^2(\theta)}{2\cos(\theta)},&t_1&=t_1(\theta) = r(\theta)e^{i\theta},\label{d1.1}\\
q=q(\theta) &= -2\cos(\theta),&t_2&=t_2(\theta) = r(\theta)e^{-i\theta},\label{d1.2}\\
z=z(\theta) &= -\frac{1}{r(\theta)^3q(\theta)},&t_3&=t_3(\theta) = r(\theta)q(\theta),\label{d1.3}
\end{align}
where $\theta\in(2\pi/3,\pi)$.
Then for each $\theta\in(2\pi/3,\pi)$, $t_1(\theta),t_2(\theta),t_3(\theta)$ are distinct and non-zero.
And, for all $t\in\mathbb{C}$, it holds that
\[
1+ t + zt^3 = z(t-t_1)(t-t_2)(t-t_3).
\]
\end{lem}
\begin{proof}
Notice that 
\[
z(t-t_1)(t-t_2)(t-t_3) = z(t^3 - (t_1 + t_2 + t_3)t^2 + (t_1t_2 + t_1t_3 + t_2t_3)t - t_1t_2t_3).
\]
We observe that 
\begin{align*}
t_1 + t_2 + t_3 &= r(\theta)e^{i\theta} + r(\theta)e^{-i\theta} + r(\theta)q(\theta)\\
&= r(\theta)\left(\cos(\theta) + i\sin(\theta) + \cos(\theta) - i\sin(\theta) - 2\cos(\theta)\right)\\
&= r(\theta)(2\cos(\theta) - 2\cos(\theta))\\
&= 0,
\end{align*}
and that
\begin{align*}
t_1t_2 + t_2t_3 + t_3t_1 = &re^{i\theta}re^{-i\theta} + re^{-i\theta}r(-2\cos(\theta)) + r(-2\cos(\theta))re^{i\theta}\\
= &r(\theta)^2 + (-2\cos(\theta))r(\theta)^2(e^{-i\theta} + e^{i\theta})\\
= &r(\theta)^2 + (-2\cos(\theta))r(\theta)^2(2\cos(\theta))\\
= &r(\theta)^2 - 4\cos^2(\theta)r(\theta)^2\\
= &r(\theta)^2(1-4\cos^2(\theta))\\
= &r(\theta)^2\frac{1-4\cos^2(\theta)}{2\cos(\theta)}2\cos(\theta)\\
= &r(\theta)^3(2\cos(\theta))\\
= &r(\theta)^3(-q(\theta))\\
= &\frac{1}{z(\theta)}.
\end{align*}
Notice that
\[
t_1t_2t_3 =  r(\theta)e^{i\theta} r(\theta)e^{-i\theta} r(\theta)q(\theta) =  r(\theta)^3q(\theta) = -\frac{1}{z(\theta)}.
\]
Above deductions imply that 
\begin{align*}
z(t-t_1)(t-t_2)(t-t_3) &= z(t^3 - (t_1 + t_2 + t_3)t^2 + (t_1t_2 + t_1t_3 + t_2t_3)t - t_1t_2t_3)\\
&= z(\theta)\left(t^3 - 0t^2+ \frac{1}{z(\theta)}t - \left(-\frac{1}{z(\theta)}\right)\right)\\
&= 1 + t + z(\theta)t^3.
\end{align*}
Let $\theta\in(2\pi/3,\pi)$ be arbitrary.
It follows that $\cos(\theta)\in \left(-1,-\frac{1}{2}\right)$, and this induces that $r=r(\theta) = \frac{1 - 4\cos^2(\theta)}{2\cos(\theta)}\neq0$.
Since $\sin\theta\neq0$, $e^{i\theta} = \cos\theta + i\sin\theta$, $e^{-i\theta} = \cos\theta - i\sin\theta$ are distinct and not real.
With $q(\theta) = -2\cos(\theta)$ being real and non-zero, we may infer that $q(\theta),e^{i\theta},e^{-i\theta}$ are distinct.
Consequently, $t_1(\theta),t_2(\theta),t_3(\theta)$ are distinct and non-zero.
\end{proof}

The next lemma gives an explicit formula for $J_m(z(\theta))$.
\begin{lem}\label{l1.2.2}
With $z(\theta)$, $q(\theta)$, $r(\theta)$, $t_k(\theta)$, $1\le k \le 3$, defined as in Lemma \ref{l1.2.1}, for each $\theta\in (2\pi/3,\pi)$ we have
\begin{align*}
J_m(z(\theta)) = \frac{q^{m+1}[(-2\cos\theta)\sin(m+1)\theta-\sin(m+2)\theta] + \sin\theta}{z(t_1-t_2)(t_2-t_3)(t_3-t_1)t_1^{m+1}t_2^{m+1}t_3^{m+1}}2ir^{2m+3}(\theta).
\end{align*}
\end{lem}
\begin{proof}
Let $\theta\in(2\pi/3,\pi)$ be arbitrary. 
We may deduce that  $q(\theta)$, $r(\theta)\neq 0$, and consequently $t_1(\theta),t_2(\theta),t_3(\theta)\neq 0$.
Using partial fractions yields that, for $|t|<\min(|t_1|,|t_2|,|t_3|)$, 
\begin{align}
\frac{1}{1 + t + zt^3} = &\frac{1}{z(t-t_1)(t-t_2)(t-t_3)}\nonumber\\
\begin{split}\label{e3.1}
= &\frac{1}{z}\cdot\frac{1}{(t-t_1)(t_2-t_1)(t_3-t_1)}+ \frac{1}{z}\cdot\frac{1}{(t-t_2)(t_1-t_2)(t_3-t_2)} \\
&+ \frac{1}{z}\cdot\frac{1}{(t-t_3)(t_1-t_3)(t_2-t_3)}.
\end{split}
\end{align}
Since $|t|<\min(|t_1|,|t_2|,|t_3|)$, we may deduce, with geometric series, that expression \eqref{e3.1} is equal to
\begin{align*}
&-\sum_{m=0}^{\infty}\frac{1}{z(t_2-t_1)(t_3-t_1)t_1^{m+1}}t^m  -\sum_{m=0}^{\infty}\frac{1}{z(t_1-t_2)(t_3-t_2)t_2^{m+1}}t^m \\ &-\sum_{m=0}^{\infty}\frac{1}{z(t_1-t_3)(t_2-t_3)t_3^{m+1}}t^m
\end{align*}
If we combine above series into one series, the coefficient of $t^m$ in the series is
\begin{align*}
    &-\frac{1}{z(t_2-t_1)(t_3-t_1)t_1^{m+1}} - \frac{1}{z(t_1-t_2)(t_3-t_2)t_2^{m+1}} - \frac{1}{z(t_1-t_3)(t_2-t_3)t_3^{m+1}}\\
    = &-\frac{-t_2^{m+1}t_3^{m+1}(t_2-t_3) - t_1^{m+1}t_3^{m+1}(t_3-t_1) - t_1^{m+1}t_2^{m+1}(t_1-t_2)}{z(t_1-t_2)(t_2-t_3)(t_3-t_1)t_1^{m+1}t_2^{m+1}t_3^{m+1}}\\
    = &\frac{t_2^{m+2}t_3^{m+1} - t_2^{m+1}t_3^{m+2} + t_1^{m+1}t_3^{m+2} - t_1^{m+2}t_3^{m+1} + t_1^{m+2}t_2^{m+1} - t_1^{m+1}t_2^{m+2}}{z(t_1-t_2)(t_2-t_3)(t_3-t_1)t_1^{m+1}t_2^{m+1}t_3^{m+1}}\\
    =& \frac{t_3^{m+2}(t_1^{m+1} - t_2^{m+1}) +t_3^{m+1}(t_2^{m+2}- t_1^{m+2}) + t_1^{m+2}t_2^{m+1} - t_1^{m+1}t_2^{m+2}}{z(t_1-t_2)(t_2-t_3)(t_3-t_1)t_1^{m+1}t_2^{m+1}t_3^{m+1}}.
\end{align*}
Recall that from Equations \eqref{d1.1}, \eqref{d1.2}, \eqref{d1.3} that 
\begin{align*}
r=r(\theta) &= \frac{1 - 4\cos^2(\theta)}{2\cos(\theta)},&t_1&=t_1(\theta) = r(\theta)e^{i\theta},\\
q=q(\theta) &= -2\cos(\theta),&t_2&=t_2(\theta) = r(\theta)e^{-i\theta},\\
z=z(\theta) &= -\frac{1}{r(\theta)^3q(\theta)},&t_3&=t_3(\theta) = r(\theta)q(\theta).
\end{align*}
These imply that 
\begin{align*}
&\sum_{m=0}^{\infty}\left(\frac{t_3^{m+2}(t_1^{m+1} - t_2^{m+1}) +t_3^{m+1}(t_2^{m+2}- t_1^{m+2}) + t_1^{m+2}t_2^{m+1} - t_1^{m+1}t_2^{m+2}}{z(t_1-t_2)(t_2-t_3)(t_3-t_1)t_1^{m+1}t_2^{m+1}t_3^{m+1}}\right)t^m\\
= &\sum_{m=0}^{\infty}\left(\frac{q^{m+2}(2i\sin(m+1)\theta) +q^{m+1}(-2i\sin(m+2)\theta) + 2i\sin\theta}{z(t_1-t_2)(t_2-t_3)(t_3-t_1)t_1^{m+1}t_2^{m+1}t_3^{m+1}}\right)r^{2m+3}t^m\\
= &\sum_{m=0}^{\infty}\left(\frac{q^{m+1}[(-2\cos\theta)\sin(m+1)\theta-\sin(m+2)\theta] + \sin\theta}{z(t_1-t_2)(t_2-t_3)(t_3-t_1)t_1^{m+1}t_2^{m+1}t_3^{m+1}}\right)2ir^{2m+3}t^m.
\end{align*}
This induces that 
\begin{align*}
\sum_{m=0}^{\infty}J_m(z(\theta))t^m = &\frac{1}{1 + t + z(\theta)t^3} 
\end{align*}
\begin{align*}
= &\sum_{m=0}^{\infty}\left(\frac{q^{m+1}[(-2\cos\theta)\sin(m+1)\theta-\sin(m+2)\theta] + \sin\theta}{z(t_1-t_2)(t_2-t_3)(t_3-t_1)t_1^{m+1}t_2^{m+1}t_3^{m+1}}\right)2ir^{2m+3}t^m
\end{align*}
which concludes that, for all $\theta\in (2\pi/3,\pi)$,
$$J_m(z(\theta)) = \frac{q^{m+1}[(-2\cos\theta)\sin(m+1)\theta-\sin(m+2)\theta] + \sin\theta}{z(t_1-t_2)(t_2-t_3)(t_3-t_1)t_1^{m+1}t_2^{m+1}t_3^{m+1}}2ir^{2m+3}.\qedhere$$
\end{proof}

With all the lemmas at our disposal, we will prove Theorem \ref{t1.2}. 
\begin{proof}
To proceed, we note that the numerator of
\[\frac{q^{m+1}[(-2\cos\theta)\sin(m+1)\theta-\sin(m+2)\theta] + \sin\theta}{z(t_1-t_2)(t_2-t_3)(t_3-t_1)t_1^{m+1}t_2^{m+1}t_3^{m+1}},
\]
is $0$ when 
$$2\cos\theta\sin(m+1)\theta+\sin(m+2)\theta = \frac{\sin\theta}{q^{m+1}}.$$
Since, for $\theta\in(2\pi/3,\pi)$, 
\[|q(\theta)| = |-2\cos\theta| > 1\]
 and 
\[
0 < \sin\theta < 1,
\] 
we have
\[
\left| \frac{\sin\theta}{q^{m+1}}\right| < 1,
\] 
for all $\theta\in(2\pi/3,\pi)$.\\

Meanwhile, notice that
\begin{align*}
&2\cos\theta\sin(m+1)\theta+\sin(m+2)\theta \\
= &\sin((m+1)\theta + \theta) + \sin((m+1)\theta - \theta) + \sin(m+2)\theta\\
= &2\sin(m+2)\theta + \sin m\theta.
\end{align*}
When $\theta = \frac{4k - 3}{2(m+2)}\pi\in(2\pi/3,\pi), k \in\mathbb{N}$, we have 
\[
2\sin(m+2)\theta + \sin m\theta = 2 + \sin m\theta \geq 1.
\]
Similarly, if $\theta = \frac{4k - 1}{2(m+2)}\pi, k\in \mathbb{N}$, then
\[
2\sin(m+2)\theta + \sin m\theta = -2 + \sin m\theta \leq -1.
\]
It follows that, at $\theta = \frac{4k - 3}{2(m+2)}\pi\in(2\pi/3,\pi), k \in\mathbb{N}$,
\[
2\cos\theta\sin(m+1)\theta+\sin(m+2)\theta - \frac{\sin\theta}{q^{m+1}} >0
\]
 and, at  $\theta = \frac{4k - 1}{2(m+2)}\pi\in(2\pi/3,\pi), k\in \mathbb{N}$,
\[
2\cos\theta\sin(m+1)\theta+\sin(m+2)\theta - \frac{\sin\theta}{q^{m+1}} <0.
\]

Then, for all $k\in\mathbb{N}$ with
$$\left(\frac{4k - 3}{2(m+2)}\pi, \frac{4k - 1}{2(m+2)}\pi\right)\subset (2\pi/3,\pi)$$ or 
$$ \left(\frac{4k - 1}{2(m+2)}\pi, \frac{4k +1}{2(m+2)}\pi\right) \subset (2\pi/3,\pi),$$ the Intermediate Value Theorem induces that 
the equation in $\theta$
\begin{equation}\label{e1.2.1}
q^{m+1}[(-2\cos\theta)\sin(m+1)\theta-\sin(m+2)\theta] + \sin\theta = 0
\end{equation}
has a root on the corresponding interval.\\

Our next goal is to show that Equation \eqref{e1.2.1} has at least $\left\lfloor\frac{m}{3}\right\rfloor$ solutions on the interval $(2\pi/3,\pi)$.
Observe that
\[
\frac{2}{3}\pi <\frac{4k+1}{2(m+2)}\pi \Leftrightarrow 2\cdot2(m + 2) < 3\cdot (4k+1) \Leftrightarrow \frac{m}{3} + \frac{5}{12} < k.
\]
Notice that
\begin{align*}
\frac{4k+3}{2(m+2)}\pi <\pi \Leftrightarrow k < \frac{m}{2} + \frac{1}{4}.
\end{align*}
These inequalities combine to induce that
\begin{equation}\label{ine1.2.1}
k\in\left(\frac{m}{3} + \frac{5}{12}, \frac{m}{2} + \frac{1}{4}\right)\Rightarrow \left(\frac{4k+1}{2(m+2)}\pi, \frac{4k+3}{2(m+2)}\pi\right)\subset \left(\frac{2}{3}\pi, \pi\right).
\end{equation}
Likewise, observe that 
\[
\frac{2}{3}\pi < \frac{4k+3}{2(m+2)}\pi\Leftrightarrow \frac{m}{3} - \frac{1}{12} < k,
\]
and
\[
\frac{4k+5}{2(m+2)}\pi < \pi\Leftrightarrow k < \frac{m}{2} - \frac{1}{4}.
\]
We get that
\begin{equation}\label{ine1.2.2}
k\in\left(\frac{m}{3} - \frac{1}{12},\frac{m}{2} - \frac{1}{4}\right)\Rightarrow\left(\frac{4k+3}{2(m+2)}\pi, \frac{4k+5}{2(m+2)}\pi\right)\subset \left(\frac{2}{3}\pi, \pi\right).
\end{equation}

Now we will check 6 cases about $m$: $m \equiv 0,1,2,3,4,5 \pmod{6}$.
\begin{enumerate}
   
 \item If $m \equiv 0 \pmod{6}$, then $m = 6n$, for some $n\geq 0$.
We apply \eqref{ine1.2.1} to infer that, for
\[
2n+ \frac{5}{12}< k < 3n + \frac{1}{4},
\]
that is, for $k = 2n+1,2n + 2, \dots, 3n$, the intervals satisfy 
\[
\left(\frac{4k+1}{2(m+2)}\pi, \frac{4k+3}{2(m+2)}\pi\right)\subset \left(\frac{2}{3}\pi, \pi\right).
\]
On the other hand, notice that \eqref{ine1.2.2} induces that, for
\[
2n - \frac{1}{12}< k < 3n - \frac{1}{4},
\]
that is, for $k = 2n, 2n+1, \dots, 3n - 1$, the intervals satisfy 
\[
\left(\frac{4k+3}{2(m+2)}\pi, \frac{4k+5}{2(m+2)}\pi\right)\subset \left(\frac{2}{3}\pi, \pi\right).
\]
To sum up, we have 
\[
(3n - 2n) + (3n-1 - (2n - 1)) = 2n = \left\lfloor 2n\right\rfloor = \left\lfloor \frac{6n}{3}\right\rfloor = \left\lfloor \frac{m}{3}\right\rfloor
\]
number of distinct intervals contained in $\left(\frac{2}{3}\pi, \pi\right)$.
Consequently, there are $\left\lfloor \frac{m}{3}\right\rfloor$ zeros.\\

\item 
If $m \equiv 1 \pmod{6}$, then $m = 6n + 1$ for some $n\geq 0$.
We apply \eqref{ine1.2.1} to infer that, for
\[
2n+ \frac{9}{12}< k < 3n + \frac{3}{4},
\]
that is, for $k = 2n+1,2n + 2, \dots, 3n$, the intervals satisfy 
\[
\left(\frac{4k+1}{2(m+2)}\pi, \frac{4k+3}{2(m+2)}\pi\right)\subset \left(\frac{2}{3}\pi, \pi\right).
\]
On the other hand, notice that \eqref{ine1.2.2} induces that, for
\[
2n + \frac{3}{12}< k < 3n + \frac{1}{4},
\]
that is, for $k = 2n + 1, 2n+2, \dots, 3n$, the intervals satisfy 
\[
\left(\frac{4k+3}{2(m+2)}\pi, \frac{4k+5}{2(m+2)}\pi\right)\subset \left(\frac{2}{3}\pi, \pi\right).
\]
To sum up, we have 
\[
(3n - 2n) + (3n - 2n) = 2n = \left\lfloor 2n + \frac{1}{3}\right\rfloor = \left\lfloor \frac{6n + 1}{3}\right\rfloor = \left\lfloor \frac{m}{3}\right\rfloor
\]
number of distinct intervals contained in $\left(\frac{2}{3}\pi, \pi\right)$.
Consequently, there are $\left\lfloor \frac{m}{3}\right\rfloor$ zeros.\\

\item 
If $m \equiv 2 \pmod{6}$, then $m = 6n + 2$ for some $n\geq 0$.
We apply \eqref{ine1.2.1} to infer that, for
\[
2n+ \frac{13}{12}< k < 3n + \frac{5}{4},
\]
that is, for $k = 2n+2,2n + 3, \dots, 3n + 1$, the intervals satisfy 
\[
\left(\frac{4k+1}{2(m+2)}\pi, \frac{4k+3}{2(m+2)}\pi\right)\subset \left(\frac{2}{3}\pi, \pi\right).
\]
On the other hand, notice that \eqref{ine1.2.2} induces that ,for
\[
2n + \frac{7}{12}< k < 3n + \frac{3}{4},
\]
that is, for $k = 2n + 1, 2n+2, \dots, 3n$, the intervals satisfy 
\[
\left(\frac{4k+3}{2(m+2)}\pi, \frac{4k+5}{2(m+2)}\pi\right)\subset \left(\frac{2}{3}\pi, \pi\right).
\]
To sum up, we have 
\[
[3n + 1 - (2n + 1)] + (3n - 2n) = 2n = \left\lfloor 2n + \frac{2}{3}\right\rfloor = \left\lfloor \frac{6n + 2}{3}\right\rfloor = \left\lfloor \frac{m}{3}\right\rfloor
\]
number of distinct intervals contained in $\left(\frac{2}{3}\pi, \pi\right)$.
Consequently, there are $\left\lfloor \frac{m}{3}\right\rfloor$ zeros.\\

\item 
If $m \equiv 3 \pmod{6}$, then $m = 6n + 3$ for some $n\geq 0$.
We apply \eqref{ine1.2.1} to infer that, for
\[
2n+ \frac{17}{12}< k < 3n + \frac{7}{4},
\]
that is, for $k = 2n+2,2n + 3, \dots, 3n + 1$, the intervals satisfy 
\[
\left(\frac{4k+1}{2(m+2)}\pi, \frac{4k+3}{2(m+2)}\pi\right)\subset \left(\frac{2}{3}\pi, \pi\right).
\]
On the other hand, notice that \eqref{ine1.2.2} induces that, for
\[
2n + \frac{11}{12}< k < 3n + \frac{5}{4},
\]
that is, for $k = 2n + 1, 2n+2, \dots, 3n + 1$, the intervals satisfy 
\[
\left(\frac{4k+3}{2(m+2)}\pi, \frac{4k+5}{2(m+2)}\pi\right)\subset \left(\frac{2}{3}\pi, \pi\right).
\]
To sum up, we have 
\[
[3n + 1 - (2n + 1)] + [3n + 1 - (2n)] = 2n + 1 = \left\lfloor 2n + 1\right\rfloor = \left\lfloor \frac{6n + 3}{3}\right\rfloor = \left\lfloor \frac{m}{3}\right\rfloor
\]
number of distinct intervals contained in $\left(\frac{2}{3}\pi, \pi\right)$.
Consequently, there are $\left\lfloor \frac{m}{3}\right\rfloor$ zeros.\\

\item 
If $m \equiv 4 \pmod{6}$, then $m = 6n + 4$ for some $n\geq 0$.
We apply \eqref{ine1.2.1} to infer that, for
\[
2n+ \frac{21}{12}< k < 3n + \frac{9}{4},
\]
that is, for $k = 2n+2,2n + 3, \dots, 3n + 2$, the intervals satisfy 
\[
\left(\frac{4k+1}{2(m+2)}\pi, \frac{4k+3}{2(m+2)}\pi\right)\subset \left(\frac{2}{3}\pi, \pi\right).
\]
On the other hand, notice that \eqref{ine1.2.2} induces that, for
\[
2n + \frac{15}{12}< k < 3n + \frac{7}{4},
\]
that is, for $k = 2n + 2, 2n+3, \dots, 3n + 1$, the intervals satisfy
\[
\left(\frac{4k+3}{2(m+2)}\pi, \frac{4k+5}{2(m+2)}\pi\right)\subset \left(\frac{2}{3}\pi, \pi\right).
\]
To sum up, we have 
\begin{align*}
    [3n + 2 - (2n + 1)] + [3n + 1 - (2n + 1)] = &2n + 1 \\
    = &\left\lfloor 2n + 1 + \frac{1}{3}\right\rfloor \\
    = &\left\lfloor \frac{6n + 4}{3}\right\rfloor \\
    = &\left\lfloor \frac{m}{3}\right\rfloor
\end{align*}
number of distinct intervals contained in $\left(\frac{2}{3}\pi, \pi\right)$.
Consequently, there are $\left\lfloor \frac{m}{3}\right\rfloor$ zeros.\\

\item 
If $m \equiv 5 \pmod{6}$, then $m = 6n + 5$ for some $n\geq 0$.
We apply \eqref{ine1.2.1} to infer that, for
\[
2n+ \frac{25}{12}< k < 3n + \frac{11}{4},
\]
that is, for $k = 2n+3,2n + 4, \dots, 3n + 2$, the intervals satisfy 
\[
\left(\frac{4k+1}{2(m+2)}\pi, \frac{4k+3}{2(m+2)}\pi\right)\subset \left(\frac{2}{3}\pi, \pi\right).
\]
On the other hand, notice that \eqref{ine1.2.2} induces that, for
\[
2n + \frac{19}{12}< k < 3n + \frac{9}{4},
\]
that is, for $k = 2n + 2, 2n+3, \dots, 3n + 2$, the intervals satisfy 
\[
\left(\frac{4k+3}{2(m+2)}\pi, \frac{4k+5}{2(m+2)}\pi\right)\subset \left(\frac{2}{3}\pi, \pi\right).
\]
To sum up, we have 
\begin{align*}
[3n + 2 - (2n + 2)] + [3n + 2 - (2n + 1)] = &2n + 1 \\
= &\left\lfloor 2n + 1 + \frac{2}{3}\right\rfloor \\
= &\left\lfloor \frac{6n + 5}{3}\right\rfloor \\
= &\left\lfloor \frac{m}{3}\right\rfloor
\end{align*}
number of distinct intervals contained in $\left(\frac{2}{3}\pi, \pi\right)$.
Consequently, there are $\left\lfloor \frac{m}{3}\right\rfloor$ zeros.\\
\end{enumerate}

Here we observe that $\left(\frac{2}{3}\pi, \pi\right)$ contains at least $\left\lfloor \frac{m}{3}\right\rfloor$ distinct zeros for all the cases $m\equiv 0,1\dots, 5\pmod{6}$.\\

The last thing to check is that $z(\theta)$ is injective, as a function about $\theta$, on $\left(\frac{2}{3}\pi, \pi\right)$.
Let us consider the notation \eqref{d1.3} of Lemma \ref{l1.2.1},
\[
z(\theta) = \frac{4\cos^2\theta}{(1-4\cos^2\theta)^3}.
\]
If we take the derivative of the function $z(\theta)$ above, we obtain 
\begin{align*}
    z(\theta)' &= -8\frac{\cos\theta\sin\theta}{(1-4\cos^2\theta)^3} - 3\cdot 8\cos\theta\sin\theta\frac{4\cos^2\theta}{(1-4\cos^2\theta)^4}\\
    &= -8\cos\theta\sin\theta\left(\frac{1-4\cos^2\theta + 12\cos^2\theta}{(1-4\cos^2\theta)^4}\right)\\
    &= -8\cos\theta\sin\theta\left(\frac{1+8\cos^2\theta}{(1-4\cos^2\theta)^4}\right).
\end{align*}
Notice that 
\[
z(\theta)'= -8\cos\theta\sin\theta\left(\frac{1+8\cos^2\theta}{(1-4\cos^2\theta)^4}\right)
\]
is positive on $\left(\frac{2}{3}\pi, \pi\right)$.
Thus $z(\theta)$ is strictly increasing on $\left(\frac{2}{3}\pi, \pi\right)$.
This concludes that 
$$J_m(z(\theta)) = \frac{q^{m+1}[(-2\cos\theta)\sin(m+1)\theta-\sin(m+2)\theta] + \sin\theta}{z(t_1-t_2)(t_2-t_3)(t_3-t_1)t_1^{m+1}t_2^{m+1}t_3^{m+1}}$$
has at least $\left\lfloor \frac{m}{3}\right\rfloor$ different $z(\theta)$ on $\left(\frac{2}{3}\pi, \pi\right)$ such that $J_m(z(\theta)) = 0$.
Recall that Lemma \ref{l2.1} states that $J_m(z)$ is  of degree at most $\left\lfloor \frac{m}{3}\right\rfloor$.
Then the Fundamental Theorem of Algebra implies that $J_m(z)$ has at most $\left\lfloor \frac{m}{3}\right\rfloor$ zeros.
Thus all the zeros of $J_m(z)$ is of the form 
\[
z = z(\theta) = \frac{4\cos^2\theta}{(1-4\cos^2\theta)^3},
\]
where $\theta\in \left(\frac{2}{3}\pi, \pi\right)$.
Recall that $z(\theta)$ is increasing on $\left(\frac{2}{3}\pi, \pi\right)$.
Thus the supremum of $z(\theta)$ is 
\[
z(\pi) = -\frac{4}{27}.
\]

This concludes that our claim holds.
\end{proof}


\clearpage


\section{POLYNOMIALS GENERATED BY $\displaystyle{\frac{1}{(at^2+bt+c)(1-tz)}}$}
The remainder of the paper provides original results about the zeros of a sequence of Taylor polynomials.
In this section and the following sections, we will study the pattern of zeros of sequence of functions generated as 
\[
\sum_{m=0}^{\infty}P_m(z)t^m = \frac{1}{(at^2+bt+c)(1-tz)}
\]
This section is contributed to find the explicit form of generated sequence of functions $\{P_m(z)\}_{m\geq 0}$.
First we need to prove that $P_m(z)$ is a polynomial, for all $m\geq 0$.

\begin{lem}\label{l5.0}
Let $a,b,c\in\mathbb{R}$ and $c\neq 0$.
Let  $\{P_m(z)\}_{m\geq 0}$ be a sequence of functions of $z$ generated as follows:
\[
\sum_{m=0}^{\infty}P_m(z) t^m = \frac{1}{(at^2+bt + c)(1-zt)}.
\]
Then $\{P_m(z)\}_{m\geq 0}$ is a sequence of polynomials in $z$ with the following recurrence relation 
\[
P_m(z) = \frac{1}{c}\left(azP_{m-3} - (a-bz)P_{m-2}(z) - (b-cz)P_{m-1}(z)\right),
\]
with the initial conditions
\begin{align*}
P_0(z) &= \frac{1}{c},\\
P_1(z) &= -\frac{1}{c^2}(b-cz),\\
P_2(z) &= \frac{1}{c^3}(b-cz)^2   -\frac{1}{c^2}(a-bz).
\end{align*}
\end{lem}

\begin{proof}
Let $z\in\mathbb{C}$ be arbitrary.
We will first find the radius of convergence.
Let us first consider the cases $a=b=z = 0$.
Observe that
\[
\frac{1}{(at^2+bt + c)(1-zt)} = \frac{1}{c}.
\]
Thus $\frac{1}{(at^2+bt + c)(1-zt)}$ is entire.
Hence, by Theorem \ref{t5.0}, the radius of convergence is infinity.

For the rest of the cases, we will find an explicit $R>0$ such that $\frac{1}{(at^2+bt + c)(1-zt)}$ is analytic on $B(R,0)$ and $(at^2+bt + c)(1-zt)$ has a zero on $|w| = R$.

If $a=b=0$ and $z\neq 0$, observe that 
\[
\frac{1}{(at^2+bt + c)(1-zt)} = \frac{1}{c(1-zt)}
\]
is analytic on $B(|1/z|;0)$.
Notice that $1/z$ is the only zero of $(at^2+bt + c)(1-zt)$. 
Thus $R = |1/z|$.

Now consider the cases: $a\neq 0$ or $b\neq 0$.
If $a = 0$ and $b\neq 0$ then 
\[
\frac{1}{(at^2+bt + c)(1-zt)} = \frac{1}{(bt + c)(1-zt)}.
\]
which is analytic on $B(R;0)$  where 
\[
R =\left\{\begin{array}{ll}\min(\left|\frac{c}{b}\right|, |1/z|) &\text{ if } z\neq 0\\ \left|\frac{c}{b}\right| &\text{ if } z = 0\end{array}\right.
\]
If $a\neq 0$, let $t_1,t_2$ be the zeros of $at^2+bt + c$.
Notice that $\frac{1}{(at^2+bt + c)(1-zt)}$ is analytic on $B(R;0)$ where 
\[
R =\left\{\begin{array}{ll}\min(|t_1|, |t_2|, |1/z|) &\text{ if } z\neq 0\\ \min(|t_1|, |t_2|) &\text{ if } z = 0\end{array}\right.
\]
Thus in all the cases we obtain an open ball, say $B(R;0)$ in which  $\frac{1}{(at^2+bt + c)(1-zt)}$ is analytic and $(at^2+bt + c)(1-zt)$ has a zero on $|w| = R$.

In fact, we are now going to prove that $R$ is the radius of convergence, $r$, of the series $\sum_{m=0}^\infty P_m(z) t^m$.
By Theorem \ref{t5.0}, $r \ge R$. Suppose, by contradiction, that $r>R$.
By Proposition \ref{p5.2}, $\sum_{m=0}^\infty P_m (z) t^m$ is analytic as a function in $t$ on $B(r,0)$.
Since $(at^2+bt + c)(1-zt)$ has a root on $|w| = R$, we may infer that $(at^2+bt + c)(1-zt)$ has a root on $\overline{B}(R,0)$, and let us denote the root by $z_0$.
Here we have $|z_0| = R < r$, and hence it holds that $z_0\in B(r,0)$.
Since $\sum_{m=0}^\infty P_m(z) t^m$ is analytic on $B(r,0)$, it holds that 
\[
\sum_{m=0}^\infty P_m(z) z_0^m  = z_1,
\]
for some $z_1\in\mathbb{C}$.
However, we have that $z_0$ is a root of $(at^2+bt + c)(1-zt)$ and thus 
\[
\sum_{m=0}^\infty P_m(z) t^m = \frac{1}{(at^2+bt + c)(1-zt)}
\]
is not defined at $z = z_0$.
So we encounter a contradiction.
Thus we may conclude that $R$ is the radius of convergence of $\sum_{m=0}^\infty P_m(z) t^m$.
This implies that 
\[
\sum_{m=0}^{\infty}P_m(z) t^m = \frac{1}{(at^2+bt + c)(1-zt)}
\]
holds for $t\in B(R;0)$. 

Let $t\in B(R;0)$ be arbitrary.
Now let us multiply both sides by 
\[
(at^2+bt + c)(1-zt) = -azt^3 + (a-bz)t^2 + (b - cz)t + c.
\]
It follows that 
\[
(-azt^3 + (a-bz)t^2 + (b - cz)t + c)\sum_{m=0}^{\infty}P_m(z) t^m = 1.
\]
If we distribute
\[
(-azt^3 + (a-bz)t^2 + (b - cz)t + c)\sum_{m=0}^{\infty}P_m(z) t^m,
\]
we obtain
\begin{align*} &-azt^3\sum_{m=0}^{\infty}P_m(z) t^m + (a-bz)t^2\sum_{m=0}^{\infty}P_m(z) t^m  + (b - cz)t\sum_{m=0}^{\infty}P_m(z) t^m \\
&+ c\sum_{m=0}^{\infty}P_m(z) t^m.
\end{align*}
We shall reorder the index of above series as follows:
\begin{align*}
&\sum_{m=3}^{\infty}-azP_{m-3}(z) t^{m} + \sum_{m=2}^{\infty}(a-bz)P_{m-2}(z) t^m + \sum_{m=1}^{\infty}(b-cz)P_{m-1}(z) t^m\\
 &+ \sum_{m=0}^{\infty}cP_m(z) t^m.
\end{align*}
We next pull a couple of terms to make all the series start from the same index, and we obtain
\begin{equation}\label{e4.1}
\begin{split}
1 =&\sum_{m=3}^{\infty}-azP_{m-3}(z) t^{m}\\ 
&+ (a-bz)P_{0}(z) t^2  + \sum_{m=3}^{\infty}(a-bz)P_{m-2}(z) t^m\\
&+ (b-cz)P_{0}(z) t^1  + (b-cz)P_{1}(z) t^2+ \sum_{m=3}^{\infty}(b-cz)P_{m-1}(z) t^m\\
&+ cP_0(z) + cP_1(z) t + cP_2(z) t^2 + \sum_{m=3}^{\infty}cP_m(z) t^m.
\end{split}
\end{equation}

Since $|t|$ is less than the radius of convergence of the power series  $\sum_{m=0}^{\infty}P_m(z) t^m$, by Theorem \ref{t5.1}, the power series $\sum_{m=0}^{\infty}P_m(z) t^m$ converges absolutely.
It follows that $\sum_{m=0}^{\infty}|P_m(z)| |t|^m$ converges.

We already fixed $t$ and we consider it as a constant.
Hence
\begin{align*}
|-azt^3|\sum_{m=0}^{\infty}|P_m(z)| |t|^m &=  \sum_{m=0}^{\infty}|-azP_m(z) t^{m+3}|,\\
 |(a-bz)t^2|\sum_{m=0}^{\infty}|P_m(z)| |t|^m &=   \sum_{m=0}^{\infty}|(a-bz)P_m(z)t^{m+2}|, \\
|(b - cz)t|\sum_{m=0}^{\infty}|P_m(z)| |t|^m  &= \sum_{m=0}^{\infty} |(b - cz)P_m(z)t^{m+1}|,\\
|c|\sum_{m=0}^{\infty}|P_m(z)| |t|^m &= \sum_{m=0}^{\infty}|cP_m(z)t^m|.
\end{align*}
Thus 
\[
\sum_{m=3}^{\infty}-azP_{m-3}(z) t^{m},\sum_{m=2}^{\infty}(a-bz)P_{m-2}(z) t^m , \sum_{m=1}^{\infty}(b-cz)P_{m-1}(z) t^m, \sum_{m=0}^{\infty}cP_m(z) t^m
\]
are absolutely convergent on $B(R;0)$.
By Proposition \ref{p5.2}, all 
\[
\sum_{m=3}^{\infty}-azP_{m-3}(z) t^{m},\sum_{m=2}^{\infty}(a-bz)P_{m-2}(z) t^m , \sum_{m=1}^{\infty}(b-cz)P_{m-1}(z) t^m, \sum_{m=0}^{\infty}cP_m(z) t^m
\]
 are analytic.
Then, by Theorem \ref{t5.0}, all 
\[
\sum_{m=3}^{\infty}-azP_{m-3}(z) t^{m},\sum_{m=2}^{\infty}(a-bz)P_{m-2}(z) t^m , \sum_{m=1}^{\infty}(b-cz)P_{m-1}(z) t^m, \sum_{m=0}^{\infty}cP_m(z) t^m
\]
have radius of convergence greater than or equal to $R$.
Hence $R$ is less than or equal to each of the radii of convergence of all the series:
\[
\sum_{m=3}^{\infty}-azP_{m-3}(z) t^{m},\sum_{m=2}^{\infty}(a-bz)P_{m-2}(z) t^m , \sum_{m=1}^{\infty}(b-cz)P_{m-1}(z) t^m , \sum_{m=0}^{\infty}cP_m(z) t^m
\]
Since we chose $t$ so that $|t| < R$, by Proposition \ref{p5.0}, we can combine those series as follows: 
\begin{align*}
    &\sum_{m=3}^{\infty}-azP_{m-3}(z) t^{m} + \sum_{m=3}^{\infty}(a-bz)P_{m-2}(z) t^m + \sum_{m=3}^{\infty}(b-cz)P_{m-1}(z) t^m \\
&+ \sum_{m=3}^{\infty}cP_m(z) t^m\\
= &\sum_{m=3}^{\infty}[-azP_{m-3}(z) + (a-bz)P_{m-2}(z) + (b-cz)P_{m-1}(z) + cP_m(z) ] t^{m}.
\end{align*}
Then Equation \eqref{e4.1} yields that
\begin{align*}
1= &cP_0(z) \\
&+\left[cP_1(z)  +(b-cz)P_{0}(z)\right] t\\
&+ \left[cP_2(z) + (b-cz)P_{1}(z)  +(a-bz)P_{0}(z)\right] t^2\\
&+\sum_{m=3}^{\infty}[-azP_{m-3}(z) + (a-bz)P_{m-2}(z) + (b-cz)P_{m-1}(z) + cP_m(z) ] t^{m}.
\end{align*}
By Proposition \ref{p5.1}, both sides have the same coefficient for $t^m$, for each $m\geq 0$, so if we consider both sides as infinite series, we obtain that 
\begin{align}
cP_0(z) &= 1\label{e5.6}\\
cP_1(z) +(b-cz)P_{0}(z) &= 0\label{e5.7}\\
cP_2(z) + (b-cz)P_{1}(z) +(a-bz)P_{0}(z)  &= 0\label{e5.8}\\
-azP_{m-3}(z) + (a-bz)P_{m-2}(z) + (b-cz)P_{m-1}(z) + cP_m(z)  &= 0.\label{e5.9}
\end{align}
Since we chose $z\in \mathbb{C}$ arbitrarily, this holds for all $z\in\mathbb{C}$.

Now we will prove, by induction, that $P_m(z)$ is a polynomial, for all $m\geq 0$.
Let us take the base step.
From Equation \eqref{e5.6}, we obtain 
\[
P_0(z) = \frac{1}{c}.
\]
So $P_0(z)$ is a polynomial.
Equation \eqref{e5.7} yields that 
\[
P_1(z) = -\frac{1}{c}(b-cz)P_{0}(z) = -\frac{1}{c^2}(b-cz)
\]
so $P_1(z)$ is a polynomial.
It follows that both $(b-cz)P_{1}(z)$ and $(a-bz)P_{0}(z)$ are polynomials.
Then with Equation \eqref{e5.8} we obtain 
\[
P_2(z) = -\frac{1}{c}\left( (b-cz)P_{1}(z) +(a-bz)P_{0}(z)\right) =   \frac{1}{c^3}(b-cz)^2   -\frac{1}{c^2}(a-bz)
\]
and we can observe that $-\frac{1}{c}(b-cz)P_{1}(z)   -\frac{1}{c}(a-bz)P_{0}(z)$ is polynomial in $z$.

Now we will set the induction hypothesis.
Assume that $P_{m-3}(z)$, $P_{m-2}(z)$, $P_{m-1}(z)$ are polynomials.
Then, by Equation \eqref{e5.9}, we obtain 
\[
P_m(z) = \frac{1}{c}\left(azP_{m-3} - (a-bz)P_{m-2}(z) - (b-cz)P_{m-1}(z)\right)
\]
which is a polynomial.
This concludes that $P_m(z)$ is a polynomial, for all $m=0,1,2,3,\dots$.
\end{proof}
This confirms that $P_m(z)$ is entire.
Before we find explicit form of $P_m(z)$, for the fluency of proof in some theorems, here we show an algebraic property.


\begin{prop}\label{e5.3}
\begin{align}
    &t_2^{m+1} - t_1^{m+1} + (t_1^{m+2} - t_2^{m+2})z + (t_2-t_1)t_1^{m+1}t_2^{m+1}z^{m+2}\label{e5.10}\\
    = &\left(\sum_{k=0}^{m}(t_2^{m+1-k} - t_1^{m+1-k})t_1^kt_2^kz^k\right)(t_1z-1)(t_2z-1).\nonumber
\end{align}
\end{prop}
\begin{proof}
One can check this holds for $m = 0,1$.
So assume $m\geq 2$.
With some algebraic calculation, we may deduce the following
\begin{align}
     t_1t_2(t_2^m-t_1^m) - t_2(t_2^{m+1} - t_1^{m+1}) - t_1(t_2^{m+1} - t_1^{m+1}) &= t_1^{m+2} - t_2^{m+2}\label{e5.11}\\
     \begin{split}
         t_1^kt_2^{m+1} - t_1^{m+1}t_2^k  - t_1^{k-1}t_2^{m+2} + t_1^{m+1}t_2^k  - t_1^kt_2^{m+1}\\
 + t_1^{m+2}t_2^{k-1} +t_1^{k-1}t_2^{m+2} - t_1^{m+2}t_2^{k-1} &= 0
     \end{split}\label{e5.12}\\
    (t_2^2 - t_1^2)t_1^mt_2^m - (t_2-t_1)t_1^mt_2^{m+1} - (t_2-t_1)t_1^{m+1}t_2^m &= 0.\label{e5.13}
\end{align}
Equations \eqref{e5.12} and \eqref{e5.13} induce that 
\begin{align}
    \begin{split}
        \sum_{k=2}^m\left(t_1^kt_2^{m+1} - t_1^{m+1}t_2^k\right)z^k + \sum_{k=2}^m\left(- t_1^{k-1}t_2^{m+2} + t_1^{m+1}t_2^k\right)z^k&\\
    +\sum_{k=2}^m\left(- t_1^kt_2^{m+1} + t_1^{m+2}t_2^{k-1}\right)z^k + \sum_{k=2}^m\left(t_1^{k-1}t_2^{m+2} - t_1^{m+2}t_2^{k-1}\right)z^k &= 0
    \end{split}\label{e5.14}\\
    [(t_2^2 - t_1^2)t_1^mt_2^m - (t_2-t_1)t_1^mt_2^{m+1} - (t_2-t_1)t_1^{m+1}t_2^m]z^{m+1} &= 0\label{e5.15},
\end{align}
respectively.
If we substitute Equation \eqref{e5.11} for the coefficient of $z$ in Equation \eqref{e5.10} and  add Equations \eqref{e5.14} and \eqref{e5.15}, then
\begin{align}
    &(t_2^{m+1} - t_1^{m+1}) + (t_1^{m+2} - t_2^{m+2})z + (t_2-t_1)t_1^{m+1}t_2^{m+1}z^{m+2} \nonumber\\
\begin{split}\label{e5.19}
    = &t_2^{m+1} - t_1^{m+1} +\left[t_1t_2(t_2^m-t_1^m) - t_2(t_2^{m+1} - t_1^{m+1}) - t_1(t_2^{m+1} - t_1^{m+1})\right]z\\
    &+ \sum_{k=2}^m\left(t_1^kt_2^{m+1} - t_1^{m+1}t_2^k\right)z^k + \sum_{k=2}^m\left(- t_1^{k-1}t_2^{m+2} + t_1^{m+1}t_2^k\right)z^k\\
    &+\sum_{k=2}^m\left(- t_1^kt_2^{m+1} + t_1^{m+2}t_2^{k-1}\right)z^k + \sum_{k=2}^m\left(t_1^{k-1}t_2^{m+2} - t_1^{m+2}t_2^{k-1}\right)z^k\\
&+\left[(t_2^2 - t_1^2)t_1^mt_2^m - (t_2-t_1)t_1^mt_2^{m+1} - (t_2-t_1)t_1^{m+1}t_2^m\right]z^{m+1}\\
&+(t_2-t_1)t_1^{m+1}t_2^{m+1}z^{m+2}.
\end{split}
\end{align}   
Observe that
\begin{itemize}
\item ${\displaystyle(t_2^{m+1} - t_1^{m+1}) +  (t_1t_2(t_2^m-t_1^m)z + \sum_{k=2}^m\left(t_1^kt_2^{m+1} - t_1^{m+1}t_2^k\right)z^k}$\\
${\displaystyle = \sum_{k=0}^m\left(t_1^kt_2^{m+1} - t_1^{m+1}t_2^k\right)z^k}$
\item
${\displaystyle -t_2(t_2^{m+1} - t_1^{m+1})z  - (t_2-t_1)t_1^mt_2^{m+1}z^{m+1} + \sum_{k=2}^{m}\left(- t_1^{k-1}t_2^{m+2} + t_1^{m+1}t_2^k\right)z^k}$\\
${\displaystyle=  \sum_{k=1}^{m+1}\left(- t_1^{k-1}t_2^{m+2} + t_1^{m+1}t_2^k\right)z^k}$
\item
${\displaystyle-t_1(t_2^{m+1} - t_1^{m+1})z  - (t_2-t_1)t_1^{m+1}t_2^mz^{m+1} + \sum_{k=2}^{m}\left( -t_1^kt_2^{m+1} + t_1^{m+2}t_2^{k-1}\right)z^k}$\\
${\displaystyle= \sum_{k=1}^{m+1}\left( -t_1^kt_2^{m+1} + t_1^{m+2}t_2^{k-1}\right)z^k}$
\item
${\displaystyle(t_2-t_1)t_1^{m+1}t_2^{m+1}z^{m+2} + (t_2^2 - t_1^2)t_1^mt_2^mz^{m+1} + \sum_{k=2}^m\left(t_1^{k-1}t_2^{m+2} - t_1^{m+2}t_2^{k-1}\right)z^k}$\\
${\displaystyle= \sum_{k=2}^{m+2}\left(t_1^{k-1}t_2^{m+2} - t_1^{m+2}t_2^{k-1}\right)z^k.}$
\end{itemize}
Then we may infer that Expression \eqref{e5.19} is equivalent to the following
\begin{align*}
&\left(\sum_{k=0}^{m}(t_2^{m+1-k} - t_1^{m+1-k})t_1^kt_2^{k}z^{k}\right) -\left(\sum_{k=1}^{m+1}(t_2^{m+2-k} - t_1^{m+2-k})t_1^{k-1}t_2^{k}z^{k}\right)\\
& - \left(\sum_{k=1}^{m+1}(t_2^{m+2-k} - t_1^{m+2-k})t_1^{k}t_2^{k-1}z^{k}\right)+ \left(\sum_{k=2}^{m+2}(t_2^{m+3-k} - t_1^{m+3-k})t_1^{k-1}t_2^{k-1}z^{k}\right)\\
= &\left(\sum_{k=2}^{m+2}(t_2^{m+3-k} - t_1^{m+3-k})t_1^{k-1}t_2^{k-1}z^{k}\right) -\left(\sum_{k=1}^{m+1}(t_2^{m+2-k} - t_1^{m+2-k})t_1^{k}t_2^{k-1}z^{k}\right)\\
& - \left(\sum_{k=1}^{m+1}(t_2^{m+2-k} - t_1^{m+2-k})t_1^{k-1}t_2^{k}z^{k}\right)+ \left(\sum_{k=0}^{m}(t_2^{m+1-k} - t_1^{m+1-k})t_1^kt_2^{k}z^{k}\right).
\end{align*}
If we make each sums start from the index 0, above expression is equal to the following
\begin{align*}
&\left(\sum_{k=0}^{m}(t_2^{m+1-k} - t_1^{m+1-k})t_1^{k+1}t_2^{k+1}z^{k+2}\right) -\left(\sum_{k=0}^{m}(t_2^{m+1-k} - t_1^{m+1-k})t_1^{k+1}t_2^kz^{k+1}\right)\\
&- \left(\sum_{k=0}^{m}(t_2^{m+1-k} - t_1^{m+1-k})t_1^kt_2^{k+1}z^{k+1}\right) + \left(\sum_{k=0}^{m}(t_2^{m+1-k} - t_1^{m+1-k})t_1^kt_2^{k}z^{k}\right).
\end{align*}
We may factor out $t_2z$ from the first and the third sums, and then above expression is equal to
\begin{align*}
&t_2z\left(\sum_{k=0}^{m}(t_2^{m+1-k} - t_1^{m+1-k})t_1^{k+1}t_2^{k}z^{k+1}\right) -\left(\sum_{k=0}^{m}(t_2^{m+1-k} - t_1^{m+1-k})t_1^{k+1}t_2^kz^{k+1}\right)\\
&- t_2z\left(\sum_{k=0}^{m}(t_2^{m+1-k} - t_1^{m+1-k})t_1^kt_2^{k}z^{k}\right) + \left(\sum_{k=0}^{m}(t_2^{m+1-k} - t_1^{m+1-k})t_1^kt_2^{k}z^{k}\right).
\end{align*}
If we factor out the common multiplicands 
$\left(\sum_{k=0}^{m}(t_2^{m+1-k} - t_1^{m+1-k})t_1^{k+1}t_2^kz^{k+1}\right)$ of the first two summands and $\left(\sum_{k=0}^{m}(t_2^{m+1-k} - t_1^{m+1-k})t_1^kt_2^{k}z^{k}\right)$ of the latter two summands,
above expression is equal to 
\begin{equation}\label{e5.16}
\begin{split}
    &\left(\sum_{k=0}^{m}(t_2^{m+1-k} - t_1^{m+1-k})t_1^{k+1}t_2^kz^{k+1}\right)(t_2z-1) \\
    &- \left(\sum_{k=0}^{m}(t_2^{m+1-k} - t_1^{m+1-k})t_1^kt_2^{k}z^{k}\right)(t_2z-1).
\end{split}
\end{equation}
Notice that 
\[
\sum_{k=0}^{m}(t_2^{m+1-k} - t_1^{m+1-k})t_1^{k+1}t_2^kz^{k+1} = t_1z\sum_{k=0}^{m}(t_2^{m+1-k} - t_1^{m+1-k})t_1^{k}t_2^kz^{k}.
\]
Then expression \eqref{e5.16} is equal to
\begin{align*}
    &t_1z\left(\sum_{k=0}^{m}(t_2^{m+1-k} - t_1^{m+1-k})t_1^kt_2^{k}z^{k}\right)(t_2z-1) \\
    &- \left(\sum_{k=0}^{m}(t_2^{m+1-k} - t_1^{m+1-k})t_1^kt_2^{k}z^{k}\right)(t_2z-1).
\end{align*}
If we factor out the common multiplicand 
$\left(\sum_{k=0}^{m}(t_2^{m+1-k} - t_1^{m+1-k})t_1^kt_2^{k}z^{k}\right)(t_2z-1)$,
above expression is equal to 
\begin{align*}
    &\left(\sum_{k=0}^{m}(t_2^{m+1-k} - t_1^{m+1-k})t_1^kt_2^kz^k\right)(t_1z-1)(t_2z-1).
\end{align*}
It concludes that 
\begin{equation}
\begin{split}
    &t_2^{m+1} - t_1^{m+1} + (t_1^{m+2} - t_2^{m+2})z + (t_2-t_1)t_1^{m+1}t_2^{m+1}z^{m+2}\\
    = &\left(\sum_{k=0}^{m}(t_2^{m+1-k} - t_1^{m+1-k})t_1^kt_2^kz^k\right)(t_1z-1)(t_2z-1).\qedhere
\end{split}
\end{equation}
\end{proof}

The following theorem shows the explicit form of $P_m(z)$ for $z \in \mathbb{C}\backslash\left\{0,\frac{1}{t_1},\frac{1}{t_2}\right\}$.

\begin{lem}\label{l5.1}
Let $a,t_1,t_2\in\mathbb{R}$ where $at_1t_2\neq 0$ and $t_1\neq t_2$.
Let  $\{P_m(z)\}_{m\geq 0}$ be a sequence of functions of $z$ generated as follows:
\[
\sum_{m=0}^{\infty}P_m(z) t^m = \frac{1}{a(t-t_1)(t-t_2)(1-zt)}.
\]
Then, for all $z\in\mathbb{C}\backslash\left\{0,\frac{1}{t_1},\frac{1}{t_2}\right\}$,
\[
P_m(z) =  \frac{1}{a(t_2-t_1)t_1^{m+1}t_2^{m+1}}\sum_{k=0}^{m}(t_2^{m+1-k} - t_1^{m+1-k})t_1^kt_2^kz^k.
\]
\end{lem}

\begin{proof}
Let $z\in\mathbb{C}\backslash\left\{0,\frac{1}{t_1},\frac{1}{t_2}\right\}$ be arbitrary. Using partial fractions, we solve for $A$, $B$, and $C$ from the equation below
\[
\frac{1}{a(1-zt)(t_1-t)(t_2-t)} = \frac{A}{t_1-t} + \frac{B}{t_2-t} + \frac{C}{1-zt},
\]
which yields 
\begin{align*}
&\frac{1}{a(t_1- t)(t_2 - t)(1 - zt)} \\
= &\frac{A}{t_1-t} + \frac{B}{t_2-t} + \frac{C}{1-zt} \\
= &\frac{A(t_2 - t)(1-zt) + B(t_1-t)(1-zt) + C(t_1-t)(t_2-t)}{(t_1-t)(t_2-t)(1-zt)}\\
=  &\frac{A[zt^2 - (zt_2 + 1)t + t_2] + B[zt^2 - (zt_1 + 1)t + t_1] + C[t^2 -(t_1+t_2)t+t_1t_2]}{(t_1-t)(t_2-t)(1-zt)}\\
= &\frac{(Az + Bz + C)t^2 + [ - A(zt_2 + 1) - B(zt_1 + 1)-(t_1+t_2)C]t + At_2 + Bt_1 + Ct_1t_2}{(t_1-t)(t_2-t)(1-zt)}.\\
\end{align*}
This yields the following 3 equations:
\begin{align}
    Az + Bz + C &= 0,\label{e4.6}\\
    A(zt_2 + 1) + B(zt_1 + 1) + (t_1+t_2)C &= 0 ,\label{e4.7}\\
    At_2 + Bt_1 + Ct_1t_2 &= \frac{1}{a}.\label{e4.8}
\end{align}
If we multiply Equation \eqref{e4.6} by $t_1+t_2$ and subtract Equation \eqref{e4.7} from it, we get
\begin{equation}\label{e4.9}
    A(t_1z - 1) + B(t_2z-1) = 0.
\end{equation}
If we multiply Equation \eqref{e4.6} by $t_1t_2$ and subtract Equation \eqref{e4.8} from it, we get
\begin{equation}\label{e4.10}
    A(t_1t_2z - t_2) + B(t_1t_2z - t_1) = -\frac{1}{a}.
\end{equation}
Multiplying Equation \eqref{e4.9} by $t_2$ and subtracting Equation \eqref{e4.10} from it yields 
\begin{equation}\label{e4.11}
B(t_2^2z - t_2 - t_1t_2z + t_1) = \frac{1}{a}.
\end{equation}
Observe that 
\begin{align*}
    t_2^2z - t_2 - t_1t_2z + t_1 &= t_2(t_2z - 1) - t_1(t_2z-1)\\
    &= (t_2-t_1)(t_2z-1).
\end{align*}
Then we can rewrite Equation \eqref{e4.11} as
\[
B(t_2-t_1)(t_2z-1) = \frac{1}{a}.
\]
It follows that 
\[
B = \frac{1}{a(t_2-t_1)(t_2z-1)}.
\]
Also, Equation \eqref{e4.9} induces
\[
A = -B\frac{t_2z-1}{t_1z-1} =  -\frac{1}{a(t_2-t_1)(t_2z-1)}\frac{t_2z-1}{t_1z-1} = -\frac{1}{a(t_2-t_1)(t_1z-1)}.
\]
On the other hand, Equation \eqref{e4.6} yields that 
\begin{align*}
    C &= -Az - Bz \\
    &= -\left(-\frac{1}{a(t_2-t_1)(t_1z-1)}\right)z -\left(\frac{1}{a(t_2-t_1)(t_2z-1)}\right)z\\
    &= \frac{(t_2z-1) - (t_1z-1)}{a(t_2-t_1)(t_1z-1)(t_2z-1)}z\\
    &= \frac{z^2}{a(t_1z-1)(t_2z-1)}.
\end{align*}

Hence we have shown that
\begin{align}
    A = &-\frac{1}{a(t_2-t_1)(t_1z-1)}\label{e5.0}\\
    B = &\frac{1}{a(t_2-t_1)(t_2z-1)}\label{e5.1}\\
    C = &\frac{z^2}{a(t_1z-1)(t_2z-1)}.\label{e5.2}
\end{align}
Meanwhile, recall that
\begin{align*}
\sum_{m=0}^\infty P_m(z) t^m = &\frac{1}{a(t - t_1)(t - t_2)(1-tz)} \\
= &\frac{A}{t_1-t} + \frac{B}{t_2-t} + \frac{C}{1-zt}.
\end{align*}
Observe that, for $t<\min(|t_1|, |t_2|, |z|)$, we can apply geometric series as follows:
\begin{align*}
\frac{A}{t_1-t} + \frac{B}{t_2-t} + \frac{C}{1-zt} = &A\frac{1}{t_1}\frac{1}{1-\frac{t}{t_1}} + B\frac{1}{t_2}\frac{1}{1-\frac{t}{t_2}} + C\frac{1}{1-zt} \\
    = &A\sum_{m=0}^\infty \frac{t^m}{t_1^{m+1}} + B\sum_{m=0}^\infty\frac{t^m}{t_2^{m+1}} + C\sum_{m=0}^\infty t^mz^m.
\end{align*}
Moewover, Theorem \ref{t5.1} implies that 
\[
A\sum_{m=0}^\infty \frac{t^m}{t_1^{m+1}} + B\sum_{m=0}^\infty\frac{t^m}{t_2^{m+1}} + C\sum_{m=0}^\infty t^mz^m
\]
is absolutely convergent for $t\in B(\min(|t_1|, |t_2|, |z|);0)$. 
Then Proposition \ref{p5.0} induces that 
\begin{align*}
    A\sum_{m=0}^\infty \frac{t^m}{t_1^{m+1}} + B\sum_{m=0}^\infty\frac{t^m}{t_2^{m+1}} + C\sum_{m=0}^\infty t^mz^m&= \sum_{m=0}^\infty t^m\left(\frac{A}{t_1^{m+1}} + \frac{B}{t_2^{m+1}} + Cz^m\right).
\end{align*}
Consequently, if we denote $f(t) = {\displaystyle\frac{1}{a(t-t_1)(t-t_2)(1-zt)}}$, by Proposition \ref{p5.1}, we may deduce that 
\[
\sum_{m=0}^\infty P_m(z) t^m = \sum_{m=0}^\infty\frac{1}{m!}f^{(m)}(0)t^m = \sum_{m=0}^\infty t^m\left(\frac{A}{t_1^{m+1}} + \frac{B}{t_2^{m+1}} + Cz^m\right).
\]
This implies that 
\[
P_m(z) = \frac{A}{t_1^{m+1}} + \frac{B}{t_2^{m+1}} + Cz^m.
\]
In view of \eqref{e5.0}, \eqref{e5.1}, \eqref{e5.2}, we can observe that 
\begin{align*}
P_m(z)\nonumber = &\frac{A}{t_1^{m+1}} + \frac{B}{t_2^{m+1}} + Cz^m\\
    = &-\frac{1}{a(t_2-t_1)(t_1z-1)}\frac{1}{t_1^{m+1}}\\
    &+ \frac{1}{a(t_2-t_1)(t_2z-1)}\frac{1}{t_2^{m+1}}\\
    &+ \frac{z^2}{a(t_1z-1)(t_2z-1)}z^m\\
    = &\frac{-(t_2z-1)t_2^{m+1} + (t_1z-1)t_1^{m+1} + (t_2-t_1)t_1^{m+1}t_2^{m+1}z^{m+2}}{a(t_2-t_1)(t_1z-1)(t_2z-1)t_1^{m+1}t_2^{m+1}}\nonumber\\
    = &\frac{t_2^{m+1} - t_1^{m+1} + (t_1^{m+2} - t_2^{m+2})z + (t_2-t_1)t_1^{m+1}t_2^{m+1}z^{m+2}}{a(t_2-t_1)(t_1z-1)(t_2z-1)t_1^{m+1}t_2^{m+1}}.\label{e5.4}
\end{align*}
Then, by Proposition \ref{e5.3},
\begin{align*}
    P_m(z) = &\frac{t_2^{m+1} - t_1^{m+1} + (t_1^{m+2} - t_2^{m+2})z + (t_2-t_1)t_1^{m+1}t_2^{m+1}z^{m+2}}{a(t_2-t_1)(t_1z-1)(t_2z-1)t_1^{m+1}t_2^{m+1}}\\
    = &\frac{1}{a(t_2-t_1)t_1^{m+1}t_2^{m+1}}\sum_{k=0}^{m}(t_2^{m+1-k} - t_1^{m+1-k})t_1^kt_2^kz^k.\qedhere
\end{align*}
\end{proof}
Here we found how $P_m(z)$ is defined except at the points $\left\{0,\frac{1}{t_1},\frac{1}{t_2}\right\}$.
We can find $P_m(z)$ for $z=  0,\frac{1}{t_1},\frac{1}{t_2}$ with the following theorem.

\begin{thm}[Corollary 3.8, \cite{John}]\label{t5.2}
Let $G$ be an open connected set.
If $f$ and $g$ are analytic on $G$ then $f\equiv g$ iff $\{z\in G; f(z) = g(z)\}$ has a limit point in $G$.
\end{thm}

Now we will use Theorem \ref{t5.2} to prove that 
\[
P_m(z) = \frac{1}{a(t_2-t_1)t_1^{m+1}t_2^{m+1}}\sum_{k=0}^{m}(t_2^{m+1-k} - t_1^{m+1-k})t_1^kt_2^kz^k,
\]
for all $z\in\mathbb{C}$.


\begin{lem}\label{l5.2}
Let $a,t_1,t_2\in\mathbb{R}$ where $at_1t_2\neq 0$ and $t_1\neq t_2$.
Let  $\{P_m(z)\}_{m\geq 0}$ be a sequence of functions of $z$ generated as follows:
\[
\sum_{m=0}^{\infty}P_m(z) t^m = \frac{1}{a(t-t_1)(t-t_2)(1-zt)}.
\]
Then, for all $z\in\mathbb{C}$,
\[
P_m(z) =  \frac{1}{a(t_2-t_1)t_1^{m+1}t_2^{m+1}}\sum_{k=0}^{m}(t_2^{m+1-k} - t_1^{m+1-k})t_1^kt_2^kz^k.
\]
\end{lem}
\begin{proof}
From Lemma \ref{l5.1}, we obtain that 
\[
P_m(z) = \frac{1}{a(t_2-t_1)t_1^{m+1}t_2^{m+1}}\sum_{k=0}^{m}(t_2^{m+1-k} - t_1^{m+1-k})t_1^kt_2^kz^k,
\]
for all $z\in\mathbb{C}\backslash\left\{0,\frac{1}{t_1},\frac{1}{t_2}\right\}$.
Meanwhile, given that $at_1t_2\neq 0$, we may rewrite 
\[
\frac{1}{a(t-t_1)(t-t_2)(1-zt)} = \frac{1}{(at^2 - (t_1+t_2)t + t_1t_2)(1-zt)}
\]
where $t_1t_2\neq 0$.
Then, by Lemma \ref{l5.0}, $P_m(z)$ is a polynomial and thus entire.
Notice that 
\[
\frac{1}{a(t_2-t_1)t_1^{m+1}t_2^{m+1}}\sum_{k=0}^{m}(t_2^{m+1-k} - t_1^{m+1-k})t_1^kt_2^kz^k
\]
is a polynomial and thus entire. We have that 
\[
P_m(z) = \frac{1}{a(t_2-t_1)t_1^{m+1}t_2^{m+1}}\sum_{k=0}^{m}(t_2^{m+1-k} - t_1^{m+1-k})t_1^kt_2^kz^k,
\]
for all $z\in\mathbb{C}\backslash\left\{0,\frac{1}{t_1},\frac{1}{t_2}\right\}.$
Thus $\mathbb{C}\backslash\left\{0,\frac{1}{t_1},\frac{1}{t_2}\right\}$ is contained in
\[
\left\{z\in\mathbb{C}; P_m(z) = \frac{1}{a(t_2-t_1)t_1^{m+1}t_2^{m+1}}\sum_{k=0}^{m}(t_2^{m+1-k} - t_1^{m+1-k})t_1^kt_2^kz^k\right\}.
\]
Notice that  $\mathbb{C}\backslash\left\{0,\frac{1}{t_1},\frac{1}{t_2}\right\}$ has a limit point, and thus, by Theorem \ref{t5.2}, 
\[
P_m(z) = \frac{1}{a(t_2-t_1)t_1^{m+1}t_2^{m+1}}\sum_{k=0}^{m}(t_2^{m+1-k} - t_1^{m+1-k})t_1^kt_2^kz^k,
\]
for all $z\in\mathbb{C}$.
\end{proof}
Now we have explicit form of $P_m(z)$ and we may deduce the following corollary.

\begin{corollary}\label{cor5.0}
Suppose $ac\ne0$ and $b^2-4ac\neq0$, and $r\neq 0$. Let $\left\{ P_{m}(z)\right\} _{m\geq 0}$
and $\left\{ H_{m}(z)\right\} _{m\geq 0}$ be two sequences
of polynomials defined by their generating functions:
\begin{align*}
\sum_{m=0}^{\infty}P_{m}(z)t^{m} & =\frac{1}{(at^{2}+bt+c)(1-tz)},\\
\sum_{m=0}^\infty H_m(z) t^m &=\frac{1}{\left[a\left(\frac{t}{r}\right)^2+b\left(\frac{t}{r}\right)+c\right]\left[1-tz\right]}.
\end{align*}
Then it holds that 
\[
H_m(z) = \frac{P_m(rz)}{r^m}.
\]
\end{corollary}
\begin{proof}
Let us denote the zeros of $at^2+bt+c$ by $t_1,t_2$.
We may deduce that $rt_1,rt_2$ are the zeros of $\frac{a}{r^2}t^2+\frac{b}{r}t+c$.
Then Lemma \ref{l5.2} induces that 
\[
P_m(z) =\frac{1}{a(t_2-t_1)t_1^{m+1}t_2^{m+1}}\sum_{k=0}^{m}(t_2^{m+1-k} - t_1^{m+1-k})t_1^kt_2^kz^k.
\]
On the other hand, observe that
\[
\frac{1}{\left[a\left(\frac{t}{r}\right)^2+b\left(\frac{t}{r}\right)+c\right]\left[1-\left(\frac{t}{r}\right)\left(rz\right)\right]} = \frac{1}{\left[\frac{a}{r^2}t^2+\frac{b}{r}t+c\right]\left[1-zt\right]}
\]
With Lemma \ref{l5.2}, we can infer that $H_m(z)$ is
\begin{align*}
&\frac{1}{\frac{a}{r^2}(rt_2-rt_1)r^{m+1}t_1^{m+1}r^{m+1}t_2^{m+1}}\sum_{k=0}^{m}({t_2}^{m+1-k} - {t_1}^{m+1-k})r^{m+1-k}{(rt_1)}^k{(rt_2)}^kz^k\\
= &\frac{1}{r^{m}a(t_2-t_1)t_1^{m+1}t_2^{m+1}}\sum_{k=0}^{m}({t_2}^{m+1-k} - {t_1}^{m+1-k}){t_1}^k{t_2}^k(rz)^k\\
= &\frac{P_m(rz)}{r^m}
\end{align*}
This concludes that 
\[
H_m(z) = \frac{P_m(rz)}{r^m}.\qedhere
\]
\end{proof}

With this corollary, we may trace the zeros with the simpler generating function as described in the following corollary.
\begin{corollary}\label{cor5.1}
Suppose $ac\ne0$. Let $\left\{ P_{m}(z)\right\} _{m\geq 0}$
and $\left\{ H_{m}(z)\right\} _{m\geq 0}$ be two sequences
of polynomials defined by their generating functions:
\begin{align*}
\sum_{m=0}^{\infty}P_{m}(z)t^{m} & =\frac{1}{(at^{2}+bt+c)(1-tz)},\\
\sum_{m=0}^\infty H_m(z) t^m &=\frac{1}{\left(\frac{ac}{|ac|}t^2 + \frac{b}{\sqrt{|ac|}}t + 1\right)(1-tz)}.
\end{align*}
Then each zero of $P_m(z)$ is a zero of $H_m(z)$ multiplied by $\left|\frac{\sqrt{|ac|}}{c}\right|$.
\end{corollary}
\begin{proof}
If we consider $t_1, t_2$, the zeros of $at^2 + bt + c$, then $at^2 + bt + c = a(t-t_1)(t- t_2)$.
Given that $ac\neq 0$, we obtain $c\neq 0$.
Thus we may divide both sides by $c$, and it yields the following:
\[
\frac{a}{c}(t- t_1)(t- t_2) = \frac{1}{c}(at^2 + bt + c) = \frac{a}{c}t^2 + \frac{b}{c}t + 1.
\]
This implies that $\frac{a}{c}t^2 + \frac{b}{c}t + 1$ has the same zeros as $at^2 + bt + c$.
Now let $\{G_m(z)\}_{m\geq 0}$ be a sequence of functions of $z$ generated as follows:
\[
\sum_{m=0}^\infty G_m(z) t^m =\frac{1}{\left(\frac{a}{c}t^2 + \frac{b}{c}t + 1\right)(1-tz)}.
\]
By Lemma \ref{l5.2}, we may deduce that
\begin{equation}\label{e5.17}
P_m(z) = \frac{G_m(z)}{c}.
\end{equation}
Recall that $\{H_m(z)\}_{m\geq 0}$ is a sequence of functions of $z$ generated as follows:
\[
\sum_{m=0}^\infty H_m(z) t^m =\frac{1}{\left(\frac{a}{c}\left(\frac{t}{\frac{\sqrt{|ac|}}{c}}\right)^2 + \frac{b}{c}\frac{t}{\frac{\sqrt{|ac|}}{c}} + 1\right)(1-tz)}.
\]
Then Corollary \ref{cor5.0} induces that
\begin{equation}\label{e5.18}
H_m(z) = \frac{G_m\left(\frac{\sqrt{|ac|}}{c}z\right)}{\left(\frac{\sqrt{|ac|}}{c}\right)^m},
\end{equation}
and Equations \eqref{e5.17} and \eqref{e5.18} yield that
\[
P_m(z) = \frac{G_m(z)}{c} = \frac{1}{c}\left(\frac{\sqrt{|ac|}}{c}\right)^mH_m\left(\frac{z}{\frac{\sqrt{|ac|}}{c}}\right)
\]
So, each zero of $P_m(z)$ is a zero of $H_m(z)$ multiplied by $\frac{\sqrt{|ac|}}{c}$.
\end{proof}

\clearpage

\begin{center}
\section{ZERO DISTRIBUTION OF $\{P_m(z)\}_{m\geq 0}$ ABOUT A CIRCLE}
\end{center}
\thispagestyle{empty} 

In this section, we will prove two of our main theorems.
Let us have a look on some examples.

\begin{figure}[htp]
    \centering
    \includegraphics[height=5cm]{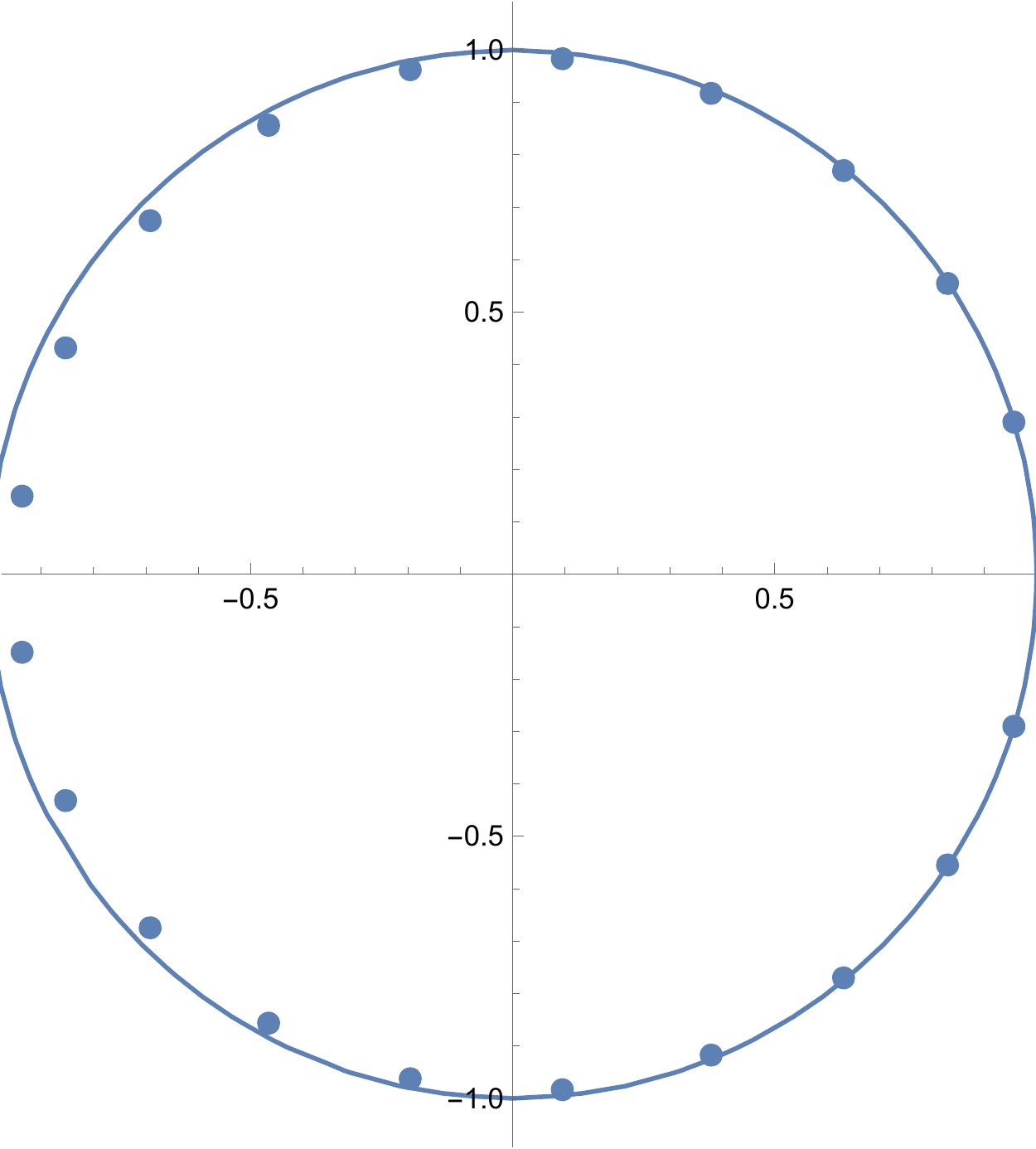}
    \includegraphics[height=5cm]{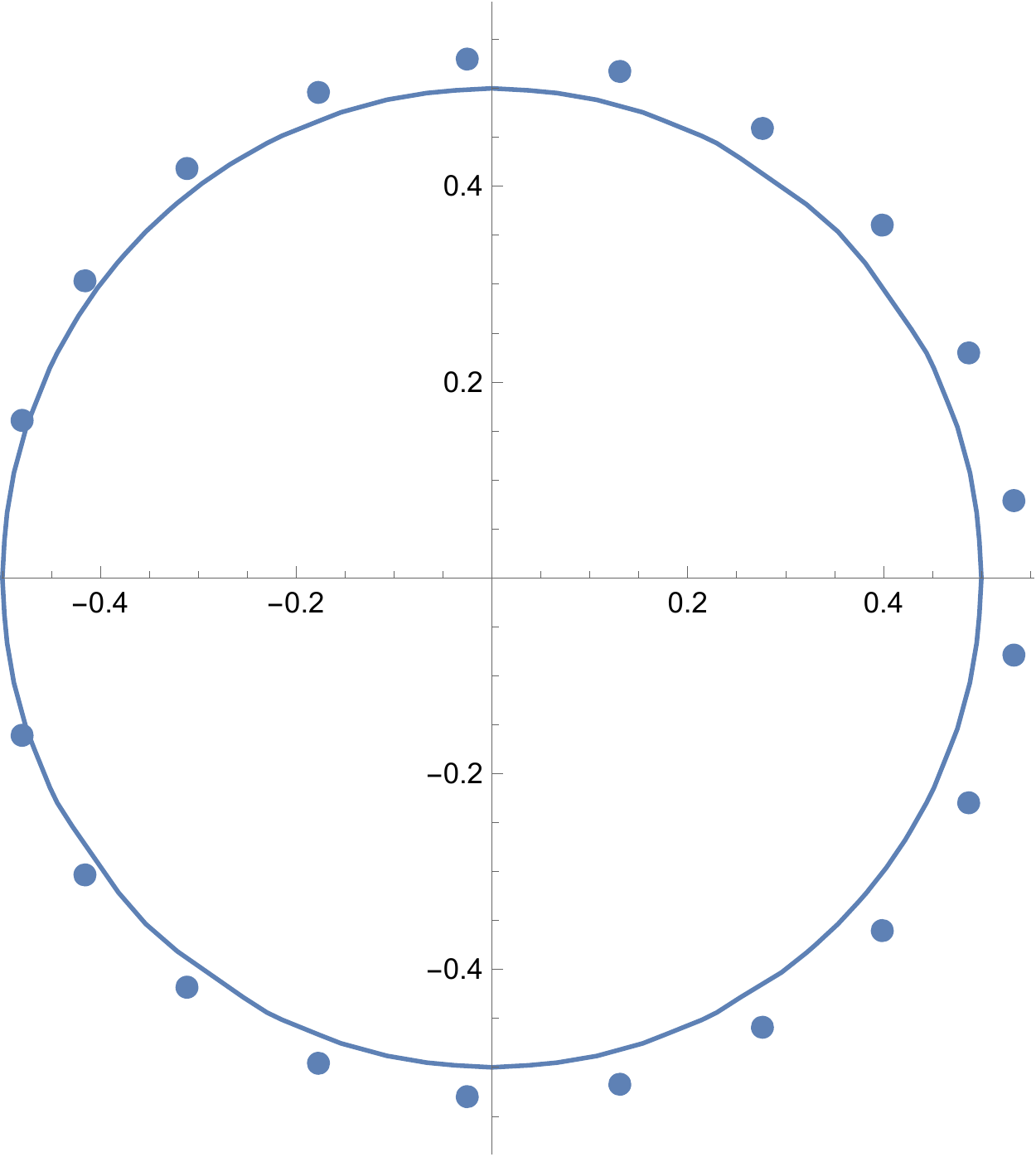}
    \includegraphics[height=5cm]{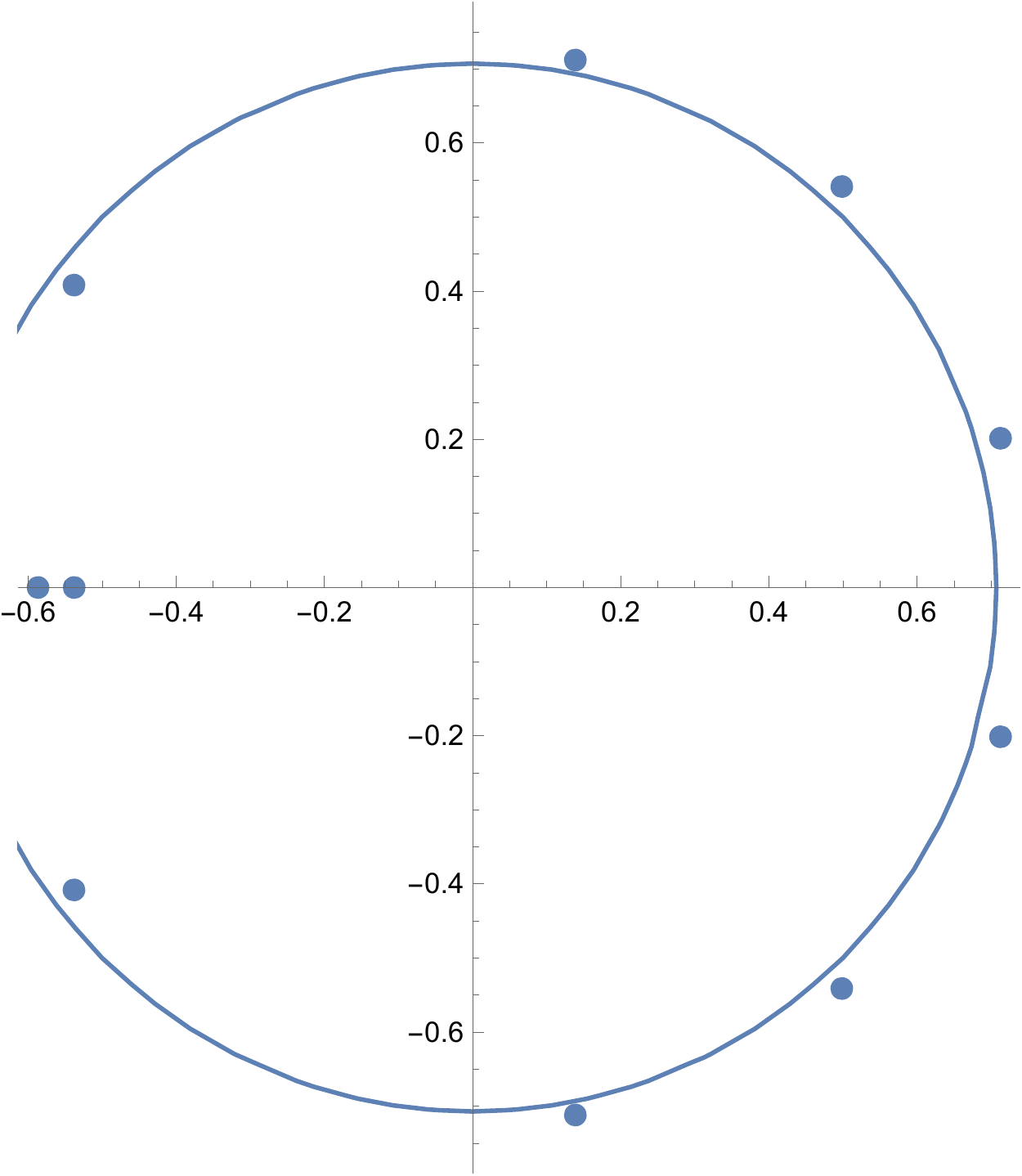}
    \caption{ Zeros inside of the corresponding circles}
    \label{f0.2}
\end{figure}

Figure \ref{f0.2} shows the zeros of polynomials generated by
\[
\frac{1}{(t^2+1t-2)(1-tz)}\text{ and }\frac{1}{(t^2+5t + 6)(1-tz)}\text{ and }\frac{1}{(t^2+t + 2)(1-tz)}
\]
respectively.
The left one has all the zeros inside the circle with radius $1$.
Notice that the roots of $t^2+1t-2 = (t-1)(t+2)$ are $1$ and $-2$.
The middle figure has all the zeros outside the circle with radius $\frac{1}{2}$.
Notice that the roots of $t^2+5t + 6 = (t+2)(t+3)$ are $-2$ and $-3$.
The right figure has some zeros outside and some zeros inside the circle with radius $\frac{1}{\sqrt{2}}$.
Notice that the roots of $t^2+t + 2 = 0$ are $\frac{-1\pm\sqrt{7}i}{2}$ and the moduli are both $\sqrt{2}$.

Let us first prove a theorem about the zeros inside the circle.
\begin{thm}\label{c5.0}
Let $a,b,c\in\mathbb{R}$. 
Let $ac<0$ and $b\neq 0$.
Let  $\{P_m(z)\}_{m\geq 0}$ be a sequence of functions of $z$ generated as follows:
\[
\sum_{m=0}^\infty P_m(z) t^m =\frac{1}{(at^2+bt+c)(1-tz)}.
\]
Then all the zeros of $P_m(z)$ lie on the closed disk centered at 0 with radius $1/|\alpha|$ on the complex plane $\mathbb{C}$ where $\alpha$ is the smallest (in modulus) zero of $at^2+bt+c$.
\end{thm}

\begin{proof}
We will consider $\{H_m(z)\}_{m\geq 0}$, the sequence of functions of $z$ generated as follows;
\[
\sum_{m=0}^\infty H_m(z) t^m =\frac{1}{\left(-t^2 + \frac{b}{\sqrt{-ac}}t + 1\right)(1-tz)}.
\]
Using the Quadratic Formula, we may deduce that the zeros of $-t^2 + \frac{b}{\sqrt{-ac}}t + 1$ are (note that $b^2-4ac>0$ since $ac<0$)
\[
\frac{-\frac{b}{\sqrt{-ac}}\pm\sqrt{-\frac{b^2}{ac} + 4}}{-2} = \frac{1}{\sqrt{-ac}}\frac{-b\pm\sqrt{b^2 - 4ac}}{-2}.
\]
Let us denote those zeros by $t_1,t_2$ where $|t_1|\leq|t_2|$.
Let $m\in\mathbb{N}\cup\{0\}$ be arbitrary.

Now we will prove that all the roots of $H_m(z)$ lie on the closed disk $G$ centered at 0 with radius $1/|t_1|$.
Here the proof of Lemma \ref{l5.1} may be applied to induce that 
\[
H_m(z) = \frac{t_2^{m+1} - t_1^{m+1} + (t_1^{m+2} - t_2^{m+2})z + (t_2-t_1)t_1^{m+1}t_2^{m+1}z^{m+2}}{-(t_2-t_1)(t_1z-1)(t_2z-1)t_1^{m+1}t_2^{m+1}}.
\]
Our goal is to use Rouch\'{e}'s Theorem to show that all the zeros of the numerator $t_2^{m+1} - t_1^{m+1} + (t_1^{m+2} - t_2^{m+2})z + (t_2-t_1)t_1^{m+1}t_2^{m+1}z^{m+2}$ are on the closed disk $|z|\leq 1/|t_1|$.\\

For the rest of the proof, we shall use the following notation:
\begin{align*}
    p(z) &= -(t_2-t_1)t_1^{m+1}t_2^{m+1}z^{m+2}\\
    q(z) &= t_2^{m+1} - t_1^{m+1} + (t_1^{m+2} - t_2^{m+2})z + (t_2-t_1)t_1^{m+1}t_2^{m+1}z^{m+2}.
\end{align*}\\

Let $\epsilon>0$ be arbitrary.
Let $z$ be  arbitrary such that $|z|= 1/|t_1| + \epsilon$.
Notice that there are 2 cases: either $m$ is even or odd.
We will prove that 
\[
|p(z) + q(z)| \leq |t_2|^{m+1}  + \frac{|t_2|^{m+2}}{|t_1|}+ (|t_2|^{m+2}+|t_1||t_2|^{m+1}) \cdot \epsilon
\]
in both cases.
\begin{enumerate}
    \item $m$ is even\\
    By Vieta's Formula, $t_1t_2 = -1$.
    Then one of $t_1,t_2$ is negative.
If $t_1<0$, given $m$ is odd, then $t_1^{m+1}>0$.
Hence it implies that
    \[
        |t_2^{m+1} - t_1^{m+1}|= t_2^{m+1} - t_1^{m+1} = |t_2^{m+1}| - |t_1^{m+1}|.
    \]
    Likewise, if $t_2<0$, given $m$ is even, $t_2^{m+1}>0$, and consequently, we may deduce that
    \[
        |t_2^{m+1} - t_1^{m+1}|= t_2^{m+1} - t_1^{m+1} = |t_2^{m+1}| + |t_1^{m+1}|.
    \]
    In either cases $t_1>0$, or $t_1<0$, we obtain
    \begin{equation}\label{e2.4}
        |t_2^{m+1} - t_1^{m+1}|= |t_2^{m+1}| + |t_1^{m+1}|.
    \end{equation}
  Similarly, if $t_1<0$, then, given $m$ is even, $t_1^{m+2}>0$, and consequently, we may deduce that
    \[
        |t_1^{m+2} - t_2^{m+2}| = t_2^{m+2} - t_1^{m+2} = |t_2^{m+2}|-|t_1^{m+2}|.
    \]
    Likewise, if $t_2<0$, given $m$ is even, $t_2^{m+2}>0$, and consequently, we may deduce that 
    \[
        |t_1^{m+2} - t_2^{m+2}| =t_2^{m+2} - t_1^{m+2} =|t_2^{m+2}|-|t_1^{m+2}|.
    \]
    Considering these conditions, we obtain, in either cases $t_1>0$, or $t_1<0$, that 
    \begin{equation}\label{e2.3}
        |t_1^{m+2} - t_2^{m+2}| = |t_2^{m+2}|-|t_1^{m+2}|.
    \end{equation}

    By sub-additivity,
    \begin{align*}
        |p(z) + q(z)| = &|t_2^{m+1} - t_1^{m+1} + (t_1^{m+2} - t_2^{m+2})z|\\
        \leq &|t_2^{m+1} - t_1^{m+1}| + |t_1^{m+2} - t_2^{m+2}|\cdot|z|\\
        = &|t_2^{m+1} - t_1^{m+1}| +|t_1^{m+2} - t_2^{m+2}|\cdot\left|\frac{1}{|t_1|} +\epsilon\right|.
    \end{align*}
    With sub-additivity and Equations \eqref{e2.3} and \eqref{e2.4}, we continue our deduction
    \begin{align*}
&|t_2^{m+1} - t_1^{m+1}| +|t_1^{m+2} - t_2^{m+2}|\cdot\left|\frac{1}{|t_1|} +\epsilon\right|\\
        = &|t_2^{m+1}| + |t_1^{m+1}| +(|t_2^{m+2}| - |t_1^{m+2}|) \cdot\left(\frac{1}{|t_1|} +\epsilon\right)\\
        = &|t_2|^{m+1}  + |t_1|^{m+1} + \frac{|t_2|^{m+2}}{|t_1|} - |t_1|^{m+1}+ (|t_2|^{m+2} - |t_1|^{m+2}) \cdot \epsilon\\
        = &|t_2|^{m+1}  + \frac{|t_2|^{m+2}}{|t_1|}+ (|t_2|^{m+2} - |t_1|^{m+2}) \cdot \epsilon\\
        \leq &|t_2|^{m+1}  + \frac{|t_2|^{m+2}}{|t_1|}+ (|t_2|^{m+2}+|t_1||t_2|^{m+1}) \cdot \epsilon.
    \end{align*}
    
    \item   $m$ is odd\\
    
    By Vieta's Formula, $t_1t_2 = -1$.
    Then one of $t_1,t_2$ is negative.
    If $t_1<0$, then, given $m$ is odd, $t_1^{m+1}>0$, and consequently, we may deduce that
    \[
        |t_2^{m+1} - t_1^{m+1}|= t_2^{m+1} - t_1^{m+1} = |t_2^{m+1}| - |t_1^{m+1}|.
    \]
    Likewise, if $t_2<0$, given $m$ is odd, $t_2^{m+1}>0$, and consequently, we may deduce that
    \[
        |t_2^{m+1} - t_1^{m+1}|= t_2^{m+1} - t_1^{m+1} = |t_2^{m+1}| - |t_1^{m+1}|.
    \]
    Considering these conditions, we obtain, in either cases $t_1>0$, or $t_1<0$, that 
    \begin{equation}\label{e2.5}
        |t_2^{m+1} - t_1^{m+1}|= |t_2^{m+1}| - |t_1^{m+1}|.
    \end{equation}
Similarly, if $t_1<0$, then, given $m$ is odd, $t_1^{m+2}<0$.
Consequently, 
    \[
        |t_1^{m+2} - t_2^{m+2}| = t_2^{m+2} - t_1^{m+2} = |t_2^{m+2}| + |t_1^{m+2}|.
    \]
    Likewise, if $t_2<0$, given $m$ is odd, $t_2^{m+2}<0$.
Then
    \[
        |t_1^{m+2} - t_2^{m+2}| = -t_2^{m+2} + t_1^{m+2} =|t_2|^{m+2} + |t_1|^{m+2}.
    \]
    In either cases $t_1>0$, or $t_1<0$, we obtain 
    \begin{equation}\label{e2.6}
        |t_1^{m+2} - t_2^{m+2}| = |t_2^{m+2}| + |t_1^{m+2}|.
    \end{equation}
    We may deduce that, by sub-additivity,
    \begin{align*}
        |p(z) + q(z)| = &|t_2^{m+1} - t_1^{m+1} + (t_1^{m+2} - t_2^{m+2})z|\\
        \leq &|t_2^{m+1} - t_1^{m+1}| +|t_1^{m+2} - t_2^{m+2}|\cdot\left|\frac{1}{|t_1|} +\epsilon\right|.
    \end{align*}
    By Equations \eqref{e2.5} and \eqref{e2.6} and sub-additivity imply that
    \begin{align*}
&|t_2^{m+1} - t_1^{m+1}| +|t_1^{m+2} - t_2^{m+2}|\cdot\left|\frac{1}{|t_1|} +\epsilon\right|\\
        \leq &|t_2^{m+1}| - |t_1^{m+1}| +(|t_2^{m+2}|+|t_1^{m+2}|) \cdot\left(\frac{1}{|t_1|} +\epsilon\right)\\
&|t_2^{m+1}| - |t_1^{m+1}| +(|t_2^{m+2}|+|t_1^{m+2}|) \cdot\left(\frac{1}{|t_1|} +\epsilon\right)\\
        = &|t_2|^{m+1}  - |t_1|^{m+1} + \frac{|t_2|^{m+2}}{|t_1|} + |t_1|^{m+1}+ (|t_2|^{m+2}+|t_1|^{m+2}) \cdot \epsilon\\
        = &|t_2|^{m+1}  + \frac{|t_2|^{m+2}}{|t_1|}+ (|t_2|^{m+2}+|t_1|^{m+2}) \cdot \epsilon.
    \end{align*}
    Since $|t_1|\leq |t_2|$ it follows that
    \begin{align*}
&|t_2|^{m+1}  + \frac{|t_2|^{m+2}}{|t_1|}+ (|t_2|^{m+2}+|t_1|^{m+2}) \cdot \epsilon\\
        \leq &|t_2|^{m+1}  + \frac{|t_2|^{m+2}}{|t_1|}+ (|t_2|^{m+2}+|t_1||t_2|^{m+1}) \cdot \epsilon.
    \end{align*}
\end{enumerate}

Here we obtained, in both cases, either $m$ is even or odd, that
\[
|p(z) + q(z)| \leq |t_2|^{m+1}  + \frac{|t_2|^{m+2}}{|t_1|}+ (|t_2|^{m+2}+|t_1||t_2|^{m+1}) \cdot \epsilon.
\]
Since $1\leq \left(m+2\atop{1}\right)$, it follows that
\begin{align*}
    &|t_2|^{m+1}  + \frac{|t_2|^{m+2}}{|t_1|}+ (|t_2|^{m+2}+|t_1||t_2|^{m+1}) \cdot \epsilon \\
    \leq &|t_2|^{m+1}  + \frac{|t_2|^{m+2}}{|t_1|}+ (|t_2|^{m+2}+|t_1||t_2|^{m+1}) \cdot \epsilon\left(m+2\atop{1}\right).
\end{align*}
Simplifying $|t_2|^{m+1}  + \frac{|t_2|^{m+2}}{|t_1|}$ yields that 
\begin{align*}
&|t_2|^{m+1}  + \frac{|t_2|^{m+2}}{|t_1|}+ (|t_2|^{m+2}+|t_1||t_2|^{m+1}) \cdot \epsilon\left(m+2\atop{1}\right)\\
    = &\frac{|t_1||t_2|^{m+1}  +|t_2|^{m+2}}{|t_1|}+ (|t_2|^{m+2}+|t_1||t_2|^{m+1}) \cdot \epsilon\left(m+2\atop{1}\right).
\end{align*}
We may factor out $(|t_2|^{m+2}+|t_1||t_2|^{m+1})$ and get
\begin{align*}
&\frac{|t_1||t_2|^{m+1}  +|t_2|^{m+2}}{|t_1|}+ (|t_2|^{m+2}+|t_1||t_2|^{m+1}) \cdot \epsilon\left(m+2\atop{1}\right)\\
    = &|t_2|^{m+1}(|t_1|  +|t_2|)\left(\frac{1}{|t_1|}+ \epsilon\left(m+2\atop{1}\right)\right).
\end{align*}
Multiplying and dividing by $|t_1|^{m+1}$, we obtain
\begin{align*}
&|t_2|^{m+1}(|t_1|  +|t_2|)\left(\frac{1}{|t_1|}+ \epsilon\left(m+2\atop{1}\right)\right)\\
    = &|t_1|^{m+1}|t_2|^{m+1}(|t_1|  +|t_2|)\left(\frac{1}{|t_1|^{m+2}}+ \frac{1}{|t_1|^{m+1}}\epsilon\left(m+2\atop{1}\right)\right).
\end{align*}
We may infer that the last multiplicand is a part of binomial expansion $\left(\frac{1}{|t_1|}+ \epsilon\right)^{m+2}$.
So it follows that
\begin{align*}
&|t_1|^{m+1}|t_2|^{m+1}(|t_1|  +|t_2|)\left(\frac{1}{|t_1|^{m+2}}+ \frac{1}{|t_1|^{m+1}}\epsilon\left(m+2\atop{1}\right)\right)\\
    < &|t_1|^{m+1}|t_2|^{m+1}(|t_1|  +|t_2|)\left(\frac{1}{|t_1|}+ \epsilon\right)^{m+2}.
\end{align*}
Recall that $t_1t_2 = -1$, thus they have opposite sign, which implies that $|t_1|  +|t_2| = |t_1 - t_2|$.
This continues our deduction as follows
\begin{align*}
&|t_1|^{m+1}|t_2|^{m+1}(|t_1|  +|t_2|)\left(\frac{1}{|t_1|}+ \epsilon\right)^{m+2}\\
    = &|t_1^{m+1}t_2^{m+1}||t_1  -t_2|\left|\frac{1}{|t_1|}+ \epsilon\right|^{m+2}\\
    = &\left|(t_1^{m+1}t_2^{m+1})(t_1  -t_2)\left(\frac{1}{|t_1|}+ \epsilon\right)^{m+2}\right|\\
    = &|p(z)|.
\end{align*}
Then, by Rouch\'{e}'s Theorem (Theorem \ref{rouche}), all the zeros of $q(z)$ lie inside the circle $|z| = \frac{1}{|t_1|} + \epsilon$.
This implies that all the zeros of $H_m(z)$ lie on the closed disk $|z| \leq \frac{1}{|t_1|}$.

Now let $z'$ be an arbitrary root of $P_m(z)$. 
Applying Corollary \ref{cor5.1}, it follows that 
\[
|z'|\leq \left|\frac{\sqrt{|ac|}}{c}\frac{1}{|t_1|}\right| = \frac{2|a|}{-|b| + \sqrt{b^2-4ac}} = \frac{1}{|\alpha|}.
\]
This implies that all the roots of $P_m(z)$ lie on the closed disk centered at 0 with radius $\frac{1}{|\alpha|}$.
\end{proof}
The next theorem propose the zeros outside of the circle
\begin{thm}\label{c5.1}
Let $a,b,c\in\mathbb{R}$. 
Suppose $ac>0$ and $b^2 - 4ac > 0$.
Let  $\{P_m(z)\}_{m\geq 0}$ be a sequence of functions of $z$ generated as follows:
\[
\sum_{m=0}^\infty P_m(z) t^m =\frac{1}{(at^2+bt+c)(1-tz)}.
\]
Then no zeros of $P_m(z)$ lie inside the open ball centered at the origin with radius $1/|\alpha|$ where $\alpha$ is the smallest (in modulus) zero of $at^2+bt+c$.
\end{thm}

\begin{proof}
We will consider $\{H_m(z)\}_{m\geq 0}$, the sequence of functions of $z$ generated as follows:
\[
\sum_{m=0}^\infty H_m(z) t^m =\frac{1}{\left(t^2 + \frac{b}{\sqrt{ac}}t + 1\right)(1-tz)}.
\]
Notice that the zeros of $t^2 + \frac{b}{\sqrt{ac}}t + 1$ are 
\[
\frac{-\frac{b}{\sqrt{ac}}\pm\sqrt{\frac{b^2}{ac} - 4}}{2} = \frac{1}{\sqrt{ac}}\frac{-b\pm\sqrt{b^2 - 4ac}}{2}.
\]
Let us denote those zeros by $t_1,t_2$ where $|t_1|\leq|t_2|$.
Let $m\in\mathbb{N}\cup\{0\}$ be arbitrary.
Our first goal is to prove that $H_m(z)$ has no root on the open ball $G$ centered at 0 with radius $1/|t_1| = |t_2|$.
Suppose, by contrary, that $H_m(z)$ has a root $\beta$ on the open ball $G$.
Given $ac>0$ and $b^2 - 4ac>0$, it induces that $|-b| = \sqrt{b^2} > \sqrt{b^2-4ac}$, and then 
\[
|-b\pm\sqrt{b^2-4ac}| \geq |-b| - |\sqrt{b^2-4ac}| > 0.
\]
This induces that $t_1$ and $t_2$  are not 0, and $b^2-4ac> 0$ implies that $t_1,t_2$ are distinct.
Lemma \ref{l5.2} implies that
\[
H_m(z) =  \frac{1}{(t_2-t_1)t_1^{m+1}t_2^{m+1}}\sum_{k=0}^{m}(t_2^{m+1-k} - t_1^{m+1-k})t_1^kt_2^kz^k.
\]
By Proposition \ref{e5.3},
\begin{align*}
&\frac{t_2^{m+1} - t_1^{m+1} + (t_1^{m+2} - t_2^{m+2})z + (t_2-t_1)t_1^{m+1}t_2^{m+1}z^{m+2}}{(t_2-t_1)(t_1z-1)(t_2z-1)t_1^{m+1}t_2^{m+1}}\\
= &\frac{1}{(t_2-t_1)t_1^{m+1}t_2^{m+1}}\sum_{k=0}^{m}(t_2^{m+1-k} - t_1^{m+1-k})t_1^kt_2^kz^k.
\end{align*}
By Vieta's Formula, we obtain $t_1t_2 = 1$.
Then it follows that
\begin{align*}
H_m(z) =  &\frac{1}{(t_2-t_1)}\sum_{k=0}^{m}(t_2^{m+1-k} - t_1^{m+1-k})z^k \\
= &\frac{t_2^{m+1} - t_1^{m+1} + (t_1^{m+2} - t_2^{m+2})z + (t_2-t_1)z^{m+2}}{(t_2-t_1)(t_1z-1)(t_2z-1)}
\end{align*}
for $z\in\mathbb{C}\backslash\left\{0,\frac{1}{t_1},\frac{1}{t_2}\right\}$.

Now let us denote 
\begin{align*}
p(z) = &t_2^{m+1} - t_1^{m+1} + (t_1^{m+2} - t_2^{m+2})z + (t_2-t_1)z^{m+2},\\
q(z) = &-(t_1^{m+2} - t_2^{m+2})z - (t_2-t_1)z^{m+2}.
\end{align*}
We will prove that $q(z)$ has only one zero in the open ball centered at the origin with radius $|t_2|$, and then we will apply Rouch\'{e}'s Theorem (c.f. Theorem \ref{rouche}) for these two polynomials to find a contradiction.

Notice that 
\[
q(z) = -(t_1^{m+2} - t_2^{m+2})z - (t_2-t_1)z^{m+2} = -z[(t_1^{m+2} - t_2^{m+2}) + (t_2-t_1)z^{m+1}].
\]
Thus $z = 0$ is one root of $q(z) = 0$, and the other roots are the solutions of 
\[
-(t_1^{m+2} - t_2^{m+2}) - (t_2-t_1)z^{m+1} = 0.
\]

We are to prove that $|t_1|\neq |t_2|$.
Suppose, by contradiction, that $|t_1| = |t_2|$.
Since $t_1,t_2$ are the zeros of $t^2 + \frac{b}{\sqrt{ac}}t + 1$, by Vieta's Formula, $t_1t_2 = 1$.
Then $|t_1| = |t_2|$ implies that 
\[
1 = |t_1t_2| = |t_1||t_2| = |t_1|^2.
\]
This induces that $|t_1| = 1$.
We may infer that
\[
t_1 = t_2 = 1\text{ or }t_1 = t_2 = -1.
\]
Notice that, by Vieta's Formula, 
\[
-\frac{b}{\sqrt{ac}} = t_1 + t_2 = \pm 2.
\]
Then 
\[
\frac{b^2}{ac} = 4,
\]
which yields that
\[
b^2 - 4ac = 0.
\]
This contradicts against our condition $b^2 - 4ac > 0$.
Thus our contrary supposition is false.
So it holds that $|t_1|\neq |t_2|$.
Given that $|t_1|\leq|t_2|$, it follows that 
\begin{equation}\label{e5.22}
|t_1|< |t_2|.
\end{equation}
Since $t_1t_2 = 1$ and $|t_1|<|t_2|$, it follows that $|t_1|^{m+1} <1< |t_2|^{m+1}$. 
It induces that 
\[
|t_2|^{m+2} - |t_1|^{m+2} > |t_2|^{m+2} - |t_2|^{m+1}|t_1|.
\]
This yields that
\begin{equation}\label{e5.20}
\frac{|t_2|^{m+2} - |t_1|^{m+2}}{|t_2| - |t_1|} > |t_2|^{m+1}.
\end{equation}
Meanwhile, for a zero $z_0$ of the equation $-(t_1^{m+2} - t_2^{m+2}) - (t_2-t_1)z^{m+1} = 0$, we can deduce that 
\begin{equation}\label{e5.21}
|z_0|^{m+1} = \left|\frac{t_2^{m+2} - t_1^{m+2}}{t_2-t_1}\right|.
\end{equation}
In view of Inequality \eqref{e5.22}, $|t_1| < |t_2|$, with that $t_1$ and $t_2$ have the same sign, it induces that $|t_2^{m+2} - t_1^{m+2}| = |t_2|^{m+2} - |t_1|^{m+2}$ and that $|t_2-t_1| = |t_2| - |t_1|$.
It implies that 
\[
\left|\frac{t_2^{m+2} - t_1^{m+2}}{t_2-t_1}\right| = \frac{|t_2|^{m+2} - |t_1|^{m+2}}{|t_2| - |t_1|}.
\]
Then Inequality \eqref{e5.20} and Equation \eqref{e5.21} yield that 
\[
|z_0|^{m+1} > |t_2|^{m+1}.
\]
This yields that $|z_0| > |t_2|$, which implies that all the zeros of 
\[
-(t_1^{m+2} - t_2^{m+2}) - (t_2-t_1)z^{m+1}
\]
are outside of the circle centered at the origin with radius $|t_2|$.
This concludes that 
\[
q(z) = -(t_1^{m+2} - t_2^{m+2})z - (t_2-t_1)z^{m+2}
\]
has only one root, $z_0 = 0$, in the open ball centered at the origin with radius $|t_2|$.

Let us denote
\begin{align*}
\sigma := &(m+2)|t_2|^{m+1} - (|t_2|^{m+1} + |t_2|^{m}|t_1| +   \cdots + |t_2||t_1|^{m} + |t_1|^{m+1}),\\
\zeta:=  &\frac{\sigma}{\sum_{k=2}^{m+2}\left({m+2}\atop{k}\right)|t_2|^{m+2-k}}.
\end{align*}
Inequality \eqref{e5.22} implies that $|t_2|^{m+2} > |t_2|^{m+1-k}|t_1|^k$ for any $k = 0,1,2,\dots,m+1$.
This implies that $\sigma>0$.
Then since $\sum_{k=2}^{m+2}\left({m+2}\atop{k}\right)|t_2|^{m+2-k} > 0$, it follows that $\zeta > 0$.
Let $\epsilon>0$ be arbitrary such that 
\begin{equation}\label{e5.5}
\epsilon<\min\left\{1,|t_2|,|t_2| - |t_1|, |t_2| - |\beta|,\zeta\right\}.
\end{equation}
Observe that
\begin{align*}
     \sigma - \epsilon\sum_{k=2}^{m+2}(-\epsilon)^{k-2}\left({m+2}\atop{k}\right)|t_2|^{m+2-k} \geq \sigma - \left|\epsilon\sum_{k=2}^{m+2}(-\epsilon)^{k-2}\left({m+2}\atop{k}\right)|t_2|^{m+2-k}\right|.
\end{align*}
By sub-additivity,
\begin{align*}
\sigma - \left|\epsilon\sum_{k=2}^{m+2}(-\epsilon)^{k-2}\left({m+2}\atop{k}\right)|t_2|^{m+2-k}\right| \geq \sigma - \epsilon\sum_{k=2}^{m+2}\epsilon^{k-2}\left({m+2}\atop{k}\right)|t_2|^{m+2-k}.
\end{align*}
Since $\epsilon < 1$, it follows that 
\begin{align*}
\sigma - \epsilon\sum_{k=2}^{m+2}\epsilon^{k-2}\left({m+2}\atop{k}\right)|t_2|^{m+2-k}    \geq \sigma - \epsilon\sum_{k=2}^{m+2}\left({m+2}\atop{k}\right)|t_2|^{m+2-k}.
\end{align*}
In view of $\epsilon<\zeta = \frac{\sigma}{\sum_{k=2}^{m+2}\left({m+2}\atop{k}\right)|t_2|^{m+2-k}}$, we can infer that
\begin{align*}
&\sigma - \epsilon\sum_{k=2}^{m+2}\left({m+2}\atop{k}\right)|t_2|^{m+2-k}\\
    > &\sigma -\frac{\sigma}{\sum_{k=2}^{m+2}\left({m+2}\atop{k}\right)|t_2|^{m+2-k}}\sum_{k=2}^{m+2}\left({m+2}\atop{k}\right)|t_2|^{m+2-k}\\
    = &\sigma - \sigma\\
    = &0.
\end{align*}
Consequently, we obtain
\begin{equation}\label{e5.23}
\sigma - \epsilon\sum_{k=2}^{m+2}(-\epsilon)^{k-2}\left({m+2}\atop{k}\right)|t_2|^{m+2-k} > 0.
\end{equation}
Let $\gamma$ be a circle with radius $|t_2| - \epsilon$ centered at $0$.
Let $z_0\in\gamma$ be arbitrary.
Notice that
\[
    |q(z_0)| = |-(t_1^{m+2} - t_2^{m+2})z_0 - (t_2-t_1)z_0^{m+2}|.
\]
By sub-additivity,
\begin{align*}
&|-(t_1^{m+2} - t_2^{m+2})z_0 - (t_2-t_1)z_0^{m+2}|    \geq &|-(t_1^{m+2} - t_2^{m+2})z_0| - |(t_2-t_1)z_0^{m+2}|.
\end{align*}
In view of $|z_0| = |t_2| - \epsilon$, it follows that
\begin{align*}
&|-(t_1^{m+2} - t_2^{m+2})z_0| - |(t_2-t_1)z_0^{m+2}|\\
    = &(|t_2|^{m+2} - |t_1|^{m+2})(|t_2| - \epsilon) - (|t_2| - |t_1|)(|t_2| - \epsilon)^{m+2}\\
    = &|t_2|^{m+3} - |t_1|^{m+1} - \epsilon|t_2|^{m+2} + \epsilon|t_1|^{m+2} \\
&- (|t_2| - |t_1|)\sum_{k=0}^{m+2}(-\epsilon)^{k}\left({m+2}\atop{k}\right)|t_2|^{m+2-k}.
\end{align*}
By pulling out the first two addends in the sums and then expanding those two addends, we get
\begin{align}
&|t_2|^{m+3} - |t_1|^{m+1} - \epsilon|t_2|^{m+2} + \epsilon|t_1|^{m+2} \\
&- (|t_2| - |t_1|)\sum_{k=0}^{m+2}(-\epsilon)^{k}\left({m+2}\atop{k}\right)|t_2|^{m+2-k}\nonumber\\
    = &|t_2|^{m+3} - |t_1|^{m+1} - \epsilon|t_2|^{m+2} + \epsilon|t_1|^{m+2}\nonumber\\
    &- (|t_2| - |t_1|)\left(|t_2|^{m+2} - \epsilon(m+2)|t_2|^{m+1} + \sum_{k=2}^{m+2}(-\epsilon)^{k}\left({m+2}\atop{k}\right)|t_2|^{m+2-k}\right)\nonumber\\
\begin{split}\label{e5.24}
   = &|t_2|^{m+3} - |t_1|^{m+1} - \epsilon(|t_2|^{m+2} -|t_1|^{m+2}) -|t_2|^{m+3} + |t_2|^{m+1} \\
&+ \epsilon(|t_2| - |t_1|)(m+2)|t_2|^{m+1} - (|t_2| - |t_1|)  \sum_{k=2}^{m+2}(-\epsilon)^{k}\left({m+2}\atop{k}\right)|t_2|^{m+2-k}.
\end{split}
\end{align}
After canceling $|t_2|^{m+3}$, we may factor as $\epsilon(|t_2|^{m+2} -|t_1|^{m+2}) = \epsilon(|t_2| - |t_1|)(|t_2|^{m+1} + \cdots + |t_1|^{m+1})$, then we can observe that Expression \eqref{e5.24} is equal to
\begin{align*}
    &|t_2|^{m+1} - |t_1|^{m+1} - \epsilon(|t_2| - |t_1|)(|t_2|^{m+1} + \cdots + |t_1|^{m+1}) \\
    &  + \epsilon(|t_2| - |t_1|)(m+2)|t_2|^{m+1} - \epsilon^2(|t_2| - |t_1|)  \sum_{k=2}^{m+2}(-\epsilon)^{k-2}\left({m+2}\atop{k}\right)|t_2|^{m+2-k}.
\end{align*}
If we factor out $\epsilon(|t_2| - |t_1|)$ from the last three addends, then above expression is equal to 
\begin{align*}
    &|t_2|^{m+1} - |t_1|^{m+1} + \epsilon(|t_2| - |t_1|)\left(\sigma- \epsilon\sum_{k=2}^{m+2}(-\epsilon)^{k-2}\left({m+2}\atop{k}\right)|t_2|^{m+2-k}\right).
\end{align*}
Then, in view of Inequality \eqref{e5.23}, it follows that
\begin{align*}
&|t_2|^{m+1} - |t_1|^{m+1} + \epsilon(|t_2| - |t_1|)\left(\sigma- \epsilon\sum_{k=2}^{m+2}(-\epsilon)^{k-2}\left({m+2}\atop{k}\right)|t_2|^{m+2-k}\right)\\
    > &|t_2|^{m+1} - |t_1|^{m+1}\\
    = &|p(z_0) + q(z_0)|.
\end{align*}
By Rouch\'{e}'s Theorem (c.f. Theorem \ref{rouche}), $p(z)$ and $q(z)$ have the same number of zeros in the open ball centered at the origin with radius $|t_2|-\epsilon$.
Meanwhile, recall from condition \eqref{e5.5} that 
\[
\epsilon<\min\left\{1,|t_2|,|t_2| - |t_1|, |t_2| - |\beta|,\zeta\right\}.
\]
It induces that $\epsilon<|t_2|-|t_1|$ and $\epsilon<|t_2|-\beta$.
It follows that $|t_2| - \epsilon > |t_1|$, and that $|t_2| - \epsilon > |\beta|$.
Hence the open ball contains $t_1$ and $\beta$.
Notice that $\beta$ is a zero of $H_m(z)$ and consequently $\beta$ is a zero of $p(z)$.
Meanwhile, if we substitute $t_1$ for $z$ in $H_m(z)$, we may infer that $t_1$ is not a zero of $H_m(z)$.
But $t_1$ is a zero of $p(z)$.
So $\beta$ and $t_1$ are distinct zeros of $p(z)$.
Then $p(z)$ has more roots than $q(z)$ has in the open ball, which contradicts against the deduction made by Rouch\'{e}'s Theorem.
So our contrary supposition is false.
This concludes that $H_m(z)$ has no root on the open ball centered at 0 with radius $1/|t_1|$.

Now let $z'$ be arbitrary root of $P_m(z)$. 
Applying Corollary \ref{cor5.1}, we obtain 
\[
|z'|\geq \frac{\sqrt{|ac|}}{|c|}\frac{1}{|t_1|} = \frac{2|a|}{|b| - \sqrt{b^2-4ac}} = \frac{1}{|\alpha|}.
\]
This concludes that no root of $P_m(z)$ lie on the open ball centered at 0 with radius $\frac{1}{|\alpha|}$.
\end{proof}

\thispagestyle{empty} 

The authors of \cite{Kakeya} introduce the history of studies about zeros of polynomials.
Gauss and Cauchy first started to study this area. 
The research of this area developed throughout the history and introduced in \cite{Kakeya}. We will use the following theorem in this section.
\begin{thm}[Enestr\"{o}m-Kakeya's Theorem, Theorem 1.4, \cite{Kakeya}]\label{kakeya}
If $p(z) = \sum_{k=0}^ma_kz^k$ is a polynomial of degree $m$ with real and positive coefficients, then all the zeros of $p(z)$ lie in the annulus $r\leq |z|\leq R$ where 
\[
r = \min_{0\leq k\leq m-1}\frac{a_k}{a_{k+1}}\text{ and }R = \max_{0\leq k\leq m-1}\frac{a_k}{a_{k+1}}.
\]
\end{thm}
We use this theorem to prove the following main theorem of this section.

\begin{thm}\label{c5.2}
Let $a,b,c\in\mathbb{R}\backslash\{0\}$. 
Suppose $b^2 - 4ac > 0$ and $ac>0$.
Let  $\{P_m(z)\}_{m\geq 0}$ be a sequence of functions of $z$ generated as follows:
\[
\sum_{m=0}^\infty P_m(z) t^m =\frac{1}{(at^2+bt+c)(1-tz)}.
\]
Then no zeros of $P_m(z)$ lie in the closed disk centered at the origin with radius $1/|\alpha|$, where $\alpha$ is a zero of $at^2+bt+c$ with greater modulus.
\end{thm}
\begin{proof}
We will consider $\{H_m(z)\}_{m\geq 0}$, the sequence of functions of $z$ generated as follows:
\[
\sum_{m=0}^\infty H_m(z) t^m =\frac{1}{\left(\frac{ac}{|ac|}t^2 + \frac{b}{\sqrt{|ac|}}t + 1\right)(1-tz)}.
\]
Notice that the zeros of $\frac{ac}{|ac|}t^2 + \frac{b}{\sqrt{|ac|}}t + 1$ are 
\[
\frac{-\frac{b}{\sqrt{|ac|}}\pm\sqrt{\frac{b^2}{|ac|} - 4\frac{ac}{|ac|}}}{2\frac{ac}{|ac|}} = \frac{1}{\sqrt{|ac|}}\frac{-b\pm\sqrt{b^2 - 4ac}}{2\frac{ac}{|ac|}}.
\]
Let us denote those zeros by $t_1,t_2$ where $|t_1|\leq|t_2|$.

We are to prove that $|t_1|\neq |t_2|$.
Suppose, by contradiction, that $|t_1| = |t_2|$.
Since $t_1,t_2$ are the zeros of $t^2 + \frac{b}{\sqrt{ac}}t + 1$, by Vieta's Formula, $t_1t_2 = 1$.
Then $|t_1| = |t_2|$ implies that 
\[
1 = |t_1t_2| = |t_1||t_2| = |t_1|^2.
\]
This induces that $|t_1| = 1$.
We may infer that
\[
t_1 = t_2 = 1\text{ or }t_1 = t_2 = -1.
\]
Notice that, by Vieta's Formula, 
\[
-\frac{b}{\sqrt{ac}} = t_1 + t_2 = \pm 2.
\]
Then 
\[
\frac{b^2}{ac} = 4,
\]
which yields that
\[
b^2 - 4ac = 0.
\]
This contradicts our condition $b^2 - 4ac > 0$.
Thus our supposition is false.
So it holds that $|t_1|\neq |t_2|$.
Given that $|t_1|\leq|t_2|$, it follows that 
\begin{equation}\label{e7.1}
|t_1| < |t_2|.
\end{equation}
Recall the explicit form of $H_m(z)$, by Lemma \ref{l5.2},
\[
H_m(z) =  \frac{1}{\frac{ac}{|ac|}(t_2-t_1)t_1^{m+1}t_2^{m+1}}\sum_{k=0}^{m}(t_2^{m+1-k} - t_1^{m+1-k})t_1^kt_2^kz^k, m\geq 0.
\]
By Vieta's Formula, $t_1t_2 = \frac{|ac|}{ac} = 1$.
Then 
\[
H_m(z) =  \frac{1}{t_2-t_1}\sum_{k=0}^{m}(t_2^{m+1-k} - t_1^{m+1-k})z^k, m\geq 0.
\]
Let us choose $m\geq 1$ arbitrarily.
We are to prove that $\frac{|t_2^{m+1-k} - t_1^{m+1-k}|}{|t_2^{m-k} - t_1^{m-k}|}>\frac{1}{|t_1|}$ for $k = 1,2,\dots, m-1$.
Let $0\leq k \leq m-1$ be arbitrary.
With $t_1t_2 = 1$, we may deduce that
\begin{align*}
    |t_2^{m+1-k} - t_1^{m+1-k}|\cdot|t_1| = &|t_2^{m+1-k}t_1 - t_1^{m+2-k}|\\
    = &|t_2^{m-k} - t_1^{m+2-k}|.
\end{align*}
Since $t_1,t_2$ have the same sign, it can be deduced that 
\begin{align*}
   |t_2^{m-k} - t_1^{m+2-k}| = &|t_2|^{m-k} - |t_1|^{m+2-k}.
\end{align*}
Inequality \eqref{e7.1}, in view of $t_1t_2 = 1$, induces that $|t_1|<1$.
It follows that 
\begin{align*}
    |t_2|^{m-k} - |t_1|^{m+2-k} >|t_2|^{m-k} - |t_1|^{m-k} = |t_2^{m-k} - t_1^{m-k}|.
\end{align*}
Consequently we obtain 
\[
|t_2^{m+1-k} - t_1^{m+1-k}|\cdot|t_1| > |t_2^{m-k} - t_1^{m-k}|,
\]
which yields that 
\begin{equation}\label{e7.4}
\frac{|t_2^{m+1-k} - t_1^{m+1-k}|}{|t_2^{m-k} - t_1^{m-k}|}>\frac{1}{|t_1|}.
\end{equation}

With this Inequality \eqref{e7.4}, we will prove that all the zeros of $H_m(z)$ are outside of the circle $|z| = \frac{1}{|t_1|}$.
Denote, for each $0\leq k\leq m$, the coefficient of $z^k$ by $a_k$ as follows 
\[
a_k :=  \frac{t_2^{m+1-k} - t_1^{m+1-k}}{t_2-t_1}.
\]
If $t_2>0$, then, under Inequality \eqref{e7.1}, 
\[
t_2-t_1 = |t_2| - t_1 \geq |t_2| - |t_1| > 0
\]
and
\[
t_2^{m+1-k} - t_1^{m+1-k} = |t_2|^{m+1-k} - t_1^{m+1-k} \geq |t_2|^{m+1-k} - |t_1|^{m+1-k}> 0.
\]
Hence $a_k$ is positive.
So $H_m(z)$ is a polynomial with all positive coefficients.
Thus 
\[
\frac{a_k}{a_{k+1}} = \frac{|a_k|}{|a_{k+1}|} = \frac{|t_2^{m+1-k} - t_1^{m+1-k}|}{|t_2^{m-k} - t_1^{m-k}|}.
\]
Here Inequality \eqref{e7.4} implies that 
\[
\min_{0\leq k\leq m-1}\frac{a_k}{a_{k+1}} =\min_{0\leq k\leq m-1}\frac{|t_2^{m+1-k} - t_1^{m+1-k}|}{|t_2^{m-k} - t_1^{m-k}|} > \frac{1}{|t_1|}.
\]
Then Theorem \ref{kakeya} induces that all the zeros of $H_m(z)$ are outside of the circle $|z| = \frac{1}{|t_1|}$.\\

If $t_2<0$, since one of $m-k$ and $m+1-k$ is odd the the other is even, one of $t_2^{m-k}$ and $t_2^{m + 1-k}$ is positive and the other is negative.
We may infer that 
\begin{equation}\label{e7.3}
\frac{a_k}{a_{k+1}} =  \frac{t_2^{m+1-k} - t_1^{m+1-k}}{t_2^{m-k} - t_1^{m-k}} < 0.
\end{equation}
Then the sign of $a_k$ alternate, since otherwise, there would exist $a_{k'},a_{k'+1}$ with the same sign, and it would yield that
\[
\frac{a_{k'}}{a_{k'+1}} >0
\]
which contradicts against Inequality \eqref{e7.3}.
Notice that
\[
a_0 = \frac{t_2^{m+1} - t_1^{m+1}}{t_2 - t_1}.
\]
Notice that 
\begin{equation}\label{e7.2}
t_2 - t_1 = -|t_2| - t_1 \leq -|t_2| + |t_1| < 0.
\end{equation}
If $m$ is odd, then $m+1$ is even.
It follows that
\[
t_2^{m+1} - t_1^{m+1} = |t_2|^{m+1} - |t_1|^{m+1} > 0.
\]
Considering Inequality \eqref{e7.2}, it holds that $a_0 < 0$.
We may infer that
\[
-H_m(-z) 
\]
have all positive coefficients.
Denote the coefficient of $z^k$ of $-H_m(-z)$ by $b_k$.
Notice that
\[
b_k =  \frac{t_2^{m+1-k} - t_1^{m+1-k}}{t_2-t_1}(-1)^{k+1}.
\]
We may deduce that
\[
\frac{b_k}{b_{k+1}} = \frac{|b_k|}{|b_{k+1}|} = \frac{|a_k|}{|a_{k+1}|} = \frac{|t_2^{m+1-k} - t_1^{m+1-k}|}{|t_2^{m-k} - t_1^{m-k}|}.
\]
Here Inequality \eqref{e7.4} implies that 
\[
\min_{0\leq k\leq m-1}\frac{b_k}{b_{k+1}} =\min_{0\leq k\leq m-1}\frac{|t_2^{m+1-k} - t_1^{m+1-k}|}{|t_2^{m-k} - t_1^{m-k}|} > \frac{1}{|t_1|}.
\]
Then Theorem \ref{kakeya} induces that all the zeros of $-H_m(-z)$ are outside of the circle $|z| = \frac{1}{|t_1|}$.
We may infer that all the zeros of $H_m(z)$ are outside of the circle $|z| = \frac{1}{|t_1|}$.

If $m$ is even, then $m+1$ is odd.
It follows that
\[
t_2^{m+1} - t_1^{m+1} = -|t_2|^{m+1} + |t_1|^{m+1} < 0.
\]
Considering Inequality \eqref{e7.2}, it holds that $a_0 > 0$
We may infer that
\[
H_m(-z)
\]
have all positive coefficients.
Denote the coefficient of $z^k$ of $H_m(-z)$ by $c_k$.
Notice that 
\[
c_k =  \frac{t_2^{m+1-k} - t_1^{m+1-k}}{t_2-t_1}(-1)^k.
\]
We may deduce that
\[
\frac{c_k}{c_{k+1}} = \frac{|c_k|}{|c_{k+1}|} =\frac{|a_k|}{|a_{k+1}|} = \frac{|t_2^{m+1-k} - t_1^{m+1-k}|}{|t_2^{m-k} - t_1^{m-k}|}.
\]
Here Inequality \eqref{e7.4} implies that 
\[
\min_{0\leq k\leq m-1}\frac{c_k}{c_{k+1}} =\min_{0\leq k\leq m-1}\frac{|t_2^{m+1-k} - t_1^{m+1-k}|}{|t_2^{m-k} - t_1^{m-k}|} > \frac{1}{|t_1|}.
\]
Then Enestr\"{o}m- Kakeya's Theorem \ref{kakeya} induces that all the zeros of $H_m(-z)$ are outside of the circle $|z| = \frac{1}{|t_1|}$.
We may infer that all the zeros of $H_m(z)$ are outside of the circle $|z| = \frac{1}{|t_1|}$.

This concluds that, in all cases, the zeros of $H_m(z)$ are outside of the circle $|z| = \frac{1}{|t_1|}$.
\end{proof}
\clearpage
\section{APPLICATION OF LIMIT DISTRIBUTION OF ZEROS}

In the previous section, we showed that the zeros of $P_m(z)$ are related to the circle centered at the origin with radius $1/|\alpha|$ where $\alpha$ is one of the roots of $at^2+bt+c$ with smallest modulus.
In this section we will study the `convergence' of zeros.
\begin{figure}[htp]
    \centering
    \includegraphics[height=4cm]{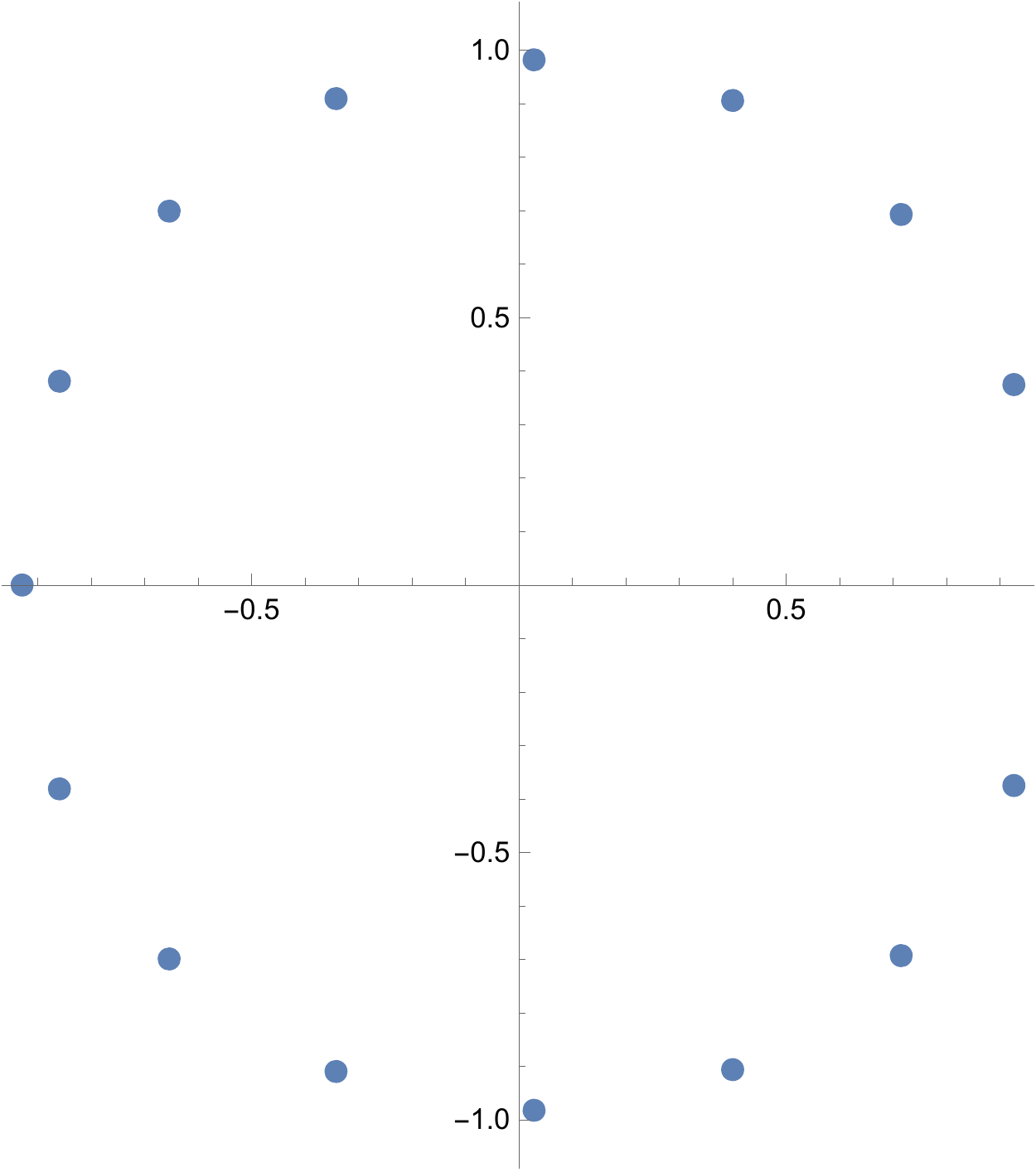}
    \includegraphics[height=4cm]{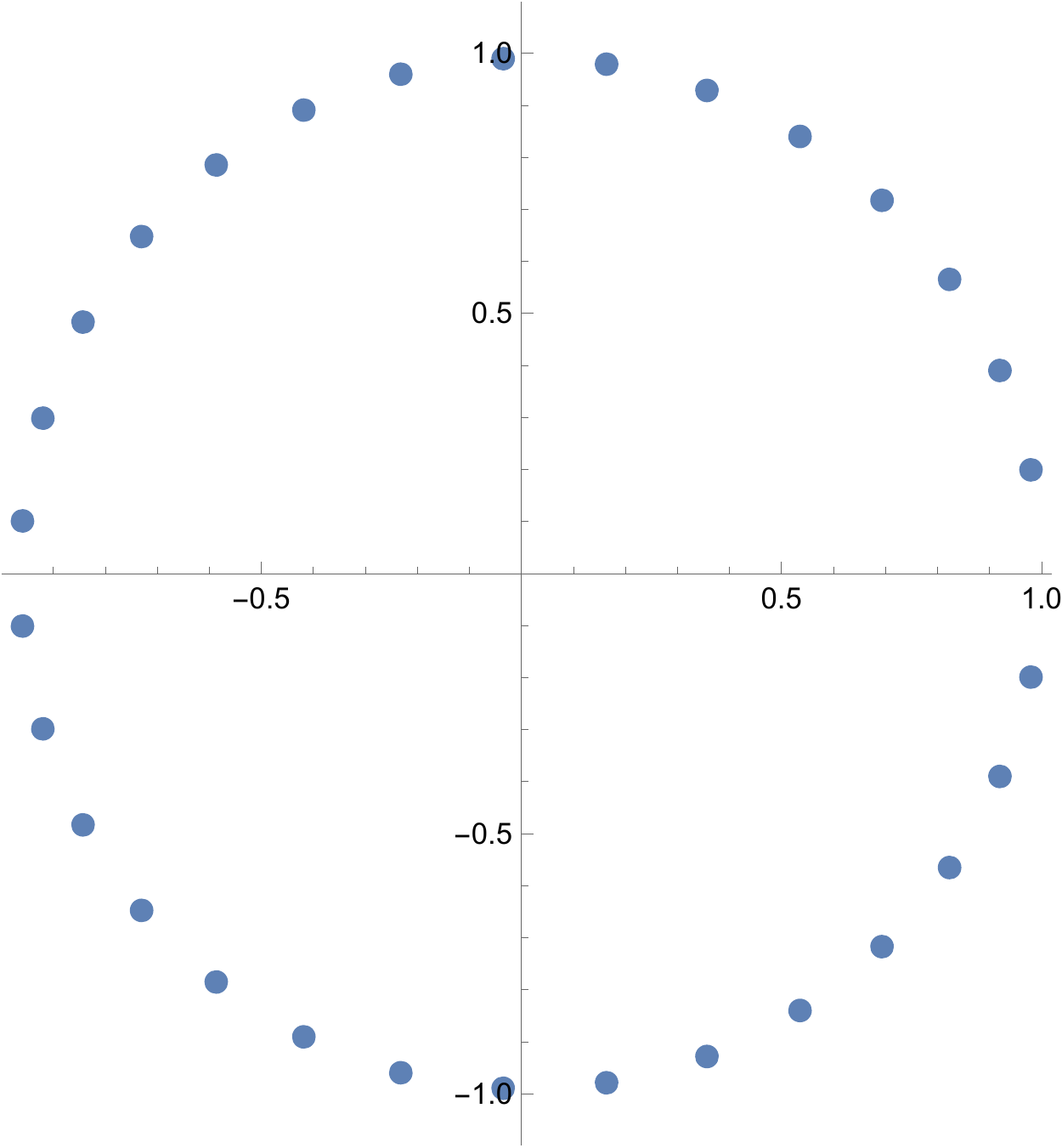}
    \includegraphics[height=4cm]{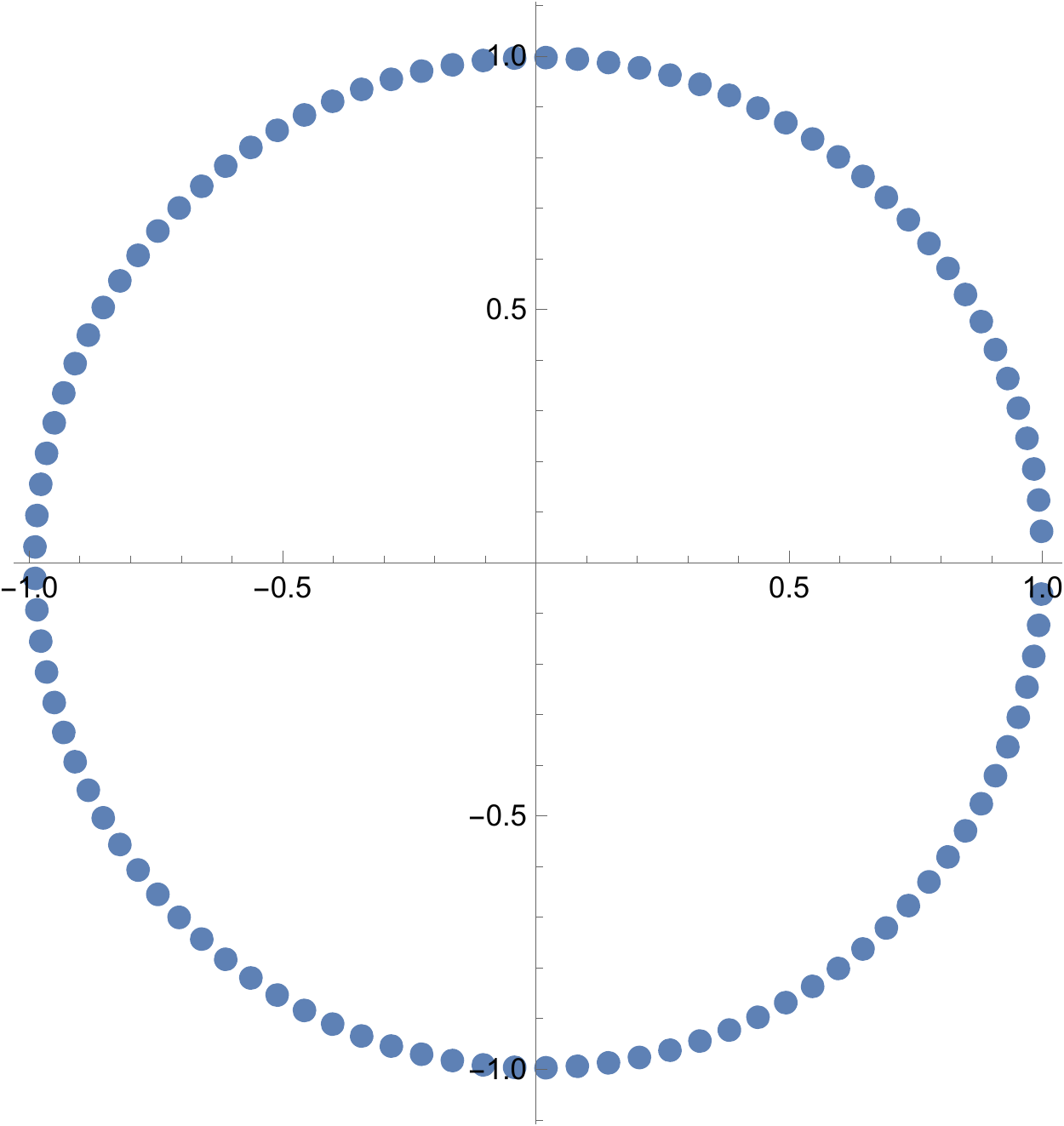}
    \caption{ Zeros of $P_m(z)$ for $m = 15, 30, 100$}
    \label{convergence}
\end{figure}
Figure \ref{convergence} shows the zeros of $P_{15}(z), P_{30}(z),P_{100}(z)$ generated as
\[
\sum_{m=0}^{\infty}P_m(z)t^m = \frac{1}{(t^2 + t-2)(1-zt)}.
\]
The zeros are becoming denser and shaping a circle $|z| = 1$.
The main theorem of this section proves that the zeros converges to the circle.
We now rigorously define what `convergence' is.
First, let us define the set of zeros.
\begin{definition}[(1.5), \cite{Sokal}]
Suppose $\{f_m\}_{m\geq 0}$ analytic in a region (an open and connected set) $D$.
We denote, for each $m\geq 0$, $\mathcal{Z}(f_m)$ the set of zeros of $f_m$ in $D$.
\end{definition}

There are two kind of `limit' of zeros.
\begin{definition}[(1.7), \cite{Sokal}]
We denote $\limsup \mathcal{Z}(f_m)$ the set of $z^*\in D$ such that 
\[
\forall r>0, \forall N\in\mathbb{N}, \exists m>N; B(r,z^*)\cap \mathcal{Z}(f_m)\neq\emptyset.
\]
\end{definition}

\begin{definition}[(1.6), \cite{Sokal}]
We denote $\liminf \mathcal{Z}(f_m)$ to be the set of $z^*\in D$ such that 
\[
\forall r>0, \exists N_r\in\mathbb{N};\forall m>N_r, B(r,z^*)\cap \mathcal{Z}(f_m)\neq\emptyset.
\]
\end{definition}
The two limits above have the following property.
\begin{prop}\label{p6.1}
Suppose $\{f_m\}_{m\geq 0}$ analytic in a region $D$.
Then
\[
\liminf \mathcal{Z}(f_m)\subseteq \limsup \mathcal{Z}(f_m)
\]
\end{prop}
\begin{proof}
Let $z^*\in \liminf \mathcal{Z}(f_m)$ be arbitrary.
Let $r>0$ and $N\in\mathbb{N}$ be arbitrary.
Under the definition of $\liminf \mathcal{Z}(f_m)$, there exists $N_r\in\mathbb{N}$ such that, for any $m>N_r$, it holds that $B(r,z^*)\cap \mathcal{Z}(f_m)\neq\emptyset$.
Then, for $n_0 = \max(N_r,N)$ we obtain $B(r,z^*)\cap \mathcal{Z}(f_{n_0})\neq\emptyset$.
This implies that $z^*\in\limsup \mathcal{Z}(f_m)$.
Thus, 
\[
\liminf \mathcal{Z}(f_m)\subseteq \limsup \mathcal{Z}(f_m).\qedhere
\]
\end{proof}
When those two limits are the same we define a new limit.
\begin{definition}
If $\liminf \mathcal{Z}(f_m) = \limsup \mathcal{Z}(f_m)$, then we define 
\[
\lim \mathcal{Z}(f_m) := \liminf \mathcal{Z}(f_m) = \limsup \mathcal{Z}(f_m).
\]
\end{definition}
The main theorem in this section will find the set $\lim \mathcal{Z}(P_m)$.
To confirm the existence of $\lim \mathcal{Z}(P_m)$, and to find the elements of $\lim \mathcal{Z}(P_m)$ the following theorem provides a convenient tool for us.
\begin{thm}\label{t6.0}[Theorem 1.5, \cite{Sokal}]
Suppose $\{f_m\}_{m\geq 0}$ is a sequence of functions analytic in an open connected set (region) $D$.
Suppose that there exists $n\in\mathbb{N}$ and $\alpha_k(z),\beta_k(z)$, $1\le k \le n$, analytic in $D$ such that 
\[
f_m(z) = \sum_{k=1}^n\alpha_k(z)\beta_k(z)^m \text{, for all } m\in \mathbb{N},
\]
and there does not exist $i\neq j$ such that $\beta_i(z) = w\beta_j(z)$ for any $z\in D$, for some $w$ with $|w| = 1$.
Then $\liminf \mathcal{Z}(f_m) = \limsup \mathcal{Z}(f_m)$.
Furthermore, $z^*\in\lim \mathcal{Z}(f_m)$ if and only if one of the following holds
\begin{itemize}
    \item $\exists i\neq j$ such that $|\beta_i(z^*)| = |\beta_j(z^*)|\geq |\beta_k(z^*)|$, for all $k = 1,2,\dots,n$, and $\alpha_i(z^*), \alpha_j(z^*)\neq 0$
    \item
    $\exists i$ such that $|\beta_i(z^*)| > |\beta_k(z^*)|$, for all $k = 1,2,\dots,n$, $k\neq i$, and $\alpha_i(z^*)= 0$
\end{itemize}
\end{thm}
Before get into our main theorem, let us study an example of this.
\begin{example}
Let us define, for each $m\in\mathbb{N}$,
\[
f_m(z) := (z-1)^m + z^m + (z+1)^m.
\]
Then 
\[
\lim \mathcal{Z}(f_m) = \{z\in\mathbb{C}: \Re(z) = 0\}.
\]
\end{example}
\begin{figure}[htp]
    \centering
    \includegraphics[height=6cm]{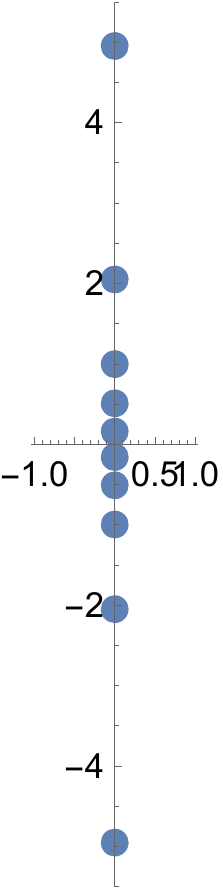}
    \includegraphics[height=6cm]{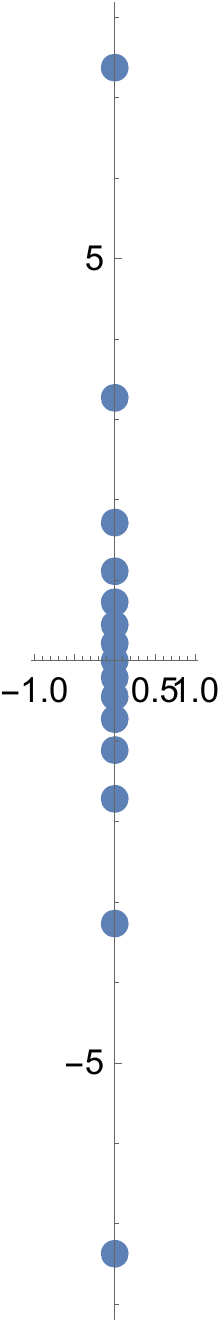}
    \includegraphics[height=6cm]{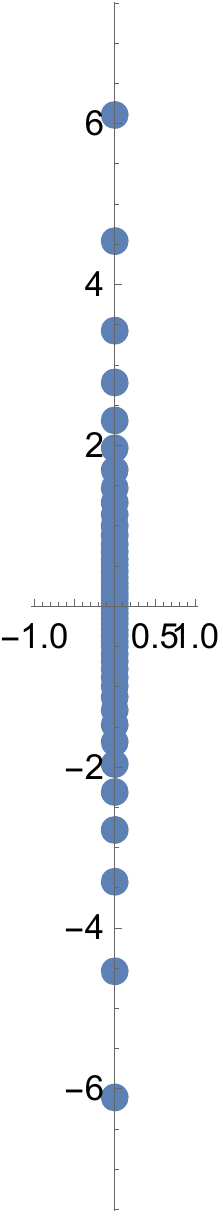}
    \caption{ Zeros of $f_m(z)$ for $m = 10, 15, 50$}
    \label{limiting}
\end{figure}
The points in the Figure \ref{limiting} become denser as $m$ increase.

Notice that 
\[
f_m(z) = \sum_{k=1}^3\alpha_k(z)\beta_k(z)^m \text{, for all } m\in \mathbb{N}
\]
where 
\begin{align*}
    \alpha_1(z) = &1 &\beta_1(z) = &z-1\\
    \alpha_2(z) = &1 &\beta_2(z) = &z\\
    \alpha_3(z) = &1 &\beta_3(z) = &z+1.
\end{align*}
For any $w\in\mathbb{C}$ with $|w| = 1$, we can observe that $|\beta_1(-1)| = 2$ and $|w\beta_2(-1)| = 1$.
Notice that, for any $w\in\mathbb{C}$ with $|w| = 1$, we can observe that $|\beta_2(-1)| = 1$ and $|w\beta_3(-1)| = 0$.
Likewise, for any $w\in\mathbb{C}$ with $|w| = 1$, we can observe that $|\beta_3(-1)| = 0$ and $|w\beta_1(-1)| = 2$.
So, there does not exist $i\neq j$ such that $\beta_i(z) = w\beta_j(z)$ for any $z\in D$, for some $w$ with $|w| = 1$.
This implies, by Theorem \ref{t6.0}, that $\liminf \mathcal{Z}(f_m) = \limsup \mathcal{Z}(f_m)$.
Thus $\lim \mathcal{Z}(f_m)$ exists.\\

Now let us find $\lim \mathcal{Z}(f_m)$.
Let $z^*\in \{z\in\mathbb{C}: \Re(z) = 0\}$ be arbitrary.
Then $z^* = yi$ for some $y\in\mathbb{R}$.
Observe that
\[
    |\beta_1(z^*)| = |yi - 1|  = \sqrt{y^2  + 1}    = |yi + 1|    = |\beta_3(z^*)|
\]
and that 
\[
|\beta_2(z^*)| = |yi| = |y| < \sqrt{y^2 + 1} = |\beta_1(z^*)| = |\beta_3(z^*)|.
\]
Then, by Theorem \ref{t6.0}, $z^*\in \lim \mathcal{Z}(f_m)$.
This implies that 
\[
\{z\in\mathbb{C}: \Re(z) = 0\}\subseteq\lim \mathcal{Z}(f_m).
\]

Let $z^*\in \{z\in\mathbb{C}: \Re(z) = 0\}^c$ be arbitrary.
Then $z^* = x + yi$ for some $x,y\in\mathbb{R}$ with $x\neq 0$.
There are 2 cases: $x < 0$ or $x > 0$.
If $x<0$, notice that 
\[
|x - 1| = -x + 1 > -x = |x|,
\]
and that
\[
|x - 1| = -x + 1 > x + 1,
\]
and that
\[
|x-1| = -x + 1 > -x-1.
\]
Consequently
\[
|x-1| > |x+1|.
\]
This implies that 
\begin{align*}
    \beta_1(z^*) = &|x-1 + yi|\\
    = &\sqrt{|x-1|^2 + y^2}\\
    > &\sqrt{|x|^2 + y^2}\\
    = &|x  + yi|\\
    = &|\beta_2(z^*)|
\end{align*}
and that
\begin{align*}
    \beta_1(z^*) = &|x-1 + yi|\\
    = &\sqrt{|x-1|^2 + y^2}\\
    > &\sqrt{|x + 1|^2 + y^2}\\
    = &|x+1  + yi|\\
    = &|\beta_3(z^*)|.
\end{align*}
Since $\alpha_1(z^*) = 1\neq 0$, Theorem \ref{t6.0} induces that $z^*\in (\lim \mathcal{Z}(f_m))^c$.

If $x>0$, notice that 
\[
|x + 1| = x + 1 > x = |x|.
\]
Similarly, observe that
\[
|x + 1| = x + 1 > x - 1,
\]
and that
\[
|x+1| = x + 1 > -x+1,
\]
and consequently
\[
|x+1| > |x-1|.
\]
This implies that 
\begin{align*}
    \beta_3(z^*) = &|x+1 + yi|\\
    = &\sqrt{|x+1|^2 + y^2}\\
    > &\sqrt{|x|^2 + y^2}\\
    = &|x  + yi|\\
    = &|\beta_2(z^*)|
\end{align*}
and that
\begin{align*}
    \beta_3(z^*) = &|x+1 + yi|\\
    = &\sqrt{|x+1|^2 + y^2}\\
    > &\sqrt{|x - 1|^2 + y^2}\\
    = &|x-1  + yi|\\
    = &|\beta_1(z^*)|.
\end{align*}
Since $\alpha_3(z^*) = 1\neq 0$, by Theorem \ref{t6.0}, induces that $z^*\in (\lim \mathcal{Z}(f_m))^c$.

Thus for both cases $x > 0$ and $x < 0$, we obtain $z^*\in(\lim \mathcal{Z}(f_m))^c$.
This implies that $ \{z\in\mathbb{C}: \Re(z) = 0\}^c\subseteq (\lim \mathcal{Z}(f_m))^c$.
It follows that 
\[
\{z\in\mathbb{C}: \Re(z) = 0\}\supseteq \lim \mathcal{Z}(f_m).
\]
Recall that $\{z\in\mathbb{C}: \Re(z) = 0\}\subseteq\lim \mathcal{Z}(f_m)$.
So we may conclude that
\[
\{z\in\mathbb{C}: \Re(z) = 0\} = \lim \mathcal{Z}(f_m).
\]

When we work on our main theorem, we will show that the set of points from the circle is $\lim \mathcal{Z}(P_m)$.
We will show this except for one limit point of the circle.
To show that the limit point is included in $\lim \mathcal{Z}(P_m)$, we need the following proposition.
\begin{prop}\label{p6.0}
Suppose $\{f_m\}_{m\geq 0}$ analytic in a region  $D$.
Then $\liminf \mathcal{Z}(f_m)$ is closed.
\end{prop}
\begin{proof}
Let $x'$ be an arbitrary limit point of $\liminf \mathcal{Z}(f_m)$.
Let $r>0$ be arbitrary.
Given $x'$ is a limit point of $\liminf \mathcal{Z}(f_m)$, it follows that 
\[
B\left(\frac{r}{2},x'\right)\cap\liminf \mathcal{Z}(f_m)\neq \emptyset.
\]
Then there exists $x^*\in B\left(\frac{r}{2},x'\right)\cap\liminf \mathcal{Z}(f_m)$.
Given $x^*\in\liminf \mathcal{Z}(f_m)$, there exists $N\in\mathbb{N}$ such that, for any $n>N$,  $B\left(\frac{r}{2},x^*\right)\cap \mathcal{Z}(f_n)\neq\emptyset$.
Hence, for arbitrary $m>N$, there exists $x_m\in B\left(\frac{r}{2},x^*\right)\cap \mathcal{Z}(f_m)$.
It induces that $|x_m-x^*|<\frac{r}{2}$.
Recall that $x^*\in B\left(\frac{r}{2},x'\right)\cap\liminf \mathcal{Z}(f_m)$, which implies that $|x^*-x'|<\frac{r}{2}$.
Then we may deduce that 
\begin{align*}
    |x'-x_m| = &|x'-x^*+x^*-x_m|\\
    \leq &|x'-x^*| + |x^*-x_m|\\
    < &\frac{r}{2} + \frac{r}{2}\\
    = &r
\end{align*}
So $x_m\in B(r,x')$.
Recall that $x_m\in \mathcal{Z}(f_m)$.
Then $x_m\in B\left(r,x'\right)\cap \mathcal{Z}(f_m)$.
Thus $B\left(r,x'\right)\cap \mathcal{Z}(f_m)\neq\emptyset$.
Since we chose $m$ arbitrarily, it holds that 
\[
\forall m>N, B\left(r,x'\right)\cap \mathcal{Z}(f_m)\neq\emptyset.
\]
Here, for arbitrary $r>0$, we found $N\in\mathbb{N}$ such that $\forall m>N, B\left(r,x'\right)\cap \mathcal{Z}(f_m)\neq\emptyset$.
That is, $x'\in \liminf \mathcal{Z}(f_m)$.
\end{proof}
\begin{lem}\label{l6.0}
Let $a,b,c\in\mathbb{R}\backslash\{0\}$ where $b^2 - 4ac > 0$.
Suppose  $\{P_m(z)\}_{m\geq 0}$ is a sequence of functions of $z$ generated as follows:
\[
\sum_{m=0}^\infty P_m(z) t^m =\frac{1}{(at^2+bt+c)(1-tz)}.
\]
Then there exists $N\in\mathbb{N}$ such that 
\[
\mathcal{Z}(P_m)\cap B\left(\frac{1}{2}\left(\frac{1}{|\alpha|} + \frac{1}{|\beta|}\right),0\right) = \emptyset,
\]
for all $m>N$, where $\alpha,\beta$ are the zeros of $at^2+bt+c$ with $|\alpha|\leq|\beta|$.
\end{lem}
\begin{proof}
There are two cases: $ac>0$ or $ac<0$.
Let us first prove the case $ac>0$.
Choose $z_0\in B\left(\frac{1}{2}\left(\frac{1}{|\alpha|} + \frac{1}{|\beta|}\right),0\right)$ be arbitrary.
Then 
\[
|z_0|< \frac{1}{2}\left(\frac{1}{|\alpha|} + \frac{1}{|\beta|}\right).
\]
In view of $|\alpha|\leq |\beta|$, we may infer that 
\[
\frac{1}{|\alpha|} + \frac{1}{|\beta|} \leq \frac{2}{|\alpha|}.
\]
It follows that 
\[
\frac{1}{2}\left(\frac{1}{|\alpha|} + \frac{1}{|\beta|}\right) \leq \frac{1}{|\alpha|}.
\]
Consequently, we obtain
\[
|z_0| < \frac{1}{2}\left(\frac{1}{|\alpha|} + \frac{1}{|\beta|}\right) \leq \frac{1}{|\alpha|}.
\]
Theorem \ref{c5.2} implies that all the zeros of $P_m(z)$ are outside of the circle $|z| = \frac{1}{|\alpha|}$, for all $m=1,2,\dots$.
Hence $z_0$ is not a zero of $P_m(z)$, for all $m=1,2,\dots$.
This implies that 
\[
z_0\notin \bigcup_{m = 0}^{\infty}\mathcal{Z}(P_m).
\]
It concludes that 
\[
\mathcal{Z}(P_m)\cap B\left(\frac{1}{2}\left(\frac{1}{|\alpha|} + \frac{1}{|\beta|}\right),0\right) = \emptyset,
\]
for all $m\geq 0$.

Now let us prove the case $ac<0$.
We will consider $\{H_m(z)\}_{m\geq 0}$, the sequence of functions of $z$ generated as follows:
\[
\sum_{m=0}^\infty H_m(z) t^m =\frac{1}{\left(\frac{ac}{|ac|}t^2 + \frac{b}{\sqrt{|ac|}}t + 1\right)(1-tz)} = \frac{1}{\left(-t^2 + \frac{b}{\sqrt{-ac}}t + 1\right)(1-tz)}.
\]
Let us denote those zeros of $-t^2 + \frac{b}{\sqrt{-ac}}t + 1$ by $t_1,t_2$ where $|t_1|\leq|t_2|$.

We are to prove that $|t_1|\neq |t_2|$.
Suppose, by contradiction, that $|t_1| = |t_2|$.
Since $t_1,t_2$ are the zeros of $\frac{ac}{|ac|}t^2 + \frac{b}{\sqrt{|ac|}}t + 1$, by Vieta's Formula, $t_1t_2 = -1$.
Then $|t_1| = |t_2|$ implies that 
\[
1 = |t_1t_2| = |t_1||t_2| = |t_1|^2.
\]
This induces that $|t_1| = |t_2| = 1$.
For the case $ac < 0$, we may infer that 
\[
\left\{\begin{array}{cc} t_1 = &1\\ t_2 = &-1\end{array}\right. \text{ or } \left\{\begin{array}{cc} t_1 = &-1\\ t_2 = &1\end{array}\right.
\]
Notice that, by Vieta's Formula, 
\[
-\frac{b}{\sqrt{ac}} = t_1 + t_2 = 0.
\]
This implies that $b = 0$.
It contradicts against our condition $b \neq 0$.
Thus our contrary supposition is false.
So it holds that $|t_1|\neq |t_2|$.
Given that $|t_1|\leq|t_2|$, it follows that 
\[
|t_1| < |t_2|.
\]
With Proposition \ref{e5.3} and Lemma \ref{l5.2}, we may infer that, for $z\neq\left\{0,\frac{1}{t_1},\frac{1}{t_2}\right\}$,
\[
H_m(z) = \frac{t_2^{m+1} - t_1^{m+1} + (t_1^{m+2} - t_2^{m+2})z + (t_2-t_1)t_1^{m+1}t_2^{m+1}z^{m+2}}{-(t_2-t_1)(t_1z-1)(t_2z-1)t_1^{m+1}t_2^{m+1}}, m\geq 0.
\]
Vieta's Formula induces $t_1t_2 = -1$, which yields that 
\[
H_m(z) = \frac{t_2^{m+1} - t_1^{m+1} + (t_1^{m+2} - t_2^{m+2})z + (t_2-t_1)(-1)^{m+1}z^{m+2}}{-(t_2-t_1)(t_1z-1)(t_2z-1)(-1)^{m+1}}, m\geq 0.
\]
Our plan is to apply Rouch\'{e}'s Theorem to prove that there exists $N\in\mathbb{N}$ such that the numerator of $H_m(z)$ has only one zero inside the circle $|z| = \frac{|t_1| + |t_2|}{2}$ for all $m > N$.
Define
\begin{align*}
p(z) = &t_2^{m+1} - t_1^{m+1} + (t_1^{m+2} - t_2^{m+2})z + (t_2-t_1)(-1)^{m+1}z^{m+2}\\
q(z) = &-t_2^{m+1} + t_1^{m+1} - (t_1^{m+2} - t_2^{m+2})z.
\end{align*}
Notice that the zero of $q(z)$ is
\[
\frac{-t_2^{m+1} + t_1^{m+1}}{t_1^{m+2} - t_2^{m+2}}.
\]
The triangle inequality implies that 
\begin{equation}\label{e6.1}
\left|\frac{-t_2^{m+1} + t_1^{m+1}}{t_1^{m+2} - t_2^{m+2}}\right| \leq \frac{|t_2|^{m+1} + |t_1|^{m+1}}{|t_2|^{m+2} - |t_1|^{m+2}}.
\end{equation}
Since $|t_1t_2| = |-1| = 1$ and $|t_1|< |t_2|$, we may infer that $|t_1|  < 1$ and $|t_2| > 1$.
From $|t_2| > 1$, it follows that
\begin{equation}\label{e6.2}
|t_2|^2 - 1 > 0.
\end{equation}
From $|t_1|<1$, it follows the existence of $N_1\in\mathbb{N}$ such that 
\[
3|t_1|^{m+1} + |t_1|^{m+3} < |t_2|^2 - 1,
\]
for all $m>N_1$.
Since $|t_2| > 1$, the inequality above implies that, for any $m>N_1$,
\[
3|t_1|^{m+1} + |t_1|^{m+3} < |t_2|^{m+1}(|t_2|^2 - 1)=|t_2|^{m+3} - |t_2|^{m+1}.
\]
Adding both sides by $2|t_2|^{m+1} - |t_1|^{m+1} - |t_1|^{m+3}$, we may infer that
\begin{align*}
3|t_1|^{m+1} + |t_1|^{m+3} &<  |t_2|^{m+3} - |t_2|^{m+1}\\
&\Updownarrow\\
2|t_2|^{m+1} + 2|t_1|^{m+1} &< |t_2|^{m+3} - |t_1|^{m+1} + |t_2|^{m+1} - |t_1|^{m+3}\\
&\Updownarrow\\
2(|t_2|^{m+1} + |t_1|^{m+1}) &< \left[|t_2|^{m+2} - |t_1|^{m+2}\right]\cdot\left[|t_2| + |t_1|\right]\\
&\Updownarrow\\
\frac{|t_2|^{m+1} + |t_1|^{m+1}}{|t_2|^{m+2} - |t_1|^{m+2}} &< \frac{|t_2| + |t_1|}{2}.
\end{align*}
It follows from \eqref{e6.1} that $q(z)$ has its only one zero inside the circle $|z| = \frac{|t_2| + |t_1|}{2}$, for all $m>N_1$.

Since $|t_1| < |t_2|$, we may infer that $\frac{|t_1| + |t_2|}{2} < |t_2|$.
Then 
\begin{equation}\label{e6.3}
\frac{|t_1| + |t_2|}{2} = r|t_2|
\end{equation}
for some $0<r<1$.
Recall from \eqref{e6.2} that $|t_2|^2 - 1 > 0$.
Since $0<r<1$, there exists $N_2\in\mathbb{N}$ such that 
\[
r^{m+1} < \frac{|t_2|^2 - 1}{4|t_2| + 4|t_1|}\frac{1}{|t_2|},
\]
for all $m>N_2$.
Observe that
\begin{equation}\label{e6.4}
\frac{|t_2|^2 - 1}{2} = \frac{|t_2|^2 + 1 - 2}{2} = \frac{|t_2|^2 + |t_1|\cdot|t_2|}{2} - 1 = \frac{|t_2| + |t_1|}{2}|t_2| - 1.
\end{equation}
Then for any $m>N_2$,
\begin{align*}
r^{m+2} &< \frac{|t_2|^2 - 1}{4|t_2| + 4|t_1|}\frac{1}{|t_2|}\\
&\Updownarrow\\
r^{m+2} &< \frac{1}{2(|t_2| + |t_1|)}\left[\frac{|t_2|^2 - 1}{2}\right]\frac{1}{|t_2|}\\
&\Updownarrow\\
(|t_2| + |t_1|)r^{m+2} &< \frac{1}{2}\left[\left(\frac{|t_2| + |t_1|}{2}\right)|t_2| - 1\right]\frac{1}{|t_2|}\\
\end{align*}
Multiplying both sides by $|t_2|^{m+2}$, we obtain, for any $m>N_2$,
\[
(|t_2| + |t_1|)r^{m+2}|t_2|^{m+2} < \frac{1}{2}\left[\left(\frac{|t_2| + |t_1|}{2}\right)|t_2|^{m+2} - |t_2|^{m+1}\right].
\]
Equation \eqref{e6.3} induces that, for any $m>N_2$,
\begin{equation}\label{e6.5}
(|t_2| + |t_1|)\left(\frac{|t_1| + |t_2|}{2}\right)^{m+2} < \frac{1}{2}\left[\left(\frac{|t_2| + |t_1|}{2}\right)|t_2|^{m+2} - |t_2|^{m+1}\right].
\end{equation}
Since $|t_1| < 1 < |t_2|$, there exists $N_3\in\mathbb{N}$ such that 
\[
|t_1|^{m+2}\left(\frac{|t_2| + |t_1|}{2}\right) + |t_1|^{m+1} < \frac{|t_2|^2 - 1}{4},
\]
for all $m>N_3$.
By Equation \eqref{e6.4},
\begin{equation}\label{e6.6}
|t_1|^{m+2}\left(\frac{|t_2| + |t_1|}{2}\right) + |t_1|^{m+1} < \frac{|t_2|^2 - 1}{4}= \frac{1}{2}\left[\left(\frac{|t_2| + |t_1|}{2}\right)|t_2| - 1\right],
\end{equation}
for all $m>N_3$.
Define 
\[
N = \max(N_1,N_2,N_3).
\]
Let $m>N$, be arbitrary.
Inequality \eqref{e6.6} yields that 
\begin{align*}
&\left(\frac{|t_2| + |t_1|}{2}\right)|t_2|^{m+2} - |t_2|^{m+1} - \left[|t_1|^{m+2}\left(\frac{|t_2| + |t_1|}{2}\right) + |t_1|^{m+1}\right] \\
> &\left(\frac{|t_2| + |t_1|}{2}\right)|t_2|^{m+2} - |t_2|^{m+1} - \frac{1}{2}\left[\left(\frac{|t_2| + |t_1|}{2}\right)|t_2| - 1\right]\\
> &\left(\frac{|t_2| + |t_1|}{2}\right)|t_2|^{m+2} - |t_2|^{m+1} - \frac{|t_2|^{m+1}}{2}\left[\left(\frac{|t_2| + |t_1|}{2}\right)|t_2| - 1\right]\\
= &\frac{1}{2}\left[\left(\frac{|t_2| + |t_1|}{2}\right)|t_2|^{m+2} - |t_2|^{m+1}\right]
\end{align*}
Then by Inequality \eqref{e6.5},
\begin{align*}
&(|t_2| + |t_1|)\left(\frac{|t_1| + |t_2|}{2}\right)^{m+2} \\
< &\left(\frac{|t_2| + |t_1|}{2}\right)|t_2|^{m+2} - |t_2|^{m+1} - \left[|t_1|^{m+2}\left(\frac{|t_2| + |t_1|}{2}\right) + |t_1|^{m+1}\right].
\end{align*}
Applying triangle inequality, we may infer that 
\begin{align*}
&\left(\frac{|t_2| + |t_1|}{2}\right)|t_2|^{m+2} - |t_2|^{m+1} - \left[|t_1|^{m+2}\left(\frac{|t_2| + |t_1|}{2}\right) + |t_1|^{m+1}\right]\\
\leq &\left|-\frac{|t_2| + |t_1|}{2}t_2^{m+2}\right| - \left|t_2^{m+1} - t_1^{m+1} + \frac{|t_2| + |t_1|}{2}t_1^{m+2}\right|\\
\leq &\left|-\frac{|t_2| + |t_1|}{2}t_2^{m+2}+ \frac{|t_2| + |t_1|}{2}t_1^{m+2} + t_2^{m+1} - t_1^{m+1} \right|\\
= &\left|\frac{|t_2| + |t_1|}{2}(t_1^{m+2} - t_2^{m+2}) + t_2^{m+1} - t_1^{m+1} \right|.
\end{align*}
Consequently, 
\[
(|t_2| + |t_1|)\left(\frac{|t_1| + |t_2|}{2}\right)^{m+2} < \left|\frac{|t_2| + |t_1|}{2}(t_1^{m+2} - t_2^{m+2}) + t_2^{m+1} - t_1^{m+1} \right|.
\]
Since $t_1t_2 = -1$, it holds that $|t_2 - t_1| = |t_2| + |t_1|$.
Then 
\[
(|t_2- t_1|)\left(\frac{|t_1| + |t_2|}{2}\right)^{m+2} < \left|\frac{|t_2| + |t_1|}{2}(t_1^{m+2} - t_2^{m+2}) + t_2^{m+1} - t_1^{m+1} \right|.
\]
Thus, by Rouch\'{e}'s Theorem (Theorem \ref{rouche}), we conclude that 
\begin{align*}
p(z) = &t_2^{m+1} - t_1^{m+1} + (t_1^{m+2} - t_2^{m+2})z + (t_2-t_1)(-1)^{m+1}z^{m+2}\\
q(z) = &-t_2^{m+1} + t_1^{m+1} - (t_1^{m+2} - t_2^{m+2})z
\end{align*}
have exactly one zero inside the circle $|z| = \frac{|t_1| + |t_2|}{2}$.
Since $|t_1t_2| = 1$, it holds that 
\[
\left|\frac{1}{t_2}\right| = |t_1| = \frac{|t_1| + |t_1|}{2} < \frac{|t_1| + |t_2|}{2}.
\]
Notice that $\frac{1}{t_2}$ is a zero of $p(z)$.
Hence $\frac{1}{t_2}$ is the only zero of $p(z)$ inside the circle $|z| = \frac{|t_1| + |t_2|}{2}$.
If we substitute $\frac{1}{t_2}$ for $z$, we may infer that $\frac{1}{t_2}$ is not a zero of $H_m(z)$.
Then $H_m(z)$ does not have a zero inside the circle $|z| = \frac{|t_1| + |t_2|}{2}$, since, otherwise, $p(z)$ would have two zeros inside the circle.
Now let $m>N$ be an arbitrary, and $z'$ be arbitrary root of $P_m(z)$. 
Applying Corollary \ref{cor5.1}, we obtain 
\begin{align*}
|z'|\geq &\frac{\sqrt{|ac|}}{|c|}\frac{|t_1| + |t_2|}{2} \\
= &\frac{\sqrt{|ac|}}{|c|}\frac{1}{2}\left(\frac{1}{|t_1|} + \frac{1}{|t_2|}\right)\\
= &\frac{\sqrt{|ac|}}{|c|}\frac{1}{2}\left(\frac{2\sqrt{|ac|}}{\sqrt{b^2-4ac} - |b|}  + \frac{2\sqrt{|ac|}}{\sqrt{b^2-4ac} + |b|} \right)\\
= &\frac{1}{2}\left(\frac{2|a|}{\sqrt{b^2-4ac} - |b|}  + \frac{2|a|}{\sqrt{b^2-4ac} + |b|} \right)\\
= &\frac{1}{2}\left(\frac{1}{|\alpha|}  + \frac{1}{|\beta|} \right).
\end{align*}
It concludes that 
\[
\mathcal{Z}(P_m)\cap B\left(\frac{1}{2}\left(\frac{1}{|\alpha|} + \frac{1}{|\beta|}\right),0\right) = \emptyset,
\]
for all $m\geq N$.
\end{proof}

We are ready to prove the main theorem of this section and paper.
\begin{thm}\label{c6.0}
Let $a,b,c\in\mathbb{R}\backslash\{0\}$ where $b^2 - 4ac > 0$.
Suppose  $\{P_m(z)\}_{m\geq 0}$ is a sequence of functions of $z$ generated as follows:
\[
\sum_{m=0}^\infty P_m(z) t^m =\frac{1}{(at^2+bt+c)(1-tz)}.
\]
Then 
\[\lim \mathcal{Z}(P_m) = \left\{z\in\mathbb{C}: |z| = \frac{1}{|t_1|}\right\}
\]
where $t_1$ is the smallest (in modulus) zero of $at^2+bt+c$.
\end{thm}
\begin{proof}
If we denote $t_2$ for the other zero of $at^2+bt+c$, Proposition \ref{e5.3} and Lemma \ref{l5.2} imply that
\begin{equation}\label{e6.0}
\begin{split}
    &{a(t_2-t_1)(t_1z-1)(t_2z-1)t_1^{m+1}t_2^{m+1}}P_m(z)\\
    = &t_2^{m+1} - t_1^{m+1} + (t_1^{m+2} - t_2^{m+2})z + (t_2-t_1)t_1^{m+1}t_2^{m+1}z^{m+2}.
\end{split}
\end{equation}

Let us define, for each $m = 0,1,2,3,\dots$,
\[
Q_m(z):= t_2^{m+1} - t_1^{m+1} + (t_1^{m+2} - t_2^{m+2})z + (t_2-t_1)t_1^{m+1}t_2^{m+1}z^{m+2}.
\]

We are to prove that $|t_1|\neq |t_2|$.
Suppose, by contradiction, that $|t_1| = |t_2|$.
Since $t_1,t_2$ are the zeros of $at^2 + bt + c$, by Vieta's Formula, $t_1t_2 = \frac{c}{a}$.
Then $|t_1| = |t_2|$ implies that 
\[
\left|\frac{c}{a}\right| = |t_1t_2| = |t_1||t_2| = |t_1|^2.
\]
This induces that $|t_1| = |t_2| = \sqrt{\left|\frac{c}{a}\right|}$.
For the case $ac < 0$, we may infer that 
\[
\left\{\begin{array}{cc} t_1 = & \sqrt{\left|\frac{c}{a}\right|}\\ t_2 = &- \sqrt{\left|\frac{c}{a}\right|}\end{array}\right. \text{ or } \left\{\begin{array}{cc} t_1 = &- \sqrt{\left|\frac{c}{a}\right|}\\ t_2 = & \sqrt{\left|\frac{c}{a}\right|}\end{array}\right..
\]
Notice that, by Vieta's Formula, 
\[
-\frac{b}{a} = t_1 + t_2 = 0.
\]
This implies that $b = 0$, which contradicts against our condition $b \neq 0$.
Thus our contrary supposition is false.
So it holds that $|t_1|\neq |t_2|$.
For the case $ac > 0$, we may infer that 
\[
t_1 = t_2 = -\sqrt{\frac{c}{a}}, \text{ or } t_1 = t_2 = \sqrt{\frac{c}{a}}.
\]
By Vieta's Formula, it holds that
\[
-\frac{b}{a} = t_1 + t_2 = \pm 2\sqrt{\frac{c}{a}}.
\]
By squaring both sides, we may infer that
\[
\frac{b^2}{a^2} = 4\frac{c}{a} \Leftrightarrow b^2 - 4ac = 0.
\]
It contradicts against our condition $b^2-4ac > 0$.
Thus our contrary supposition is false.
So it holds that $|t_1|\neq |t_2|$.
Here we have, in both cases $ac < 0$ and $ac > 0$, obtain $|t_1|\neq |t_2|$.
Given that $|t_1|\leq|t_2|$, it follows that 
 we may conclude that $|t_1| < |t_2|$.

Let us define the entire functions
\begin{align*}
\alpha_1(z):= &t_2 - t_2^2z, &\alpha_2(z):= &-t_1+t_1^2z, &\alpha_3(z):= &t_1t_2^2z^2-t_1^2t_2z^2,\\
\beta_1(z):= &t_2, &\beta_2(z):= &t_1, &\beta_3(z):= &t_1t_2z.
\end{align*}
Observe that 
\begin{align*}
    &\sum_{k=1}^3\alpha_k(z)\beta_k(z)^m \\
    = &(t_2 - t_2^2z)t_2^m + (-t_1+t_1^2z)t_1^m + (t_1t_2^2z^2-t_1^2t_2z^2)t_1^mt_2^mz^m\\
    = &t_2^{m+1} - t_2^{m+2}z - t_1^{m+1} + t_1^{m+2}z + t_1^{m+1}t_2^{m+2}z^{m+2} - t_1^{m+2}t_2^{m+1}z^{m+2}\\
    = & t_2^{m+1} - t_1^{m+1} + (t_1^{m+2} - t_2^{m+2})z + (t_2-t_1)t_1^{m+1}t_2^{m+1}z^{m+2}\\
    = &Q_m(z).
\end{align*}

Now our goal is to apply Theorem \ref{t6.0} to prove that $\liminf \mathcal{Z}(Q_m) = \limsup \mathcal{Z}(Q_m)$ by first checking the first condition of this theorem.
Let us consider $\beta_1$ and $\beta_2$.
Let $w\in\mathbb{C}$ arbitrary such that $|w| = 1$.
Observe that $|\beta_1(z)| = |t_2|$ and $|w\beta_2(z)| = |t_1|$ for any $z\in\mathbb{C}$.
Then $|\beta_1(z)|\neq |w\beta_2(z)|$, for all $z\in\mathbb{C}$, and it implies that $\beta_1(z)\neq w\beta_2(z)$, for all $z\in\mathbb{C}$.

Now let us consider $\beta_1$ and $\beta_3$.
Let $w\in\mathbb{C}$ be arbitrary such that $|w| = 1$.
With $a,c\neq 0$, Vieta's Formula induces that $t_1,t_2\neq 0$.
So $\frac{1}{t_2}$ exists.
Observe that $\left|\beta_1\left(\frac{1}{t_2}\right)\right| = |t_2|$, and that $\left|w\beta_3\left(\frac{1}{t_2}\right)\right| = |t_1|$.
Hence $\beta_1\left(\frac{1}{t_2}\right)\neq w\beta_3\left(\frac{1}{t_2}\right)$. Similarly, we may deduce that $\beta_2\left(\frac{1}{t_1}\right)\neq w \beta_3\left(\frac{1}{t_1}\right)$ for any $w$ with $|w| = 1$.
Since the conditions of Theorem \ref{t6.0} are met, we apply this theorem and conclude $\liminf \mathcal{Z}(Q_m) = \limsup \mathcal{Z}(Q_m)$. 

The next goal is to prove that 
\[
\lim \mathcal{Z}(Q_m) = \left\{z\in\mathbb{C}: |z| = \frac{1}{|t_1|}\right\}\cup\left\{\frac{1}{t_2}\right\},
\]
by showing each set is a subset and a superset of the other.

Let $z^*\in \lim \mathcal{Z}(Q_m)$ be arbitrary.
Suppose, by contradiction, that $z^*\notin \left\{z\in\mathbb{C}; |z| = \frac{1}{|t_1|}\right\}\cup\left\{\frac{1}{t_2}\right\}$.
Then $|z^*|\neq \frac{1}{|t_1|}$, and there are 2 cases: $|z^*| < \frac{1}{|t_1|}$ or $|z^*| >\frac{1}{|t_1|}$.

If $|z^*| < \frac{1}{|t_1|}$, observe that $|\beta_3(z^*)| = |z^*|\cdot|t_1|\cdot|t_2| < |t_2| = |\beta_1(z^*)|$, and that $|\beta_2(z^*)| = |t_1| < |t_2| = |\beta_1(z^*)|$.
Thus $|\beta_1(z^*)| > |\beta_k(z^*)|$ for $k = 2,3$.
Notice that, since $z^*\neq\frac{1}{t_2}$,
\[
|\alpha_1(z^*)| = |t_2 - t_2^2z^*| \neq 0.
\]
This, by Theorem \ref{t6.0}, implies that $z^*\notin \lim \mathcal{Z}(Q_m)$, which contradicts our assumption.

If  $|z^*| > \frac{1}{|t_1|}$, observe that $|\beta_3(z^*)| = |z^*|\cdot|t_1|\cdot|t_2| > |t_2| = |\beta_1(z^*)|$.
Notice that $|\beta_2(z^*)| = |t_1| < |t_2| = |\beta_1(z^*)| < |\beta_3(z^*)|$.
Thus $|\beta_3(z^*)| > |\beta_k(z^*)|$ for $k = 1,2$.
Notice that 
\[
|\alpha_3(z^*)| = |t_1t_2^2z^{*2} - t_1^2t_2z^{*2}| = |t_1t_2|\cdot|t_2-t_1|\cdot|z^*|^2 > 0.
\]
This, by Theorem \ref{t6.0}, implies that $z^*\notin \lim \mathcal{Z}(Q_m)$, which contradicts our assumption.

Hence in both cases $|z^*| < \frac{1}{|t_1|}$ or $|z^*| >\frac{1}{|t_1|}$, we encounter contradictions, which induces that $z^*\in \left\{z\in\mathbb{C}; |z| = \frac{1}{|t_1|}\right\}\cup\left\{\frac{1}{t_2}\right\}$.
Thus $\lim \mathcal{Z}(Q_m) \subseteq \left\{z\in\mathbb{C}; |z| = \frac{1}{|t_1|}\right\}\cup\left\{\frac{1}{t_2}\right\}$.

Now let $z^*\in \left\{z\in\mathbb{C}; |z| = \frac{1}{|t_1|}\right\}\cup\left\{\frac{1}{t_2}\right\}$ be arbitrary.
First let us consider the case $z^*\in \left\{z\in\mathbb{C}; |z| = \frac{1}{|t_1|}\right\}$.
Observe that 
\[
|\beta_1(z^*)| = |t_2|\geq |t_1| = |\beta_2(z^*)|,
\]
and that 
\[
|\beta_3(z^*)| = |t_1|\cdot |t_2|\cdot|z^*| = |t_2| = |\beta_1(z^*)| \geq |\beta_2(z^*)|.
\]
Notice that 
\[
|\alpha_3(z^*)| = |t_1t_2^2z^{*2} - t_1^2t_2z^{*2}| = |t_1t_2|\cdot|t_2-t_1|\cdot|z^*|^2 > 0
\]
and
\[
|\alpha_1(z^*)| = |t_2-t_2^2z^*| \geq |t_2|^2|z^*| - |t_2| = |t_2|\frac{|t_2|}{|t_1|} - |t_2| > |t_2| - |t_2| = 0.
\]
Then Theorem \ref{t6.0} implies that $z^*\in \lim \mathcal{Z}(Q_m)$.
Now let us consider the case $z^* = \frac{1}{t_2}$.
Observe that 
\[
|\beta_3(z^*)| = |z^*|\cdot|t_1|\cdot|t_2|  = |t_1| < |t_2| = |\beta_1(z^*)|,
\]
and that 
\[
|\beta_2(z^*)| = |t_1| < |t_2| = |\beta_1(z^*)|.
\]
Thus $|\beta_1(z^*)| > |\beta_k(z^*)|$ for $k = 2,3$.
Notice that, since $z^* = \frac{1}{t_2}$,
\[
|\alpha_1(z^*)| = |t_2 - t_2^2z^*| = 0.
\]
This, by Theorem \ref{t6.0}, implies that $z^*\in \lim \mathcal{Z}(Q_m)$.

It concludes that 
\[
\lim \mathcal{Z}(Q_m) \supseteq \left\{z\in\mathbb{C}; |z| = \frac{1}{|t_1|}\right\}\cup\left\{\frac{1}{t_2}\right\}.
\]
Here we have shown that $\lim \mathcal{Z}(Q_m)$ and $\left\{z\in\mathbb{C}; |z| = \frac{1}{|t_1|}\right\}\cup\left\{\frac{1}{t_2}\right\}$ are subset and superset of each other.
Hence
\[
\lim \mathcal{Z}(Q_m) = \left\{z\in\mathbb{C}; |z| = \frac{1}{|t_1|}\right\}\cup\left\{\frac{1}{t_2}\right\}.
\]

The next step is to show 
\begin{enumerate}
\item $\limsup \mathcal{Z}(P_m) \subseteq \limsup \left.\mathcal{Z}(Q_m)\middle\backslash\left\{\frac{1}{t_2}\right\}\right.$, and\\
\item $\left.\liminf \mathcal{Z}(Q_m)\middle\backslash\left\{\frac{1}{t_2}\right\}\right. \subseteq \liminf \mathcal{Z}(P_m)$.
\end{enumerate}

Let $z^*\in \limsup \mathcal{Z}(P_m)$ be arbitrary.
Choose $r>0$ and $N\in\mathbb{N}$ arbitrarily.
Given $z^*\in \limsup \mathcal{Z}(P_m)$, there exists $n>N$ such that 
\[
B(r,z^*)\cap \mathcal{Z}(P_n) \neq \emptyset.
\]
In view of Equation \eqref{e6.0}, we may deduce that $\mathcal{Z}(P_m)\subseteq \mathcal{Z}(Q_m)$, for any $m\in\mathbb{N}$.
Then 
\[
B(r,z^*)\cap \mathcal{Z}(P_n) \subseteq B(r,z^*)\cap \mathcal{Z}(Q_n),
\]
which implies that 
\[
B(r,z^*)\cap \mathcal{Z}(Q_n) \neq \emptyset.
\]
This concludes that 
\[
\limsup \mathcal{Z}(P_m) \subseteq \limsup \mathcal{Z}(Q_m).
\]

We need to show that $\frac{1}{t_2}\notin\limsup \mathcal{Z}(P_m)$.
There are two cases: $ac > 0$ and $ac < 0$.
Let us consider the case $ac>0$.
In view of $|t_1| < |t_2|$, we may infer that $\frac{1}{t_2}\in B\left(\frac{1}{|t_1|},0\right)$.
Then there exists $r>0$ such that $B\left(r,\frac{1}{t_2}\right)\subseteq B\left(\frac{1}{|t_1|},0\right)$.
By Theorem \ref{c5.2}, we may infer that $B\left(\frac{1}{t_2},r\right)\cap \mathcal{Z}(P_m) =\emptyset$ for any $m\geq 0$.
This implies that 
\[
\frac{1}{t_2}\notin \limsup\mathcal{Z}(P_m).
\]
Now let us consider the case $ac<0$.
By Lemma \ref{l6.0}, there exists $N\in\mathbb{N}$ such that 
\[
\mathcal{Z}(P_m)\cap B\left(\frac{1}{2}\left(\frac{1}{|t_1|}  + \frac{1}{|t_2|} \right),0\right) = \emptyset,
\]
for any $m>N$.
Since $|t_1| < |t_2|$, we may infer that
\[
\left|\frac{1}{t_2}\right| = \frac{1}{2}\left(\frac{1}{|t_2|}  + \frac{1}{|t_2|} \right) < \frac{1}{2}\left(\frac{1}{|t_1|}  + \frac{1}{|t_2|} \right).
\]
Then there exists $r>0$ such that $B\left(r,\frac{1}{t_2}\right)\subseteq B\left(\frac{1}{2}\left(\frac{1}{|t_1|}  + \frac{1}{|t_2|} \right), 0\right)$.
It follows that 
\[
\mathcal{Z}(P_m)\cap B\left(r,\frac{1}{t_2}\right) = \emptyset,
\]
for all $m\geq N$.
Thus $\frac{1}{t_2}\notin \limsup\mathcal{Z}(P_m)$.
Here we have, in both cases $ac> 0$ and $ac < 0$, that $\frac{1}{t_2}\notin \limsup\mathcal{Z}(P_m)$.
Consequently, we obtain
\[
\limsup \mathcal{Z}(P_m) \subseteq \left.\limsup \mathcal{Z}(Q_m)\middle\backslash\left\{\frac{1}{t_2}\right\}\right..
\]

Let $z^*\in \left.\liminf \mathcal{Z}(Q_m)\middle\backslash\left\{\frac{1}{t_2},\frac{1}{t_1}\right\}\right.$ be arbitrary. 
Choose
\begin{equation}\label{e6.7}
0 < r < \min\left(\left|\frac{1}{t_1} - z^*\right|, \left|\frac{1}{t_2} - z^*\right|\right)
\end{equation}
arbitrarily.
Since $z^*\in \liminf \mathcal{Z}(Q_m)$, there exists $N\in\mathbb{N}$ such that 
\[
B(r,z^*)\cap \mathcal{Z}(Q_m)\neq\emptyset,
\]
for any $m>N$.
Let $n>N$ be arbitrary.
Then there exists $z_n\in B(r,z^*)\cap \mathcal{Z}(Q_n)$.
Since $z_n\in B(r,z^*)$, Inequality \eqref{e6.7} induces that
\[
|z_n - z^*| < r< \left|\frac{1}{t_1} - z^*\right|,
\]
This implies that $z_n\neq 1/t_1$.
In a similar way, we can deduce that $z_n\neq 1/t_2$.
Given that $z_n \in \mathcal{Z}(Q_n)$, $z_n$ is a root of $Q_n$.
Since $z_n\neq 1/t_1, 1/t_2$, Equation \eqref{e6.0} implies that  $z_n$ is a root of $P_n(z)$.
Hence $B(r,z^*)\cap \mathcal{Z}(P_n)\neq\emptyset.$
This implies that $z^*\in \liminf \mathcal{Z}(P_m)$.
Thus 
\[
\left.\left\{z\in\mathbb{C}; |z| = \frac{1}{|t_1|}\right\}\middle\backslash\left\{\frac{1}{t_2}, \frac{1}{t_1}\right\}\right.\subseteq \liminf \mathcal{Z}(P_m).
\]
By Proposition \ref{p6.0}, we may infer that $1/t_1\in \liminf \mathcal{Z}(P_m)$.
It follows that 
\[
\left.\liminf \mathcal{Z}(Q_m)\middle\backslash\left\{\frac{1}{t_2}\right\}\right. \subseteq \liminf \mathcal{Z}(P_m).
\]

Here we have 
\begin{enumerate}
\item $\limsup \mathcal{Z}(P_m) \subseteq \left.\limsup \mathcal{Z}(Q_m)\middle\backslash\left\{\frac{1}{t_2}\right\}\right.$,\\
\item $\left.\liminf \mathcal{Z}(Q_m)\middle\backslash\left\{\frac{1}{t_2}\right\}\right. \subseteq \liminf \mathcal{Z}(P_m)$,\\
\item  $\liminf \mathcal{Z}(P_m) \subseteq \limsup \mathcal{Z}(P_m)$ by Proposition \ref{p6.1}, and \\
\item $\liminf \mathcal{Z}(Q_m) = \limsup \mathcal{Z}(Q_m) = \lim\mathcal{Z}(Q_m)$.
\end{enumerate}
Consequently, we obtain
\[
\left.\lim\mathcal{Z}(Q_m)\middle\backslash\left\{\frac{1}{t_2}\right\}\right. \subseteq \liminf \mathcal{Z}(P_m) \subseteq \limsup \mathcal{Z}(P_m) \subseteq \left.\lim\mathcal{Z}(Q_m)\middle\backslash\left\{\frac{1}{t_2}\right\}\right..
\]
This implies that 
\[
\lim\mathcal{Z}(P_m) = \left\{z\in\mathbb{C}; |z| = \frac{1}{|t_1|}\right\}.\qedhere
\]
\end{proof}


\clearpage

\begin{center}
\section{CONCLUSIONS}
\end{center}
\thispagestyle{empty} 

For a long time, many mathematicians studied the pattern of zeros of functions.
This area of research is still active nowadays.
In this paper, we reviewed what has been studied in the area and made our own contribution to this subject.
In the development of our theorems, we use the classical Rouch\'{e}'s Theorem, Enestr\"{o}m-Kakeya's Theorem, and then study the limiting behaviors of the zeros of our sequence of polynomials.

Our main theorems were about the generating function
\[
\frac{1}{(at^2 + bt + c)(1-zt)}
\]
and a circle centered at the origin with radius $\frac{1}{|\alpha|}$ where $\alpha$ is a zero of $at^2 + bt + c$ with smallest modulus.
When we were working with Rouch\'{e}'s Theorem, we found that, when $ac < 0$, our zeros lied inside a closed disk, and when $ac>0$, the zeros lied outside an open ball.
Using Enestr\"{o}m-Kakeya Theorem, we prove that, when $ac>0$, the zeros are outside of the closed disk.
However, there is a problem in applying Enestr\"{o}m-Kakeya Theorem for the case $ac<0$.
For any $m\geq0$, the ratio of the coefficient of $z^{m-2}$ to the coefficient of $z^{m-1}$ is always greater than $1/|\alpha|$.

In this paper we cover the case $at^2 +bt + c$ having real zeros with different modulus.
In begining of Section `Zero Distribution of $\{P_m(z)\}_{m\geq 0}$ about a Circle', we saw that zeros are both inside and outside of the circle when $at^2 +bt + c$ has complex zeros.
With Rouch\'{e}'s Theorem, we might be able to obtain the ratio of the number of zeros inside to the number of zeros outside.
Or we might be able to count exact number of zeros inside and outside of the circle.

In Section `Application of Limit Distribution of Zeros', we only cover the case $at^2+bt+c$ has real zeros.
An intuition is saying we maybe able to handle the case of imaginary zeros.
\begin{conjecture}
Let  $\{P_m(z)\}_{m\geq 0}$ be a sequence of functions of $z$ generated as follows:
\[
\sum_{m=0}^\infty P_m(z) t^m =\frac{1}{(at^2+bt+c)(1-tz)}
\]
where $a,b,c\in \mathbb{R}\backslash\{0\}$ and $b^2-4ac\leq 0$.
Then $\lim \mathcal{Z}(P_m) = \left\{z\in\mathbb{C}; |z| = \frac{1}{|t_1|}\right\}$ where $t_1$ is the smallest (in modulus) zero of $at^2+bt+c$.
\end{conjecture}
Computer experiments suggest that this conjecture is true.
However, we have not been able to apply Theorem \ref{t6.0} directly for the reason that we will now mention.
If we follow the procedure as in the proof of Theorem \ref{c6.0}, we need to apply Theorem \ref{t6.0} to the polynomial
\[
Q_m(z) = t_2^{m+1} - t_1^{m+1} + (t_1^{m+2} - t_2^{m+2})z + (t_2-t_1)t_1^{m+1}t_2^{m+1}z^{m+2}
\]
by setting 
\begin{align*}
\alpha_1(z):= &t_2 - t_2^2z &\alpha_2(z):= &-t_1+t_1^2z &\alpha_3(z):= &t_1t_2^2z^2-t_1^2t_2z^2\\
\beta_1(z):= &t_2 &\beta_2(z):= &t_1 &\beta_3(z):= &t_1t_2z
\end{align*}
However, in this conjecture, we may deduce that $|\beta_1(z)| = |t_1| = |t_2| = |\beta_2(z)|$, and this violates the conditions to apply Theorem \ref{t6.0}.

\clearpage


\end{document}